\documentclass[twoside,11pt]{article}
\usepackage{isipta}

\usepackage[utf8]{inputenc}
\inputencoding{utf8} 

\usepackage{hyperref,color,soul,booktabs}
\setulcolor{blue}
\newcommand{\BibTeX}{\textsc{B\kern-0.1emi\kern-0.017emb}\kern-0.15em\TeX}

\usepackage{url}
\usepackage{mathptmx}
\usepackage{amsmath}
\usepackage{courier}
\usepackage{enumerate}
\usepackage{enumitem,multicol}
\usepackage{tikz}
\usepackage{nicefrac}
\usepackage{bm}

\usepackage{etoolbox}
\newtoggle{arxiv}
\toggletrue{arxiv}

\DeclareMathAlphabet\mathbfcal{OMS}{cmsy}{b}{n}

\newcommand{\nats}{\mathbb{N}}
\newcommand{\natswith}{\nats_{0}}
\newcommand{\reals}{\mathbb{R}}
\newcommand{\rationals}{\mathbb{Q}}
\newcommand{\rationalsnonneg}{\mathbb{Q}_{\geq0}}

\newcommand{\realspos}{\reals_{>0}}
\newcommand{\realsnonneg}{\reals_{\geq 0}}

\newcommand{\states}{\mathcal{X}}

\newcommand{\power}{\mathcal{P}(\hspace{-1.5pt}\states\hspace{-0.5pt})}

\newcommand{\poweron}[1]{\mathcal{P}(\hspace{-1.5pt}\states_{#1}\hspace{-0.5pt})}
\newcommand{\nonemptypower}{\mathcal{P}_{\emptyset}(\hspace{-1.5pt}\states\hspace{-0.5pt})}

\newcommand{\nonemptypoweron}[1]{\mathcal{P}_{\emptyset}(\hspace{-1.5pt}\states_{#1}\hspace{-0.5pt})}

\newcommand{\gambles}{\mathcal{G}(\hspace{-1.5pt}\states\hspace{-0.5pt})}

\newcommand{\gamblespos}{\mathcal{G}_{>0}(\hspace{-1.5pt}\states\hspace{-0.5pt})}
\newcommand{\gamblesnonneg}{\mathcal{G}_{\geq0}(\hspace{-1.5pt}\states\hspace{-0.5pt})}

\newcommand{\gambleson}[1]{\mathcal{G}(\hspace{-1.5pt}\states_{#1}\hspace{0.pt})}

\newcommand{\desir}{\mathcal{D}}

\newcommand{\ind}[1]{\mathbb{I}_{#1}}

\newcommand{\asa}{\Leftrightarrow}
\newcommand{\then}{\Rightarrow}

\newcommand{\abs}[1]{\left\vert #1 \right\vert}

\newcommand{\coloneqq}{:\!=}

\newcommand{\lp}{\underline{P}}
\newcommand{\pr}{P}
\newcommand{\natexLP}{\underline{E}}
\newcommand{\natexUP}{\overline{E}}

\newcommand{\posi}{\mathrm{posi}}
\newcommand{\marg}{\mathrm{marg}}

\newcommand{\A}{\mathcal{A}}
\newcommand{\B}{\mathcal{B}}
\newcommand{\C}{\mathcal{C}}
\newcommand{\E}{\mathcal{E}}
\newcommand{\Y}{\mathcal{Y}}
\newcommand{\G}{\mathcal{G}}

\ShortHeadings{Independent Natural Extension for Infinite Spaces}{De Bock}

\begin{document}

\title{Independent Natural Extension for Infinite Spaces}
\author{\name Jasper De Bock \email jasper.debock@ugent.be\\
\addr Ghent University - ELIS, SYSTeMS\\
Technologiepark -- Zwijnaarde 914, 9052 Zwijnaarde, Belgium}
\maketitle
\begin{abstract}
We define and study the independent natural extension of two local uncertainty models for the general case of infinite spaces, using the frameworks of sets of desirable gambles and conditional lower previsions. In contrast to~\cite{Miranda2015460}, we adopt Williams-coherence instead of Walley-coherence. We show that our notion of independent natural extension always exists---whereas theirs does not---and that it satisfies various convenient properties, including factorisation and external additivity. The strength of these properties depends on the specific type of epistemic independence that is adopted. In particular, epistemic event-independence is shown to outperform epistemic atom-independence. Finally, the cases of lower expectations, expectations, lower probabilities and probabilities are obtained as special instances of our general definition. By applying our results to these instances, we demonstrate that epistemic independence is indeed epistemic, and that it includes the conventional notion of independence as a special case.
\end{abstract}
\begin{keywords}
independent natural extension; epistemic independence; Williams-coherence; infinite spaces; sets of desirable gambles; conditional lower previsions.
\end{keywords}

\section{Introduction}\label{sec:introduction}

When probabilities are imprecise, in the sense that they are only partially specified, it is no longer clear what it means for two variables to be independent~\citep{Couso:1999wh}. One approach is to apply the standard notion of independence to every element of some set of probability measures. The alternative, called epistemic independence, is to define independence as mutual irrelevance, in the sense that receiving information about one of the variables will not affect our uncertainty model for the other. 
The advantage of this intuitive alternative is that it has a much wider scope: since epistemic independence is expressed in terms of uncertainty models instead of probabilities, it can easily be applied to a variety of such models, including non-probabilistic ones. We here focus on sets of desirable gambles and conditional lower previsions, the latter of which includes lower expectations, expectations, lower probabilities and probabilities as special cases.

When an assessment of epistemic independence is combined with local uncertainty models, it leads to a unique corresponding joint uncertainty model that is called the independent natural extension. If the variables involved can take only a finite number of values, this independent natural extension always exists, and it then satisfies various convenient properties that allow for the design of efficient algorithms~\citep{deCooman:2011ey,deCooman:2012vba}.
If the variables involved take values in an infinite set, the situation becomes more complicated. On the one hand, for the specific case of lower probabilities,~\cite{Vicig:2000vh} managed to obtain results that resemble those of the finite case. On the other hand, for the more general case of lower previsions,~\cite{Miranda2015460} recently found that the independent natural extension may not even exist. 

Our main contribution consists in generalising the results of~\cite{Vicig:2000vh} to the case of conditional lower previsions, using sets of desirable gambles as an intermediate step. The key technical difference with~\cite{Miranda2015460} is that we use Williams-coherence~\citep{williams1975,Williams:2007eu} instead of Walley-coherence~\citep{Walley:1991vk}. This difference turns out to be crucial because our notion of independent natural extension always exists. Furthermore, as we will see, it satisfies the same convenient properties that are known to hold in the finite case, including factorisation and external additivity. 

An essential feature of our approach is that we adopt a very general notion of epistemic independence, where the choice of conditioning events is initially left open. Important examples such as epistemic atom-independence and epistemic event-independence are then obtained as special cases. In this way, we are able to study the effect of different types of epistemic independence on the resulting notion of independent natural extension. 

The rest of the paper is structured as follows. We start in Section~\ref{sec:prelim} by introducing some basic technical concepts and notation that will be needed further on. Next, in Section~\ref{sec:modellinguncertainty}, we explain how uncertainty can be modelled with sets of desirable gambles and conditional lower previsions. For readers that are unfamiliar with these frameworks, this section serves as a stand-alone introduction. With all these preliminaries in place, the paper then moves on to its main contributions, which we report on in Sections~\ref{sec:independence}--\ref{sec:factadd}. Section~\ref{sec:independence} introduces our general notion of epistemic independence and explains how it subsumes epistemic atom- and event-independence as special cases. Next, Section~\ref{sec:indnatext} introduces our central object of interest, which is the independent natural extension, and Section~\ref{sec:choiceofevents} discusses the extent to which it depends on the specific type of epistemic independence that is adopted. Crucially, we find that in our approach, regardless of the chosen type of epistemic independence, the independent natural extension always exists. Our external additivity and factorisation results are reported on in Section~\ref{sec:factadd}; the strength of these results does depend on the chosen type of epistemic independence. In the last part of the paper, which consists of Sections~\ref{sec:probs} and~\ref{sec:specialcases}, we reinterpret our results in terms of lower expectations, expectations, lower probabilities and probabilities. Section~\ref{sec:probs} recalls how all of these models are special cases of conditional lower previsions, which then enables us to apply our results to them in Section~\ref{sec:specialcases}. We also use this connection to explain why epistemic independence is indeed epistemic, and how it includes the conventional notion of independence as a special case.
We end the paper in Section~\ref{sec:conclusions} with a brief summary of our main results and findings. The proofs of our results are gathered in an appendix.

Finally, I would like to add that this paper is an extended version of an earlier conference version; see~\cite{debock2017williamsPMLR}. The most substantial additions are the results in Sections~\ref{sec:probs} and~\ref{sec:specialcases} and the proofs in Appendix~\ref{sec:proofs}. We have also added several examples.

\section{Preliminaries and Notation}\label{sec:prelim}

We use $\nats$ to denote the natural numbers without zero and let $\natswith\coloneqq\nats\cup\{0\}$. $\reals$ is the set of real numbers and $\rationals$ is the set of rational numbers. Sign restrictions are imposed with subscripts. For example, we let $\reals_{>0}$ be the set of positive real numbers and let $\rationals_{\geq0}$ be the set of non-negative rational numbers. The extended real numbers are denoted by $\overline{\reals}\coloneqq\reals\cup\{-\infty,+\infty\}$.

For any non-empty set $\states$, the power set of $\states$---the set of all subsets of $\states$---is denoted by $\power$, and we let $\nonemptypower\coloneqq\power\setminus\{\emptyset\}$ be the set of all non-empty subsets of $\states$. Elements of $\power$ are called events. A set of events $\B\subseteq\power$ is called a field if it is non-empty and closed with respect to complements and finite intersections and unions. If it is also closed with respect to countable intersections and unions, it is called a sigma field. A partition of $\states$ is a set $\B\subseteq\nonemptypower$ of pairwise disjoint non-empty subsets of $\states$ whose union is equal to $\states$. We also adopt the notational trick of identifying $\states$ with the set of atoms $\{\{x\}\colon x\in\states\}$, which allows us to regard $\states$ as a partition of $\states$.

A bounded real-valued function on $\states$ will be called a gamble on $\states$. The set of all gambles on $\states$ is denoted by $\gambles$, the set of all non-negative gambles on $\states$ is denoted by $\mathcal{G}_{\geq0}(\states)$, and we let $\gamblespos\coloneqq\mathcal{G}_{\geq0}(\states)\setminus\{0\}$ be the set of all non-negative non-zero gambles. For any set of gambles $\A\subseteq\gambles$, we let
\vspace{-2pt}
\begin{equation}\label{eq:posi}
\posi(\A)
\coloneqq
\left\{
\sum_{i=1}^n \lambda_if_i
\colon
n\in\nats,
\lambda_i\in\realspos,f_i\in\mathcal{A}
\right\}
\end{equation}
and
\begin{equation}\label{eq:natextop}
\E(\A)
\coloneqq
\posi\left(
\A\cup\gamblespos
\right).\vspace{10pt}
\end{equation}
Indicators are a particular type of gamble. For any $A\in\power$, the corresponding indicator $\ind{A}$ of $A$ is a gamble in $\G_{\geq0}(\states)$, defined for all $x\in\states$ by $\ind{A}(x)\coloneqq1$ if $x\in A$ and $\ind{A}(x)\coloneqq0$ otherwise.

Finally, for any---possibly empty---$\B\subseteq\nonemptypower$, we will also require the notion of a non-negative $\B$-measurable gamble, which we define as a uniform limit of simple $\B$-measurable gambles.

\begin{definition}\label{def:measurable:simple}Let $\B\subseteq\nonemptypower$. We call $\smash{g\in\mathcal{G}_{\geq0}(\states)}$ a simple $\B$-measurable gamble if there are $c_0\in\reals_{\geq0}$, $n\in\natswith$ and, for all $i\in\{1,\dots,n\}$, $c_i\in\reals_{\geq0}$ and $B_i\in\B$, such that $g=c_0+\sum_{i=1}^nc_i\ind{B_i}$.
\end{definition}

\begin{definition}\label{def:measurable:uniform}Let $\B\subseteq\nonemptypower$. A gamble $g\in\mathcal{G}_{\geq0}(\states)$ is $\B$-measurable if it is a uniform limit of non-negative simple $\B$-measurable gambles, in the sense that there is a sequence $\{g_n\}_{n\in\nats}$ of simple $\B$-measurable gambles in $\mathcal{G}_{\geq0}(\states)$ such that $\lim_{n\to+\infty}\sup\abs{g-g_n}=0$.
\end{definition}

Readers that are familiar with the concepts of simple and measurable functions that are common in measure theory will observe some similarities. However, there are some important differences as well. On the one hand, our definitions are more restrictive: we only consider bounded non-negative functions, Definition~\ref{def:measurable:simple} requires that the coefficients $c_i$ are non-negative, and Definition~\ref{def:measurable:uniform} considers uniform limits instead of pointwise limits. On the other hand, our definitions are more general because we allow for $\B$ to be any subset of $\nonemptypower$. Nevertheless, if $\B\cup\{\emptyset\}$ is a sigma field, we have the following equivalence.

\begin{proposition}\label{prop:measurability:equivalenceforsigmafield}
Consider any $\B\subseteq\nonemptypower$ such that $\B^*\coloneqq\B\cup\{\emptyset\}$ is a sigma field. Then for any $g\in\gamblesnonneg$, $g$ is $\B^*$-measurable in the measure-theoretic sense~\cite[Definition~10.1]{Nielsen1997} if and only if it is $\B$-measurable in the sense of Definition~\ref{def:measurable:uniform}.
\end{proposition}

The proof of this result is based on the following sufficient condition for $\B$-measurability, which provides a convenient tool for establishing the $\B$-measurability of a given function. In particular, it implies that every non-negative gamble is $\nonemptypower$-measurable.

\begin{proposition}\label{prop:measurable:sufficient:general}
Let $\B\subseteq\nonemptypower$ and $g\in\mathcal{G}_{\geq0}(\states)$. If, for all $r\in\rationalsnonneg$, the set $\{x\in\states\colon g(x)\geq r\}$ is a finite union of pairwise disjoint events in $\B\cup\{\states,\emptyset\}$, then $g$ is $\B$-measurable.
\end{proposition}

\begin{corollary}\label{corol:measurable:sufficient:allsets}
Every $g\in\mathcal{G}_{\geq0}(\states)$ is $\nonemptypower$-measurable.
\end{corollary}

The following three examples  provide these abstract concepts and results with some intuition, by studying the (non-)$\B$-measurability of various functions, for different choices of $\B$. They also demonstrate that $\B$-measurability is easier to achieve if $\B$ contains more events.

\begin{example}\label{ex:oneovern}
Let $\states=\nats$ and let $g\in\mathcal{G}_{\geq0}(\states)$ be defined by $g(x)\coloneqq \nicefrac{1}{x}$ for all $x\in\nats$. Corollary~\ref{corol:measurable:sufficient:allsets} then trivially implies that $g$ is $\mathcal{P}_{\emptyset}(\nats)$-measurable. However, $g$ is also $\B$-measurable for some strict subsets $\B$ of $\mathcal{P}_{\emptyset}(\nats)$. For example, Proposition~\ref{prop:measurable:sufficient:general} implies that $g$ is $\nats$-measurable, because for every $r\in\rationals_{\geq0}$, the set
\vspace{-6pt}
\begin{equation*}
\{x\in\nats\colon g(x)\geq r\}
=\{x\in\nats\colon \frac{1}{x}\geq r\}
=
\begin{cases}
\emptyset&\text{ if $r>1$}\\
\{1,\dots,\lfloor\nicefrac{1}{r}\rfloor\}&\text{ if $0<r\leq1$}\\
\nats&\text{ if $r=0$}
\end{cases}
\vspace{8pt}
\end{equation*}
is clearly a finite union of pairwise disjoint events in $\B\cup\{\states,\emptyset\}=\nats\cup\{\nats,\emptyset\}$. The set $\B$ cannot be too small though. For example, in the extreme case where $\B=\emptyset$, $g$ is no longer $\B$-measurable. The easiest way to see this is to infer from Definition~\ref{def:measurable:simple} that simple $\emptyset$-measurable gambles are constant. Therefore, since uniform limits of constant gambles are constant, Definition~\ref{def:measurable:uniform} implies that all $\emptyset$-measurable gambles are constant. Hence, since $g$ is not constant, it is not $\emptyset$-measurable.\hfill$\lozenge$
\end{example}

\begin{example}\label{ex:oddandeven}
Let $\states=\nats$ and let $\ind{\mathrm{odd}}\in\G_{\geq0}(\nats)$ be the indicator of the odd numbers, defined for all $x\in\nats$ by $\ind{\mathrm{odd}}(x)=1$ if $x$ is odd and $\ind{\mathrm{odd}}(x)=0$ otherwise. Similarly, let $\ind{\mathrm{even}}\in\G_{\geq0}(\nats)$ be the indicator of the even numbers, defined by $\ind{\mathrm{even}}(x)\coloneqq1$ if $x$ is even and $\ind{\mathrm{even}}(x)\coloneqq0$ otherwise. Here too, it follows from Corollary~\ref{corol:measurable:sufficient:allsets} that $\ind{\mathrm{odd}}$ and $\ind{\mathrm{even}}$ are $\mathcal{P}_{\emptyset}(\nats)$-measurable. However, as we are about to prove, and in contrast with the previous example, $\ind{\mathrm{odd}}$ and $\ind{\mathrm{even}}$ are not $\nats$-measurable. We focus on $\ind{\mathrm{odd}}$; the argument for $\ind{\mathrm{even}}$ is completely analogous.

Assume \emph{ex absurdo} that $\ind{\mathrm{odd}}$ is $\nats$-measurable. It then follows from Definition~\ref{def:measurable:uniform} that there is some simple $\nats$-measurable function $g\in\G_{\geq0}$ such that $\vert \ind{\mathrm{odd}}(x)-g(x)\vert<\nicefrac{1}{2}$ for all $x\in\nats$. Since $\B=\nats$, Definition~\ref{def:measurable:simple} then implies that there is some $c_0\in\reals_{\geq0}$ and some finite set $A$ such that $g(x)=c_0$ for all $x\in\nats\setminus A$. Since $\nats$ contains infinitely many odd and even numbers, this in turn implies that there is some (and in fact infinitely many) odd $x_{\mathrm{odd}}$ and even $x_{\mathrm{even}}$ in $\nats\setminus A$ for which $g(x_{\mathrm{odd}})=g(x_{\mathrm{even}})=c_0$. Hence, we find that
\begin{align*}
1
=\abs{\ind{\mathrm{odd}}(x_{\mathrm{odd}})-\ind{\mathrm{odd}}(x_{\mathrm{even}})}
&\leq\abs{\ind{\mathrm{odd}}(x_{\mathrm{odd}})-c_0}+
\abs{c_0-\ind{\mathrm{odd}}(x_{\mathrm{even}})}\\
&=\abs{\ind{\mathrm{odd}}(x_{\mathrm{odd}})-g(x_{\mathrm{odd}})}+
\abs{g(x_{\mathrm{even}})-\ind{\mathrm{odd}}(x_{\mathrm{even}})}
<\frac{1}{2}+\frac{1}{2}=1,
\end{align*}
a contradiction.
\hfill$\lozenge$
\end{example}

\begin{example}\label{ex:onthereals}
Let $\states=\reals$ and let $\B\coloneqq\B^*\setminus\{\emptyset\}$, with $\B^*$ the $\sigma$-algebra of Lebesgue measurable subsets of $\reals$. Let $g\in\G_{\geq0}(\reals)$ be the indicator of the non-negative reals, defined for all $x\in\reals$ by $g(x)\coloneqq 1$ if $x\geq0$ and $g(x)\coloneqq0$ otherwise, and let $h\in\G(\reals)$ be defined by $h(x)\coloneqq x^3$ for all $x\in\reals$. Then $g$ is Lebesgue measurable because it is a step function and $h$ is Lebesgue measurable because it is continuous. Therefore, it follows from Proposition~\ref{prop:measurability:equivalenceforsigmafield} that $g$ and $h$ are $\B$-measurable.
\hfill$\lozenge$
\end{example}

\section{Modelling Uncertainty}\label{sec:modellinguncertainty}

A subject's uncertainty about a variable $X$ that takes values $x$ in some non-empty set $\states$ can be mathematically represented in various ways. The most popular such method is probability theory, but it is by no means the only one, nor is it the most general one. We here adopt the more general frameworks of sets of desirable gambles and conditional lower previsions.

The main aim of this section is to provide an overview of the basic technical aspects of these frameworks, because they will be essential in the rest of the paper. Notably, we do not impose any constraints on the cardinality of $\states$: it may be finite, countably infinite or uncountably infinite. Connections with other---perhaps better known---models for uncertainty, including probability theory, will be  briefly touched upon at the end of this section; detailed connections will be established in Section~\ref{sec:probs}.

The basic idea behind \emph{sets of desirable gambles} is to model a subject's uncertainty about $X$ by considering her attitude towards gambles---bets---on $\states$. In particular, we consider the gambles $f\in\gambles$ that she finds \emph{desirable}, in the sense that she is willing to engage in a transaction where, once the actual value $x\in\states$ of $X$ is known, she will receive a---possibly negative---reward $f(x)$ in some linear utility scale. Even more so, she prefers these desirable gambles over the status quo, that is, over not conducting any transaction at all. A set of desirable gambles is called coherent if it satisfies the following rationality requirements.

\begin{definition}
\label{def:SDG}
A coherent set of desirable gambles $\desir$ on $\states$ is a subset of\/ $\gambles$ such that, for any two gambles $f,g\in\gambles$ and any positive real number $\lambda\in\realspos$:
\vspace{2pt}

\begin{enumerate}[label=\emph{D\arabic*:},ref=D\arabic*]
\item
if $f\geq0$ and $f\neq0$, then $f\in\desir$;\label{def:SDG:partialgain}
\item
if $f\in\desir$ then $\lambda f\in\desir$;\label{def:SDG:homo}
\item
if $f,g\in\desir$, then $f+g\in\desir$;\label{def:SDG:convex}
\item
if $f\leq0$, then $f\notin\desir$.\label{def:SDG:partialloss}
\vspace{2pt}
\end{enumerate}
\end{definition}
\noindent
The first axiom states that the possibility of a positive reward without risking a negative reward should always be desirable, whereas the fourth axiom states that gambles that  offer no positive rewards should never be desirable. The other two axioms are immediate consequences of the linearity of the utility scale.
Despite their simplicity, sets of desirable gambles offer a surprisingly powerful framework for modelling uncertainty; see for example~\citep{Walley:2000ef} and~\citep{Quaeghebeur:2014tjb}. 
For our present purposes though, all we need for now is Definition~\ref{def:SDG}.

\emph{Conditional lower previsions} also model a subject's uncertainty about $X$ by considering her attitude towards gambles on $\states$. However, in this case, instead of considering sets of gambles, we consider the prices at which a subject is willing to buy these gambles. Let
\begin{equation*}
\mathcal{C}(\states)
\coloneqq
\gambles\times\nonemptypower
\end{equation*}
be the set of all pairs $(f,B)$, where $f$ is a gamble on $\states$ and $B$ is a non-empty subset of $\states$---an event. A conditional lower prevision is then defined as follows.
\begin{definition}
A conditional lower prevision $\lp$ on $\C\subseteq\C(\states)$ is a map
\begin{equation*}
\lp\colon
\C
\to
\overline{\reals}
\colon
(f,B)\to \lp(f\vert B).\vspace{7pt}
\end{equation*}
\end{definition}
For any $(f,B)$ in the domain $\C$, the lower prevision $\lp(f\vert B)$ of $f$ conditional on $B$ is interpreted as a subject's supremum price $\mu$ for buying $f$, under the condition that the transaction is called off when $B$ does not happen---if $x\notin B$. In other words, $\lp(f\vert B)$ is the supremum value of $\mu$ for which she is willing to engage in a transaction where she receives $f(x)-\mu$ if $x\in B$ and zero otherwise, and furthermore prefers this transaction to the status quo. If $B=\states$, we adopt the shorthand notation $\lp(f)\coloneqq \lp(f\vert\states)$ and then call $\lp(f)$ the lower prevision of $f$. If $B=\states$ for all $(f,B)\in\C$, meaning that there is some $\G\subseteq\gambles$ such that $\C=\{(f,\states)\colon f\in\G\}$, this convention allows us to regard $\lp$ as an operator on $\G$, and we then say that $\lp$ is a \emph{lower prevision}. In this sense, lower previsions are a special case of conditional lower previsions.

It is also possible to consider \emph{conditional upper previsions} $\overline{P}(f\vert B)$, which are interpreted as infimum selling prices. However, since selling $f$ for $\mu$ is equivalent to buying $-f$ for $-\mu$, we have that $\overline{P}(f\vert B)=-\lp(-f\vert B)$. For that reason, we will mainly focus on conditional lower previsions. A similar remark applies to (unconditional) \emph{upper previsions}.

Because of their interpretation in terms of buying prices for gambles, a particularly intuitive way to obtain a conditional lower prevision $\lp$ is to derive it from a set of gambles $\desir$. Specifically, for every $\desir\subseteq\gambles$, we let 
\begin{equation}\label{eq:LPfromD}
\lp_\desir(f\vert B)
\coloneqq
\sup\{\mu\in\reals\colon
[f-\mu]\ind{B}\in\desir
\}
\text{~~for all $(f,B)\in\C(\states)$.}
\vspace{3pt}
\end{equation}
A conditional lower prevision is then called coherent if can be derived from a coherent set of desirable gambles in this way.

\begin{definition}
\label{def:cohlp}
A conditional lower prevision $\lp$ on a domain $\C\subseteq\C(\states)$ is coherent if there is a coherent set of desirable gambles $\desir$ on $\states$ such that $\lp$ coincides with $\lp_\desir$ on $\C$.
\end{definition}

This definition of coherence is heavily inspired by the work of~\cite{williams1975,Williams:2007eu}. The only two minor differences are that our rationality axioms on $\desir$ are slightly different from his---for example, he allows for $\desir$ to include the zero gamble---and that we do not impose any structure on the domain $\C$. Nevertheless, when the domain $\C$ satisfies the structural constraints in~\citep{Williams:2007eu}, Definition~\ref{def:cohlp} is equivalent to that of Williams. More generally, as the following result establishes, it is equivalent to the structure-free notion of Williams-coherence that was developed by \cite{Pelessoni:2009co}. 


\begin{proposition}\label{prop:equivalentToPelessoniAndVicig}
A conditional lower prevision $\lp$ on $\C\subseteq\C(\states)$ is coherent if and only if it is real-valued and, for all $n\in\natswith$ and all choices of $\lambda_0,\dots,\lambda_n\in\reals_{\geq0}$ and $(f_0,B_0),\dots,(f_n,B_n)\in\mathcal{C}$:
\begin{equation}\label{eq:prop:equivalentToPelessoniAndVicig}
\sup_{x\in B}
\Big(\,
\sum_{i=1}^n
\lambda_i\ind{B_i}(x)
[f_i(x)-\lp(f_i\vert B_i)]
-\lambda_0\ind{B_0}(x)
[f_0(x)-\lp(f_0\vert B_0)]
\Big)
\geq0,
\vspace{6pt}
\end{equation}
where we let $B\coloneqq\cup_{i=0}^nB_i$.
\end{proposition}

The advantage of this alternative characterisation is that it is expressed directly in terms of lower previsions. Nevertheless, we consider Equation~\eqref{eq:prop:equivalentToPelessoniAndVicig} to be less intuitive than Definition~\ref{def:cohlp}, which is why we prefer the latter.

From a mathematical point of view, Definition~\ref{def:cohlp} also has the advantage that it allows for simple and elegant proofs of some well-known results. For example, it follows trivially from our definition of coherence that the domain of a coherent conditional lower prevision can be arbitrarily extended while preserving coherence, whereas deriving this result directly from Equation~\eqref{eq:prop:equivalentToPelessoniAndVicig} is substantially more involved; see for example the proof of \cite[Proposition~1]{Pelessoni:2009co}. Furthermore, our definition also allows for a very natural derivation of the so-called \emph{natural extension} of $\lp$, that is, the most conservative extension of $\lp$ to $\C(\states)$. In particular, instead of having to derive this natural extension directly, Definition~\ref{def:cohlp} allows us to rephrase this problem into a closely related yet simpler question: what is the smallest coherent set of desirable gambles $\desir$ on $\states$ such that $\lp_\desir$ coincides with $\lp$ on $\C$? The answer turns out to be surprisingly simple.

\begin{proposition}\label{prop:smallestSDGfromLP}
Consider a coherent conditional lower prevision $\lp$ on $\C\subseteq\C(\states)$ and let
\begin{equation}\label{eq:AfromLP}
\A_{\lp}
\coloneqq
\big\{
[f-\mu]\ind{B}
\colon
(f,B)\in\C, \mu<\lp(f\vert B)
\big\}
~~\text{and}~~
\E(\lp)\coloneqq\E(\A_{\lp}).
\end{equation}
Then $\E(\lp)$ is a coherent set of desirable gambles on $\states$ and $\lp_{\E(\lp)}$ coincides with $\lp$ on $\C$. Furthermore, for any other coherent set of desirable gambles $\desir$ on $\states$ such that $\lp_{\desir}$ coincides with $\lp$ on $\C$, we have that $\E(\lp)\subseteq\desir$.
\end{proposition}

Abstracting away some technical details, the reason why this result holds should be intuitively clear. First, since conditional lower previsions are interpreted as called-off supremum buying prices, we see that the gambles in $\A_{\lp}$ should be desirable. Combined with~\ref{def:SDG:partialgain}--\ref{def:SDG:convex}, the desirability of the gambles in $\E(\lp)$ then follows.

Since smaller sets of desirable gambles lead to more conservative---pointwise smaller---lower previsions, we conclude that the natural extension of $\lp$ is given by
\begin{equation}\label{eq:naturalextension}
\natexLP(f\vert B)\coloneqq\lp_{\E(\lp)}(f\vert B)
\text{~~for all $(f,B)\in\C(\states)$.}
\end{equation}
The following proposition  provides a formal statement of this result.

\begin{proposition}\label{prop:naturalextension:full}
Let $\lp$ be a coherent conditional lower prevision on $\C\subseteq\C(\states)$. Then $\natexLP$, as defined by Equation~\eqref{eq:naturalextension}, is the pointwise smallest coherent conditional lower prevision on $\C(\states)$ that coincides with $\lp$ on $\C$.
\end{proposition}

All in all, we conclude that Definition~\ref{def:cohlp} provides an intuitive as well as mathematically convenient characterisation of Williams-coherence that is furthermore equivalent to the structure-free version of~\cite{Pelessoni:2009co}. 
From a technical point of view, this equivalence will not be essential further on, because most of our arguments will be based on the connection with sets of desirable gambles. 
From a practical point of view though, this equivalence is highly important, because the Williams-coherent conditional lower previsions that are considered in~\citep{Pelessoni:2009co} are well-known to include as special cases a variety of other uncertainty models, including expectations, lower expectations, probabilities and lower probabilities. For that reason, our results can be applied to---and interpreted in terms of---these special cases as well. A detailed discussion of this point is deferred to Sections~\ref{sec:probs} and~\ref{sec:specialcases}; for now, we focus on sets of desirable gambles and conditional lower previsions. We end this section by listing some well-known properties of the latter; see for example~\cite{Pelessoni:2009co},~\cite{Williams:2007eu} and~\cite{Walley:1991vk}. \iftoggle{arxiv}{Proofs---or explicit references to proofs---can be found in Appendix~\ref{app:modellinguncertainty}.}

\begin{proposition}\label{prop:propertiesofLP}
Let $\lp$ be a coherent conditional lower prevision on $\C\subseteq\C(\states)$. Then for any two gambles $f,g\in\gambles$, any two events \mbox{$A,B\in\nonemptypower$}, any real number $\lambda\in\reals$ and any sequence of gambles $\{f_n\}_{n\in\nats}\subseteq\gambles$, whenever the involved conditional lower and upper previsions are well-defined---that is, if the arguments belong to their domain---we have that 
\vspace{5pt}

\begin{enumerate}[label=\emph{LP\arabic*:},ref=LP\arabic*]
\item
$\lp(f\vert B)\geq\inf_{x\in B}f(x)$\label{def:lowerprev:bounded}\hfill\emph{[boundedness]}
\item
$\lp(\lambda f\vert B)=\lambda\lp(f\vert B)$ if $\lambda\geq0$\label{def:lowerprev:homo}\hfill\emph{[non-negative homogeneity]}
\item
$\lp(f+g\vert B)\geq \lp(f\vert B)+\lp(g\vert B)$\label{def:lowerprev:superadditive}\hfill\emph{[superadditivity]}
\item
$\lp(\ind{B}[f-\lp(f\vert A\cap B)]\vert\, A)=0$ if $A\cap B\neq\emptyset$\label{def:lowerprev:GBR}\hfill\emph{[generalised Bayes rule]}
\item
$\lim_{n\to+\infty}\sup\abs{f-f_n}=0\then\lim_{n\to+\infty}\lp(f_n\vert B)=\lp(f\vert B)$
\label{def:lowerprev:uniformcontinuity}\hfill\emph{[uniform continuity]}
\item
$\lp(f+\lambda\vert B)=\lp(f\vert B)+\lambda$\label{def:lowerprev:constantadditivity}\hfill\emph{[constant additivity]}
\item
$f(x)\geq g(x)\text{ for all }x\in B\then\lp(f\vert B)\geq\lp(g\vert B)$\label{def:lowerprev:monotonicity}\hfill\emph{[monotonicity]}
\item
$\lp(f\vert B)\leq-\lp(-f\vert B)=\overline{P}(f\vert B)$\label{def:lowerprev:lowerbelowupper}
\end{enumerate}
\end{proposition}
\vspace{-6pt}

\begin{proposition}\label{prop:cohlpifffouraxioms}
Consider a set of events $\B\subseteq\nonemptypower$ that is closed under finite unions and let $\mathcal{F}\subseteq\gambles$  be a linear space of gambles such that $\ind{B}f\in\mathcal{F}$ and $\ind{B}\in\mathcal{F}$ for every $f\in\mathcal{F}$ and $B\in\B$. Now let $\C\coloneqq\{(f,B)\colon f\in\mathcal{F},B\in\B\}$. 
Then a conditional lower prevision $\lp$ on $\C$ is coherent if and only if it is real-valued and satisfies~\emph{\ref{def:lowerprev:bounded}--\ref{def:lowerprev:GBR}}. 
\end{proposition}

\begin{corollary}\label{corol:cohlpifffouraxioms}
A conditional lower prevision $\lp$ on $\C(\states)$ is coherent if and only if it is real-valued and satisfies~\emph{\ref{def:lowerprev:bounded}--\ref{def:lowerprev:GBR}}. 
\end{corollary}

\begin{corollary}\label{corol:cohlpiffthreeaxioms}
Consider a linear space of gambles $\G\subseteq\gambles$ that includes the constant gamble $1$.
A lower prevision $\lp$ on $\G$ is then coherent if and only if it is real-valued and satisfies~\emph{\ref{def:lowerprev:bounded}--\ref{def:lowerprev:superadditive}}. 
\end{corollary}

\section{Epistemic Independence}\label{sec:independence}

Having introduced our main tools for modelling uncertainty, the next step towards developing a notion of independent natural extension is to agree on what we mean by independence.

The approach that we adopt here is to define it as an assessment of mutual irrelevance. In particular, we say that $X_1$ and $X_2$ are independent if our uncertainty model for $X_1$ is not affected by conditioning on information about $X_2$, and vice versa. As we will see in Section~\ref{sec:specialcases}, this definition can be applied to a probability measure, and then yields the usual notion of independence. However, and that is what makes this approach powerful and intuitive, it can just as easily be applied to lower previsions, sets of desirable gambles, or any other type of uncertainty model. This type of independence is usually referred to as epistemic independence. The aim of this section is to formalize this concept for the case of two variables, in terms of sets of desirable gambles and conditional lower previsions.

Consider two variables $X_1$ and $X_2$ where, for every $i\in\{1,2\}$, $X_i$ takes values $x_i$ in a non-empty set $\states_i$ that may be uncountably infinite. We assume that $X_1$ and $X_2$ are logically independent, meaning that $X_1 = x_1$ and $X_2 = x_2$ are jointly possible, for all $x_1\in\states_1$ and $x_2\in\states_2$. The corresponding joint variable $X\coloneqq(X_1,X_2)$ therefore takes values $x\coloneqq(x_1,x_2)$ in $\states_1\times\states_2$. In this context, whenever convenient, we will identify $B_1\in\nonemptypoweron{1}$ with $B_1\times\states_2$ and $B_2\in\nonemptypoweron{2}$ with $\states_1\times B_2$. For any two events $B_1\in\nonemptypoweron{1}$ and $B_2\in\nonemptypoweron{2}$, this allows us to use $B_1\cap B_2$ as an intuitive alternative notation for $B_1\times B_2$.  Similarly, for any $i\in\{1,2\}$, we identify $f\in\gambleson{i}$ with its cylindrical extension to $\mathcal{G}(\states_1\times\states_2)$, defined by 
\begin{equation*}
f(x_1,x_2)\coloneqq f(x_i)
\text{~~for all $x=(x_1,x_2)\in\states_1\times\states_2$}.
\end{equation*}
In order to make this explicit, we will then often denote this cylindrical extension by $f(X_i)$. In this way, for example, for any $f\in\gambleson{2}$ and $B\in\poweron{1}$, we can write $f(X_2)\ind{B}(X_1)$ to denote a gamble in $\mathcal{G}(\states_1\times\states_2)$ whose value in $(x_1,x_2)$ is equal to $f(x_2)$ if $x_1\in B$ and equal to zero otherwise. Using these conventions, for any set of gambles $\desir$ on $\states_1\times\states_2$, we define the marginal models\vspace{2pt}
\begin{equation*}
\marg_1(\desir)\coloneqq\{f\in\gambleson{1}\colon f(X_1)\in\desir\}
~~\text{and}~~
\marg_2(\desir)\coloneqq\{f\in\gambleson{2}\colon f(X_2)\in\desir\}
\vspace{2pt}
\end{equation*}
and, for any events $B_1\in\nonemptypoweron{1}$ and $B_2\in\nonemptypoweron{2}$, the conditional models
\vspace{3pt}
\begin{equation*}
\marg_1(\desir\vert B_2)\coloneqq\{f\in\gambleson{1}\colon f(X_1)\ind{B_2}(X_2)\in\desir\}\vspace{-6pt}
\end{equation*}
and
\begin{equation*}
\marg_2(\desir\vert B_1)\coloneqq\{f\in\gambleson{2}\colon f(X_2)\ind{B_1}(X_1)\in\desir\}.\vspace{8pt}
\end{equation*}
Conditioning and marginalisation both preserve coherence: if $\desir$ is a coherent set of desirable gambles on $\states_1\times\states_2$, then $\marg_1(\desir)$ and $\marg_1(\desir\vert B_2)$ are coherent sets of desirable gambles on $\states_1$, and $\marg_2(\desir)$ and $\marg_2(\desir\vert B_1)$ are coherent sets of desirable gambles on $\states_2$.

That said, let us now recall our informal definition of epistemic independence, which was that the uncertainty model for $X_1$ is not affected by conditioning on information about $X_2$, and vice versa. In the context of sets of desirable gambles, this can now be formalized as follows:
\begin{equation*}
\marg_1(\desir\vert B_2)=\marg_1(\desir)
~~\text{and}~~
\marg_2(\desir\vert B_1)=\marg_2(\desir).
\end{equation*}
The only thing that is left to specify are the conditioning events $B_1$ and $B_2$ for which we want this condition to hold. We think that the most intuitive approach is to impose this for every $B_1\in\nonemptypoweron{1}$ and $B_2\in\nonemptypoweron{2}$, and will call this epistemic event-independence. However, this is not what is usually done. The conventional approach, which we will refer to as epistemic atom-independence, is to focus on singleton events of the type $B_1=\{x_1\}$ and $B_2=\{x_2\}$; see for example~\citep{Walley:1991vk} and~\citep{deCooman:2012vba}.\footnote{Readers that are familiar with some of my previous work \citep{de2015credal,debock2017williamsPMLR} may notice that I have changed terminology: what I now call epistemic event- and atom-independence, I previously referred to as epistemic subset- and value-independence. This new terminology was suggested to me by an anonymous reviewer, and I could not but agree that it indeed better reflects the meaning of the respective concepts.}

The main reason why epistemic atom-independence is the conventional go-to definition for epistemic independence is that Walley adopted it in his seminal book~\cite[Section~9.2]{Walley:1991vk}. Walley seems to take this choice for granted; we assume that this is a consequence of his focus on conditional lower previsions whose conditioning events belong to a (finite number of) partition(s). The advantage of using such a partition---and a set of atoms in particular---is that it can be regarded as representing the possible outcomes of an experiment, allowing for a quite natural study of statistical inference and updating. Partitions are also essential if one wiches to impose the controversial property of conglomerability~\citep{Miranda2013}. Nevertheless, we consider this focus on partitions---which is inherent in Walley's approach---to be overly restrictive; we prefer having the option to condition on any possible event, especially since one can always zoom in on a particular partition whenever needed or convenient. This serves as a first reason why we prefer epistemic event-independence over epistemic atom-independence.

Other than that, as we will see further on, epistemic event-independence also has several technical advantages; in fact, this will be one of the main conclusions of this contribution. For now, however, we postpone this debate between event- and atom-independence by adopting a very general approach that subsumes the former two as special cases. In particular, for every $i\in\{1,2\}$, we simply fix a generic set of conditioning events $\B_i\subseteq\nonemptypoweron{i}$. Epistemic atom-independence corresponds to choosing $\B_i=\states_i$, whereas epistemic event-independence corresponds to choosing $\B_i=\nonemptypoweron{i}$.

For sets of desirable gambles, this leads us to the following definition.

\begin{definition}
\label{def:irrSDG}
Let $\desir$ be a coherent set of desirable gambles on $\states_1\times\states_2$. Then $\desir$ is epistemically independent if, for any $i$ and $j$ such that $\{i,j\}=\{1,2\}$:
\begin{equation*}
\marg_i(\desir\vert B_j)=\marg_i(\desir)
\text{~~for all $B_j\in\B_j$.}
\end{equation*}
\end{definition}

For coherent lower previsions, as a prerequisite for defining epistemic independence, we require that the domain $\C\subseteq\C(\states_1\times\states_2)$ is independent, by which we mean that for any $i$ and $j$ such that $\{i,j\}=\{1,2\}$, any pair $(f_i,B_i)\in\C(\states_i)$ and any event $B_j\in\B_j$:
\begin{equation}\label{eq:independentdomain}
(f_i,B_i)\in\C
\asa
(f_i,B_i\cap B_j)\in\C.
\end{equation}
Other than that, we impose no restrictions on $\C$; its elements $(f,B)\in\C$ are for example not restricted to the types that appear in Equation~\eqref{eq:independentdomain}. As a result, the following definition of epistemic independence is applicable beyond the context of lower previsions. For example, by restricting the domain to indicators, we obtain a notion of epistemic independence that applies to conditional lower probabilities. A detailed discussion of these special cases, however, is deferred to Section~\ref{sec:probs}.

\begin{definition}\label{def:epistemicindependence:LP}
Let $\C\subseteq\C(\states_1\times\states_2)$ be an independent domain. A coherent conditional lower prevision $\lp$ on $\C$ is then epistemically independent if, for any $i$ and $j$ such that $\{i,j\}=\{1,2\}$:
\begin{equation*}
\lp(f_i\vert B_i)=\lp(f_i\vert B_i\cap B_j)
\text{~~for all $(f_i,B_i)\in\C(\states_i)\cap\C$ and $B_j\in\B_j$.}
\end{equation*}
\end{definition}
Another important feature of this definition is that $B_j$ is not only irrelevant to the unconditional lower previsions of local gambles $f_i$---in the sense that $\lp(f_i)=\lp(f_i\vert B_j)$---but also to their conditional local lower previsions---in the sense that $\lp(f_i\vert B_i)=\lp(f_i\vert B_i\cap B_j)$. This type of irrelevance is called h-irrelevance; see~\cite{Cozman:2013us} and~\cite{de2015credal}. Note however that this feature is optional within our framework; it only appears when $\C$ is sufficiently large. If instead $B_i=\states_i$ for all $(f_i,B_i)\in\C(\states_i)\cap\C$, then our definition reduces to the more simple requirement that $\lp(f_i)=\lp(f_i\vert B_j)$. The following example illustrates this subtle feature and also demonstrates the difference between epistemic atom- and event-independence.

\begin{example}\label{ex:independenceforlps}
Let $\states_1=\reals$ and $\states_2=\nats$ and consider any coherent conditional lower prevision $\lp$ on $\C\subseteq\C(\states_1\times\states_2)$ that is epistemically independent, with $\C$ an independent domain. For any $(f_1,B_1)\in\C(\states_1)\cap\C$ and $B_2\in\B_2$, it then follows from Definition~\ref{def:epistemicindependence:LP} that $\lp(f_1\vert B_1)=\lp(f_1\vert B_1\cap B_2)$. To make this more concrete, we now consider several examples.

We first consider the most powerful case, where the domain $\C$ is equal to $\C(\states_1\times\states_2)$---the largest possible domain---and where the type of independence that is considered is event-independence.
Now let $f_1\coloneqq\sin(X_1)$ and let $B_2\coloneqq\{2n\colon n\in\nats\}$ be the event that $X_2$ is even. For $B_1=\states_1$, we then find that $\lp(\sin(X_1))=\lp(\sin(X_1)\vert X_2~\mathrm{even})$, meaning that conditioning on the event that $X_2$ is even has no effect on the lower prevision of $\sin(X_1)$. In much the same way, for $B_1=\reals_{\geq0}$, we find that $\lp(\sin(X_1)\vert X_1\geq0)=\lp(\sin(X_1)\vert X_1\geq0\text{ and }X_2\text{ even})$, which means that the conditional lower prevision of $\sin(X_1)$ given $X_1\geq0$ does not change if we additionally condition on the event that $X_2$ is even. This second example provides a nice illustration of the fact that our definition of independence imposes mutual h-irrelevance rather than mutual irrelevance.

If we shrink the domain $\C$ sufficiently, to the extent that every $(f_1,B_1)\in\C(\states_1)\cap\C$ is of the form $(f_1,\states_1)$, then the added value of h-irrelevance disappears because our definition then only imposes assessments of the type $\lp(f_1)=\lp(f_1\vert B_2)$, such as, for example, the assessement $\lp(\sin(X_1))=\lp(\sin(X_1)\vert X_2~\text{even})$ that we have seen before.

The effect of replacing event-independence with atom-independence is also quite substantial. In particular, since $\B_2$ then changes from $\mathcal{P}_{\emptyset}(\nats)$ to $\nats$, we can then no longer condition on the event $B_2\coloneqq\{2n\colon n\in\nats\}$ that $X_2$ is even, because that event is not a singleton. Instead, for atom-indepence, we can only condition on events $B_2$ of the form $\{x_2\}$, such as, for example, the assessement $\lp(\sin(X_1))=\lp(\sin(X_1)\vert X_2=5)$.
\hfill$\lozenge$
\end{example}

\section{The Independent Natural Extension}\label{sec:indnatext}

All of that said, we are now finally ready to introduce our central object of interest, which is the \emph{independent natural extension}. Basically, the question to which this concept provides an answer is: given two local uncertainty models and an assessment of epistemic independence, what then should be the corresponding joint model? The answer depends on the specific framework that is being considered.

Within the framework of sets of desirable gambles, the local uncertainty models are coherent sets of desirable gambles. In particular, for each $i\in\{1,2\}$, we are given a coherent set of desirable gambles $\desir_i$ on $\states_i$. The aim is to combine these local models with an assessment of epistemic independence to obtain a coherent set of desirable gambles $\desir$ on $\states_1\times\states_2$. The first requirement on $\desir$, therefore, is that it should have $\desir_1$ and $\desir_2$ as its marginals, in the sense that $\marg_i(\desir)=\desir_i$ for all $i\in\{1,2\}$. The second is that $\desir$ should be epistemically independent. If both requirements are met, $\desir$ is called an independent product of $\desir_1$ and $\desir_2$. The most conservative among these independent products is called the independent natural extension.

\begin{definition}\label{def:subsetindependentproduct}
An independent product of $\desir_1$ and $\desir_2$ is an epistemically independent coherent set of desirable gambles $\desir$ on $\states_1\times\states_2$ that has $\desir_1$ and $\desir_2$ as its marginals.
\end{definition}

\begin{definition}\label{def:subsetindependentnatex}
The independent natural extension of $\desir_1$ and $\desir_2$ is the smallest independent product of $\desir_1$ and $\desir_2$.
\end{definition}

If all we know is that $\desir$ is epistemically independent and has $\desir_1$ and $\desir_2$ as its marginal models, then the safest choice for $\desir$---the only choice that does not require any additional assessments---is their independent natural extension, provided of course that it exists. In order to show that it always does, we let
\begin{equation}\label{eq:indnatext:SDG}
\desir_1\otimes\desir_2
\coloneqq
\E
\left(
\A_{1\to2}
\cup
\A_{2\to1}
\right),
\vspace{-4pt}
\end{equation}
with
\begin{equation}\label{eq:A12s}
\A_{1\to2}
\coloneqq
\left\{
f_2(X_2)\ind{B_1}(X_1)
\colon
f_2\in\desir_2, 
B_1\in\B_1\cup\{\states_1\}
\right\}
\vspace{-3pt}
\end{equation}
and
\begin{equation}\label{eq:A21s}
\A_{2\to1}
\coloneqq
\left\{
f_1(X_1)\ind{B_2}(X_2)
\colon
f_1\in\desir_1, 
B_2\in\B_2\cup\{\states_2\}
\right\}.
\vspace{8pt}
\end{equation}
The following result establishes that $\desir_1\otimes\desir_2$ is the independent natural extension of $\desir_1$ and $\desir_2$.

\begin{theorem}\label{theo:natext:SDG}
$\desir_1\otimes\desir_2$ is the independent natural extension of $\desir_1$ and $\desir_2$.
\end{theorem}

Similar concepts can be defined for conditional lower previsions as well. In that case, the local uncertainty models are coherent conditional lower previsions. In particular, for every $i\in\{1,2\}$, we are given a coherent conditional lower prevision $\lp_i$ on some freely chosen local domain $\C_i\subseteq\C(\states_i)$. Note that this freedom implies that $\lp_i$ can also be an unconditional; this corresponds to choosing $\C_i\coloneqq\{(f_i,\states_i)\colon f_i\in\G_i\}$ for some $\G_i\subseteq\gambleson{i}.$
In any case, the aim is now to construct an epistemically independent coherent conditional lower prevision $\lp$ on $\C\subseteq\C(\states_1\times\states_2)$ that has $\lp_1$ and $\lp_2$ as its marginals, in the sense that $\lp$ coincides with $\lp_1$ and $\lp_2$ on their local domain: $\lp(f_i\vert B_i)=\lp_i(f_i\vert B_i)$ for all $i\in\{1,2\}$ and $(f_i,B_i)\in\C_i$. As before, a model that meets these criteria is then called an independent product, and the most conservative among them is called the independent natural extension. Clearly, in order for these notions to make sense, the global domain $\C$ must at least include the local domains $\C_1$ and $\C_2$ and must furthermore be independent in the sense of Equation~\eqref{eq:independentdomain}. The definitions and results below take this for granted.

\begin{definition}\label{def:independentproduct:LP}
 An independent product of $\lp_1$ and $\lp_2$ is an epistemically independent coherent conditional lower prevision on $\C$ that has $\lp_1$ and $\lp_2$ as its marginals.
\end{definition}

\begin{definition}\label{def:independentnatex:LP}
The independent natural extension of $\lp_1$ and $\lp_2$ is the point-wise smallest independent product of $\lp_1$ and $\lp_2$.
\end{definition}

Here too, if all we know is that $\lp$ is epistemically independent and has $\lp_1$ and $\lp_2$ as its marginal models, then the safest choice for $\lp$---the only choice that does not require any additional assessments---is the independent natural extension, provided that it exists. The following result establishes that it does, by showing that it is a restriction of the operator $\lp_1\otimes\lp_2$, defined by
\vspace{2pt}
\begin{equation}\label{eq:indnatex:LP}
(\lp_1\otimes\lp_2)(f\vert B)
\coloneqq
\lp_{\desir}(f\vert B)
\text{~~for all $(f,B)\in\C(\states_1\times\states_2)$, with $\desir=\E(\lp_1)\otimes\E(\lp_2)$.}
\vspace{2pt}
\end{equation}

\begin{theorem}\label{theo:natext:LP}
The independent natural extension of $\lp_1$ and $\lp_2$ is the restriction of $\lp_1\otimes\lp_2$ to $\C$.
\end{theorem}
Interestingly, as can be seen from this result, the choice of the joint domain $\C$ does not affect the resulting independent natural extension, in the sense that any $\C$ that includes $(f,B)$ will lead to the same value of $(\lp_1\otimes\lp_2)(f\vert B)$. For that reason, we will henceforth assume without loss of generality that $\C=\C(\states_1\times\states_2)$.

\section{On the Choice of Conditioning Events}\label{sec:choiceofevents}

The fact that the existence results in the previous section are valid regardless of the choice of $\B_1$ and $\B_2$ should not be taken to mean that this choice does not affect the model. In some cases, it most definitely does. In the remainder of this contribution, we will study the extent to which it does, and how it affects the properties of the resulting notion of independent natural extension.

As a first observation, we note that larger sets of conditioning events correspond to stronger assessments of epistemic independence, and therefore lead to more informative joint models. For example, as can be seen from Equations~\eqref{eq:indnatext:SDG}--\eqref{eq:A21s}, adding events to $\B_1$ and $\B_2$ leads to a larger---more informative---set of desirable gambles $\desir_1\otimes\desir_2$. Similarly, as can be seen from Equation~\eqref{eq:indnatex:LP}, it leads to a joint lower prevision that is higher---and therefore again more informative.
There is one important exception to this observation though, which occurs when we add conditioning events that are a finite disjoint union of other conditioning events. In that case, the resulting notion of independent natural extension does not change.

\begin{proposition}\label{prop:addfiniteunionstoB}
For each $i\in\{1,2\}$, let $\B'_i$ be a superset of $\B_i$ that consists of finite disjoint unions of events in $\B_i$. Replacing $\B_1$ by $\B'_1$ and $\B_2$ by $\B'_2$ then has no effect on the resulting independent natural extension $\desir_1\otimes\desir_2$ or $\lp_1\otimes\lp_2$. 
\end{proposition}

As a particular case of this result, it follows that if $\B_i$ is a finite partition of $\states_i$, we can replace it by the generated algebra---minus the empty event. As an even more particular case, if $\states_1$ and $\states_2$ are finite, we find that epistemic atom- and event-independence lead to the same notion of independent natural extension. For that reason, in the finite case, it does not really matter which of these two types of epistemic independence is adopted.

In the infinite case though, we will see that the difference does matter, which requires one to choose between epistemic atom- and event-independence. For lower previsions,~\cite{Miranda2015460} recently adopted epistemic atom-independence in combination with Walley-coherence. Unfortunately, they found that the corresponding notion of independent natural extension does not always exist. They also considered the combination of epistemic atom-independence with Williams-coherence, and argued that the resulting model was too weak.
For the case of lower probabilities,~\cite{Vicig:2000vh} adopted epistemic event-independence in combination with Williams-coherence, showed that the corresponding independent natural extension always exists, and proved that it satisfies factorisation properties. Our results so far can be regarded as a generalisation of the existence results of~\cite{Vicig:2000vh}. 
As we are about to show, his factorisation results can be generalised as well.

\section{Factorisation and External Additivity}\label{sec:factadd}

When $\states_1$ and $\states_2$ are finite, the independent natural extension of two lower previsions $\lp_1$ and $\lp_2$ is well-known to satisfy the properties of factorisation and external additivity~\citep{deCooman:2011ey}. Factorisation, on the one hand, states that
\begin{equation}\label{eq:finitefactorisation}
(\lp_1\otimes\lp_2)(gh)=\lp_1(g\lp_2(h))=
\begin{cases}
\lp_1(g)\lp_2(h)
&\text{ if $\lp_2(h)\geq0$}\\
\overline{P}_1(g)\lp_2(h)
&\text{ if $\lp_2(h)\leq0$},
\end{cases}
\end{equation}
where $g$ is a non-negative gamble on $\states_1$, $h$ is a gamble on $\states_2$ and $\overline{P}_1(g)\coloneqq-\lp_1(-g)$. By symmetry, the role of $1$ and $2$ can of course be reversed. External additivity, on the other hand, states that
\begin{equation}\label{eq:finiteexternaladditivity}
(\lp_1\otimes\lp_2)(f+h)=\lp_1(f)+\lp_2(h)
\end{equation}
where $f$ and $h$ are gambles on $\states_1$ and $\states_2$, respectively. 

Compared to the properties that are satisfied by the joint expectation of a product measure of two probability measures, these notions of factorisation and external additivity are rather weak. For example, for a product measure, additivity is not `external', in the sense that $f$ and $h$ do not have to be defined on separate variables, nor does factorisation require $g$ to be non-negative. Nevertheless, even in this weaker form, these properties remain of crucial practical importance. For example, as explained in \citep{deCooman:2011ey}, factorisation properties such as Equation~\eqref{eq:finitefactorisation}---when applied to more than two variables---are sufficient in order to establish laws of large numbers for lower previsions \citep{deCooman:2008}. As another example, in the context of credal networks, which are Bayesian networks whose local models are partially specified, properties such as Equations~\eqref{eq:finitefactorisation} and~\eqref{eq:finiteexternaladditivity} turned out to be the key to the development of efficient inference algorithms; see for example~\citep{deCooman:2010gd},~\citep{DeBock:2014ts} and~\citep{de2015credal}. Any notion of independent natural extension that aims to extend such algorithms to infinite spaces, therefore, should preserve some suitable version of Equations~\eqref{eq:finitefactorisation} and~\eqref{eq:finiteexternaladditivity}.

The aim of this section is to study the extent to which these equations are satisfied by the notion of independent natural extension that was developed in this paper. As we will see, the answer ends up being surprisingly positive. 

For all $i\in\{1,2\}$, let $\lp_i$ be a coherent conditional lower prevision on $\C_i\subseteq\C(\states_i)$, let $\natexLP_i$ be its natural extension to $\C(\states_i)$, and let $\B_i$ be a subset of $\nonemptypoweron{i}$. The independent natural extension of $\lp_1$ and $\lp_2$ then satisfies the following three properties, the first of which implies the other two as special cases.

\begin{theorem}\label{theo:fact-add-measurable}
Let $\{i,j\}=\{1,2\}$. For any $f\in\gambleson{i}$, $h\in\gambleson{j}$ and $\B_i$-measurable $g\in\mathcal{G}_{\geq0}(\states_i)$, we then have that
\begin{equation*}
(\lp_1\otimes\lp_2)(f+gh)
=\natexLP_i\big(f+g\natexLP_j(h)\big).
\vspace{6pt}
\end{equation*}
\end{theorem}

\begin{corollary}[Factorisation]\label{corol:fact-measurable}
Let $\{i,j\}=\{1,2\}$. For any $h\in\gambleson{j}$ and any $g\in\mathcal{G}_{\geq0}(\states_i)$ that is $\B_i$-measurable, we then have that
\vspace{-5pt}
\begin{equation*}
(\lp_1\otimes\lp_2)(gh)
=\natexLP_i\big(g\natexLP_j(h)\big)
=
\begin{cases}
\natexLP_i(g)\natexLP_j(h)
&\text{ if $\natexLP_j(h)\geq0$;}\\
\natexUP_i(g)\natexLP_j(h)
&\text{ if $\natexLP_j(h)\leq0$.}
\end{cases}
\vspace{6pt}
\end{equation*}
\end{corollary}

\begin{corollary}[External additivity]\label{corol:add}
For any $f\in\gambleson{1}$ and $h\in\gambleson{2}$, we have that
\begin{equation*}
{(\lp_1\otimes\lp_2)(f+h)
=\natexLP_1(f)+\natexLP_2(h)}.
\vspace{2pt}
\end{equation*}
\end{corollary}
In each of these results, if the local domains $\C_1$ and $\C_2$ are sufficiently large---that is, if they include the gambles that appear in the statement of the results---it follows from Proposition~\ref{prop:naturalextension:full} that $\natexLP_i$ and $\natexLP_j$ can be replaced by $\lp_i$ and $\lp_j$, respectively, and similarly for $\natexUP_i$ and $\overline{P}_i$.

That said, let us now go back to the question of whether or not Equations~\eqref{eq:finitefactorisation} and~\eqref{eq:finiteexternaladditivity} can be generalised  to the case of infinite spaces. For the case of external additivity, it clearly follows from Corollary~\ref{corol:add} that the answer is fully positive.
Furthermore, this conclusion holds regardless of our choice for $\B_1$ and $\B_2$; they can even be empty. For factorisation, the answer does depend on $\B_1$ and $\B_2$. If we adopt epistemic event-independence---that is, if we choose $\B_1=\nonemptypoweron{1}$ and $\B_2=\nonemptypoweron{2}$---it follows from Corollaries~\ref{corol:measurable:sufficient:allsets} and~\ref{corol:fact-measurable} that the answer is again fully positive, because $\nonemptypoweron{i}$-measurability then holds trivially. If $\B_1\cup\{\emptyset\}$ and $\B_2\cup\{\emptyset\}$ are sigma fields, the answer remains fairly positive as well, because Proposition~\ref{prop:measurability:equivalenceforsigmafield} then implies that it suffices for $g$ to be measurable in the usual, measure-theoretic sense.

\begin{example}\label{ex:factorisationforreals}
Let $\states_1=\states_2=\reals$ and let $\B_1=\B_2=\B$, with $\B=\B^*\setminus\{\emptyset\}$  and with $\B^*$ the $\sigma$-algebra of Lebesgue measurable subsets of $\reals$. Furthermore, let $g\in\G_{\geq0}(\states_1)$ be the indicator of the nonnegative reals, defined for all $x_1\in\reals$ by $g(x_1)\coloneqq1$ if $x_1\geq0$ and $g(x_1)\coloneqq0$ otherwise, and let $h\in\G(\states_2)$ be defined by $h(x_2)\coloneqq x_2^3$ for all $x_2\in\reals$. We then know from Example~\ref{ex:onthereals} that $g$ and $h$ are both Lebesgue-measurable and therefore also $\B$-measurable. Therefore, for any two coherent conditional lower previsions $\lp_1$ on $\C(\states_1)$ and $\lp_2$ on $\C(\states_2)$, Corollary~\ref{corol:fact-measurable} implies that
\begin{equation*}
(\lp_1\otimes\lp_2)(gh)=
\begin{cases}
\lp_1(g)\lp_2(h)&\text{ if $\lp_2(h)\geq0$;}\\
\overline{P}_1(g)\lp_2(h)&\text{ if $\lp_2(h)\leq0.$}
\end{cases}
\vspace{-8pt}
\end{equation*}
\hfill$\lozenge$
\end{example}

However, these positive conclusions do not apply if we adopt epistemic atom-independence---that is, if we choose $\B_1=\states_1$ and $\B_2=\states_2$---because our factorisation result then requires $g$ to be $\states_i$-measurable, which, as we know from Example~\ref{ex:oddandeven}, is a rather strong requirement that easily fails. Nevertheless, the factorisation properties that do hold for atom-independence are stronger than what is suggested in~\cite{Miranda2015460}. In order to illustrate this, we consider the following adaptation of their Example 4.

\begin{example}\label{ex:oneovern:factorisation}
Let $\states_1=\states_2=\nats$ and let $\lp_1$ and $\lp_2$ be coherent lower previsions on $\gambleson{1}$ and $\gambleson{2}$, respectively. Let $g\in\mathcal{G}_{\geq0}(\states_1)$ and $h\in\mathcal{G}_{\geq0}(\states_2)$ be defined by $g(x_1)\coloneqq\nicefrac{1}{x_1}$ and $h(x_2)\coloneqq\nicefrac{1}{x_2}$~~for all $x_1\in\states_1$ and $x_2\in\states_2$. We then know from Example~\ref{ex:oneovern} that the gambles $g$ and $h$ are both $\mathcal{P}_{\emptyset}(\nats)$-measurable and $\nats$-measurable. Furthermore, since $g$ and $h$ are nonnegative,~\ref{def:lowerprev:bounded} implies that $\lp_1(g)\geq0$ and $\lp_2(h)\geq0$. Hence, regardless of whether we adopt epistemic event- or atom-independence, we can apply Corollary~\ref{corol:fact-measurable} to find that $(\lp_1\otimes\lp_2)(gh)
=
\lp_1(g)\lp_2(h)$.
\hfill$\lozenge$
\end{example}

For readers that are familiar with the work of~\cite{Miranda2015460}, the final conclusion of this example may seem surprising at first, because in their Example 4, Miranda and Zaffalon show that for the same two gambles $g$ and $h$, for atom-independence and Williams-coherence, the independent natural extension assigns lower prevision zero to $gh$, regardless of the chosen local models $\lp_1$ and $\lp_2$. In contrast, our example above concludes that this lower prevision is equal to the product of $\lp_1(g)$ and $\lp_2(h)$. This apparent contradiction is a consequence of the fact that Miranda and Zaffalon do not require the independent natural extension of $\lp_1$ and $\lp_2$ to have $\lp_1$ and $\lp_2$ as its marginals. Instead---using their notation and terminology---they only require weak coherence with $\lp_1(\cdot\vert X_2)$ and $\lp_2(\cdot\vert X_1)$; see~\citep{Miranda2015460} for more information. As we can see here, this leads to a notion of independent natural extension that satisfies fewer factorisation properties. Our Definitions~\ref{def:independentproduct:LP} and~\ref{def:independentnatex:LP} avoid this, by explicitly imposing that the independent natural extension of $\lp_1$ and $\lp_2$ should have $\lp_1$ and $\lp_2$ as its marginals.

Still, even with our strengthened definition, the issue remains that for atom-independence, our factorisation result requires the stringent assumption that $g$ is $\states_i$-measurable. This issue is fundamental because, as our next example demonstrates, it is not just a feature of Corollary~\ref{corol:fact-measurable} but rather an inherent property of atom-independence: epistemic atom-independence indeed leads to weaker factorisation properties.

\begin{example}\label{ex:oddandeven:factorisation}
Let $\states_1=\states_2=\nats$ and let $g=\ind{\mathrm{odd}}\in\mathcal{G}_{\geq0}(\states_1)$ and $h=\ind{\mathrm{even}}\in\mathcal{G}_{\geq0}(\states_2)$, with $\ind{\mathrm{odd}}$ and $\ind{\mathrm{even}}$ defined as in Example~\ref{ex:oddandeven}. Furthermore, let $\lp_1$ be a coherent lower prevision on $\gambleson{1}$ and let $\lp_2$ be a coherent lower prevision on $\gambleson{2}$. Since $g$ and $h$ are nonnegative,~\ref{def:lowerprev:bounded} then implies that $\lp_1(g)\geq0$ and $\lp_2(h)\geq0$. 
Therefore, and because we know from Example~\ref{ex:oddandeven} that $g$ and $h$ are $\mathcal{P}_{\emptyset}(\nats)$-measurable, we can apply Corollary~\ref{corol:fact-measurable} to find that for event-independence: $(\lp_1\otimes\lp_2)(gh)=\lp_1(g)\lp_2(h)$. However, unfortunately, this corollary cannot be applied for atom-independence, because we know from Example~\ref{ex:oddandeven} that $g$ and $h$ are not $\nats$-measurable. 

Of course, one could still believe that factorisation can be established in some other way, and that it is simply Corollary~\ref{corol:fact-measurable} that is lacking in power, rather than the concept of atom-independence itself. This is however not the case: for atom-independence, as we will demonstrate in Example~\ref{ex:oddandeven:lower}, $(\lp_1\otimes\lp_2)(\ind{\mathrm{odd}}(X_1)\ind{\mathrm{even}}(X_2))$ can be strictly smaller than $\lp_1(\ind{\mathrm{odd}}(X_1))\lp_2(\ind{\mathrm{even}}(X_2))$\hfill$\lozenge$
\end{example}

Because of these weak factorisation properties, we think that for the case of infinite spaces, when it comes to choosing between epistemic atom- and event-independence, the latter should be preferred over the former. That is not the only reason though. There is also a second, closely related reason, which is that event-independence leads to much more informative inferences. However, in order to explain and demonstrate that, we first need to establish a connection between conditional lower previsions, probabilities and expectations, which is what we now set out to do.

\section{Connecting lower previsions with expectations and probabilities}\label{sec:probs}

The key to understanding the connection between lower previsions, expectations and probabilities is to consider conditional lower previsions that are self-conjugate, in the sense that they coincide with their corresponding upper prevision. In that case, we simply refer to them as \emph{conditional previsions} and denote them by $\pr$ instead of $\lp$. 

\begin{definition}[Conditional prevision]\label{def:prev}
A conditional prevision $\pr$ on $\C\subseteq\C(\states)$ is a conditional lower prevision on $\C$ that is self-conjugate, in the sense that 
\begin{equation}
(-f,B)\in\C
~~\text{and}~~
\pr(f\vert B)=-\pr(-f\vert B)
~~\text{for all $(f,B)\in\C$.}
\label{eq:self-conjugate}
\vspace{3pt}
\end{equation}
\end{definition}

\emph{Unconditional previsions} correspond to a special case. First, if $B=\states$, then similarly to what we did for conditional lower previsions, we adopt the shorthand notation $\pr(f)\coloneqq \pr(f\vert\states)$ and call $\pr(f)$ the prevision of $f$. Second, if $B=\states$ for all $(f,B)\in\C$, meaning that there is some $\G\subseteq\gambles$ such that $\C=\{(f,\states)\colon f\in\G\}$, we regard $\pr$ as an operator on $\G$ and then call $\pr$ a(n unconditional) \emph{prevision}.

If a conditional prevision is coherent, we refer to it as a \emph{conditional linear prevision}. Similarly, coherent (unconditional) previsions are called \emph{linear previsions}.

\begin{definition}[Conditional linear prevision]\label{def:linearprev}
A conditional linear prevision $\pr$ on $\C\subseteq\C(\states)$ is a coherent conditional prevision on $\C$.
\end{definition}

The reason for this terminology, quite obviously, is that conditional linear previsions can be shown to be linear operators. In fact, they satisfy various other convenient properties as well, which, for the sake of completeness, are listed below.

\begin{proposition}\label{prop:propertiesofP}
Let $\pr$ be a conditional linear prevision on $\C\subseteq\C(\states)$. Then for any two gambles $f,g\in\gambles$, any two events \mbox{$A,B\in\nonemptypower$}, any real number $\lambda\in\reals$ and any sequence of gambles $\{f_n\}_{n\in\nats}\subseteq\gambles$, whenever the involved conditional previsions are well-defined---that is, if the arguments belong to their domain---we have that 
\vspace{5pt}

\begin{enumerate}[label=\emph{P\arabic*:},ref=P\arabic*]
\item
$\pr(f\vert B)\geq\inf_{x\in B}f(x)$\label{def:prev:bounded}\hfill\emph{[boundedness]}
\item
$\pr(\lambda f\vert B)=\lambda\pr(f\vert B)$\label{def:prev:homo}\hfill\emph{[homogeneity]}
\item
$\pr(f+g\vert B)=\pr(f\vert B)+\pr(g\vert B)$\label{def:prev:additive}\hfill\emph{[additivity]}
\item
$\pr(\ind{B}f\vert A)=\pr(f\vert A\cap B)\pr(B\vert A)$ if $A\cap B\neq\emptyset$\label{def:prev:GBR}\hfill\emph{[Bayes rule]}
\item
$\lim_{n\to+\infty}\sup\abs{f-f_n}=0\then\lim_{n\to+\infty}\pr(f_n\vert B)=\pr(f\vert B)$
\label{def:prev:uniformcontinuity}\hfill\emph{[uniform continuity]}
\item
$\pr(f+\lambda\vert B)=\pr(f\vert B)+\lambda$\label{def:prev:constantadditivity}\hfill\emph{[constant additivity]}
\item
$f(x)\geq g(x)\text{ for all }x\in B\then\pr(f\vert B)\geq\pr(g\vert B)$\label{def:prev:monotonicity}\hfill\emph{[monotonicity]}
\end{enumerate}
\vspace{0pt}
\end{proposition}

\begin{proposition}\label{prop:linearPifffouraxioms}
Consider a set of events $\B\subseteq\nonemptypower$ that is closed under finite unions and let $\mathcal{F}\subseteq\gambles$  be a linear space of gambles such that $\ind{B}f\in\mathcal{F}$ and $\ind{B}\in\mathcal{F}$ for every $f\in\mathcal{F}$ and $B\in\B$. Now let $\C\coloneqq\{(f,B)\colon f\in\mathcal{F},B\in\B\}$. 
Then a conditional prevision $\pr$ on $\C$ is a conditional linear prevision on $\C$ if and only if it is real-valued and satisfies~\emph{\ref{def:prev:bounded}--\ref{def:prev:GBR}}. 
\end{proposition}

\begin{corollary}\label{corol:fulllinearPifffouraxioms}
A conditional prevision $\pr$ on $\C(\states)$ is linear if and only if it is real-valued and satisfies~\emph{\ref{def:prev:bounded}--\ref{def:prev:GBR}}. 
\end{corollary}

\begin{corollary}\label{corol:fulllinearPiffthreeaxioms}
Consider a linear space of gambles $\G\subseteq\G(\states)$ that includes the constant gamble $1$. A prevision $\pr$ on $\G$ is then linear if and only if it is real-valued and satisfies~\emph{\ref{def:prev:bounded}--\ref{def:prev:additive}}.
\end{corollary}

By comparing these properties with the ones in Section~\ref{sec:modellinguncertainty}, we see that the linearity of conditional linear previsions---the fact that they satisfy \ref{def:prev:homo} and~\ref{def:prev:additive}---is their most important property, in the sense that it distinguishes them from general coherent conditional lower previsions. Furthermore, this property is also what allows us to establish a connection with expectations. In particular, if we allow ourselves a small leap of faith here, then since expectations are well known to be linear, the fact that conditional linear previsions are also linear suggests that we can simply interpret them as conditional expectations.

In order to clarify why this is more than just intuition, it is instrumental to restrict the domain of $\pr$ to elements that are of the form $(\ind{A},B)$, where $\ind{A}$ is the indicator of an event $A$, and to then follow~\cite{DeFinetti:1970vq} in adopting the alternative notation $\pr(A\vert B)\coloneqq\pr(\ind{A}\vert B)$. As this notation already suggests, $\pr(A\vert B)$ can then be interpreted as the probability of $A$ conditional on $B$. This interpretation is furthermore mathematically sound, because the obtained objects $\pr(A\vert B)$ can be shown to satisfy all the essential properties of conditional probabilities, including finite---but not necessarily countable---additivity and, if $B\neq\emptyset$, Bayes's rule. Hence, by restricting the domain of a conditional linear prevision $\pr$ to elements of the form $(\ind{A},B)$, we obtain a conditional probability measure. The original unrestricted conditional linear prevision $\pr$ is then the conditional expectation operator that corresponds to this conditional probability measure. Here too, this connection is not merely intuitive, but can be made mathematically rigorous. A detailed account of the mathematics behind these connections, however, is beyond the scope of this contribution. For more information about (finitely additive) conditional probability measures, the interested reader is referred to the work of \cite{Dubins:1975}. 

For our present purposes, it suffices to know that conditional linear previsions can indeed be interpreted as conditional expectation operators and that conditional probabilities are conditional linear previsions whose domain contains only---or is restricted to---elements of the form $(\ind{A},B)$. A similar observation applies to unconditional expectation operators and probability measures, with the role of the conditional linear prevision now taken up by an unconditional one.

Since conditional linear previsions are themselves a special (self-conjugate) case of coherent conditional lower previsions, we conclude that conditional expectations and conditional probability measures can both be regarded as special cases of conditional coherent lower previsions. Similarly, unconditional expectations and probability measures are special cases of coherent lower previsions. However, the connection goes much further, because conditional linear previsions are not just a special case of coherent conditional lower previsions: they can also be used to characterise them.

\begin{proposition}\label{prop:lowerenvelope}
A conditional lower prevision $\lp$ on $\C\subseteq\C(\states)$ is coherent if and only if there is a non-empty set $\mathbb{P}^*$ of conditional linear previsions on $\C(\states)$ such that
\begin{equation}\label{eq:prop:lowerenvelope}
\lp(f\vert B)=\inf\{P(f\vert B)\colon P\in\mathbb{P}^*\}
~~\text{for all $(f,B)\in\C$.}
\end{equation}
The same is true if the infimum in this expression is replaced by a minimum.
\end{proposition}

This well-known result is essentially due to~\cite{williams1975,Williams:2007eu}\iftoggle{arxiv}{; our proof is a minor variation of his}{}. The result is fundamental, because it provides coherent conditional lower previsions with a second, alternative interpretation. Indeed, because of Proposition~\ref{prop:lowerenvelope}, a conditional lower prevision is not only a supreming buying price; alternatively, it can also be regarded as an infimum of conditional previsions. Since---as we have just seen---conditional previsions can themselves be interpreted as conditional expectations, this implies that coherent conditional lower previsions can be interpreted as lower envelopes of expectations, often referred to as lower expectations.

This interpretation can also be used to develop an alternative characterisation for the natural extension $\natexLP$ of a coherent conditional lower prevision $\lp$ on $\C\subseteq\C(\states)$.
In order to do that, we let $\mathbb{P}$ be the set of all conditional linear previsions on $\C(\states)$ and then let
\begin{equation}\label{eq:setofdominatingcondlinprev}
\mathbb{P}_{\lp}\coloneqq\{P\in\mathbb{P}\colon P(f\vert B)\geq\lp(f\vert B)\text{ for all }(f,B)\in\C\}
\end{equation}
be the subset that dominates $\lp$. The natural extension $\natexLP$ is then the lower envelope of $\mathbb{P}_{\lp}$ and, similarly, $\natexUP$ is its upper envelope.

\begin{proposition}\label{prop:lowerenvelope:onlyif:natex}
Let $\lp$ be a coherent conditional lower prevision $\lp$ on $\C\subseteq\C(\states)$, let $\natexLP$ be its natural extension to $\C(\states)$ and let $\natexUP$ be the corresponding conditional upper prevision on $\C(\states)$. Then $\mathbb{P}_{\lp}\neq\emptyset$ and, for all $(f,B)\in\C(\states)$:
\begin{equation}\label{eq:prop:lowerenvelope:onlyif:natex}
\natexLP(f\vert B)=\min\{P(f\vert B)\colon P\in\mathbb{P}_{\lp}\}
\text{~~and~~}
\natexUP(f\vert B)=\max\{P(f\vert B)\colon P\in\mathbb{P}_{\lp}\}.
\end{equation}
Also, for any $(f,B)\in\C(\states)$ and $\alpha\in[\natexLP(f\vert B),\natexUP(f\vert B)]$, there is some $P\in\mathbb{P}_{\lp}$ such that $P(f\vert B)=\alpha$.
\end{proposition}

The final connection that remains to be discussed is that between conditional lower previsions and conditional lower probabilities. However, since conditional lower previsions are lower envelopes of conditional linear previsions, and since conditional probability measures are conditional linear previsions whose domain is restricted to elements of the form $(\ind{A},B)$, this connection is immediate: conditional lower probabilities are simply conditional lower previsions whose domain is restricted to elements of the form $(\ind{A},B)$. In that case, in order to emphasize this, we adopt $\lp(A\vert B)$ as an intuitive alternative notation for $\lp(\ind{A}\vert B)$.

\section{Special cases of the independent natural extension}\label{sec:specialcases}

Now that we have established that expectations, lower expectations, probabilities and lower probabilities are indeed all special cases of lower previsions, we can come back to our claim at the end of Section~\ref{sec:modellinguncertainty}, which was that our results can be applied to---and interpreted in terms of---these special cases as well. 

Applying our results to the case of lower expectations is straightforward. Mathematically, nothing changes. The only difference is that the local conditional lower previsions that we start from are now interpreted---or defined---as lower bounds on expectations, and similarly for the independent natural extension that is derived from them. For lower probabilities, it suffices to restrict the domain of the local models $\lp_1$ and $\lp_2$ to elements of the form $(\ind{A_1},B_1)$ and $(\ind{A_2},B_2)$, respectively, and to similarly restrict the domain $\C$ of $\lp_1\otimes\lp_2$ to elements of the form $(\ind{A},B)$. Other than that, here too, the only difference is the interpretation. For results that are tailored to this specific case, we refer the interested reader to the work of~\cite{Vicig:2000vh}, who focused on the independent natural extension for lower probabilities, but within a more general context that allows for more than two variables. As explained before, our results are basically a generalisation of his, extending them from lower probabilities to conditional lower previsions.

Once we interpret the independent natural extension as a lower expectation or a lower probability, it makes sense to consider the set of conditional linear previsions $\mathbb{P}_{\lp_1\otimes\lp_2}$ that dominates the independent natural extension $\lp_1\otimes\lp_2$, and to then interpret the latter in terms of the elements of the former. An essential observation here is that the elements of $\mathbb{P}_{\lp_1\otimes\lp_2}$ need not be independent themselves, nor is this the case for its extreme points. Consequently, the independent natural extension is not in general a lower envelope of precise independent models~\cite[Section~9.3.4]{Walley:1991vk}. For example, for any $f_1\in\gambleson{1}$ and $B_2\in\B_2$, since $\lp_1\otimes\lp_2$ is an independent product of $\lp_1$ and $\lp_2$, it follows from Definitions~\ref{def:independentproduct:LP} and~\ref{def:epistemicindependence:LP} that
\begin{equation*}
(\lp_1\otimes\lp_2)(f_1\vert B_2)
=(\lp_1\otimes\lp_2)(f_1)
=\lp_1(f_1).
\vspace{3pt}
\end{equation*}
However, for $P\in\mathbb{P}_{\lp_1\otimes\lp_2}$, this does not necessarily imply that $P(f_1\vert B_2)=P(f_1)$. Instead, the only constraint that is imposed on $P(f_1\vert B_2)$ and $P(f_1)$ is that they both belong to $[\lp_1(f_1),\overline{P}_1(f_1)]$:
\begin{equation*}
\lp_1(f_1)\leq P(f_1\vert B_2)\leq\overline{P}_1(f_1)
\text{ and }
\lp_1(f_1)\leq P(f_1)\leq\overline{P}_1(f_1).
\end{equation*}
This feature is an essential aspect of epistemic independence: it imposes independence on the uncertainty model itself. If this uncertainty model is a set of conditional expectations, then epistemic independence imposes constraints on this set---in this case, on the resulting lower and upper expectations---but not on the individual expectations themselves. Similarly, for lower probabilities, epistemic independence does not require $P(A_1\vert B_2)$ and $P(A_1)$ to be equal, but only requires that
\begin{equation}\label{eq:maybecorrelated}
\lp_1(A_1)\leq P(A_1\vert B_2)\leq\overline{P}_1(A_1)
\text{ and }
\lp_1(A_1)\leq P(A_1)\leq\overline{P}_1(A_1).
\end{equation}
In this sense, epistemic independence requires that our knowledge about $P(A_1\vert B_2)$ and $P(A_1)$ is identical: conditioning on $B_2$ should have no effect on our bounds for $P(A_1)$. This also explains the prefix epistemic: epistemic independence imposes a constraint on our knowledge. In other words, and at the risk of oversimplifying it, one could say that an assessment of epistemic independence does not entail a belief of independence---because, for a given $P\in\mathbb{P}_{\lp_1\otimes\lp_2}$, $X_1$ and $X_2$ may very well be correlated---but rather an independence of beliefs.

Nevertheless, the standard notion of independence does correspond to a special case of epistemic independence: it suffices to use---restrictions of---two linear previsions $P_1$ on $\gambleson{1}$ and $P_2$ on $\gambleson{2}$ as our local models. Indeed, in that case, since the linearity of $P_1$ implies that $\lp_1(A_1)=\overline{P}_1(A_1)=P_1(A_1)$, Equation~\eqref{eq:maybecorrelated} implies that for all $P\in\mathcal{P}_{P_1\otimes P_2}$:
\begin{equation*}
P(A_1\vert B_2)=P(A_1)
\text{ for all $A_1\in\mathcal{P}(\states_1)$ and $B_2\in\B_2$,}
\end{equation*}
 and similarly if the indexes $1$ and $2$ are reversed. Furthermore, it then follows from Corollary~\ref{corol:fact-measurable} and Property~\ref{def:prev:bounded} that $P(A_1\cap B_2)=P(A_1)P(B_2)$, which is the conventional and well known defining factorisation property of independence.

We end by taking an even closer look at this specific case. So consider a linear prevision $P_1$ on $\gambleson{1}$ and a linear prevision $P_2$ on $\gambleson{2}$ or---in case we want to consider probability measures instead of expectations---their restrictions to a suitable set of indicators. Using the notation that we have adopted so far, the independent natural extension of these local models is then denoted by $P_1\otimes P_2$. In this case however, this notation is a bit unfortunate, because it suggests that $P_1\otimes P_2$ is a (conditional) linear prevision itself, which, as we will see, may not be the case. Therefore, we will adopt $\pr_1\underline{\otimes}\pr_2$ as an alternative---less suggestive---notation for $\pr_1\otimes\pr_2$ and will then use $\pr_1\overline{\otimes}\pr_2$ to denote its corresponding conditional upper prevision, defined by
\begin{equation*}
(\pr_1\overline{\otimes}\pr_2)(f\vert B)\coloneqq-(\pr_1\otimes\pr_2)(-f\vert B)
\text{ for all $(f,B)\in\C(\states_1\times\states_2)$.}
\end{equation*}

For atom-independence, the difference between $(\pr_1\overline{\otimes}\pr_2)(f\vert B)$ and $(\pr_1\underline{\otimes}\pr_2)(f\vert B)$ can be surprisingly large. For example, if the local models assign probability zero to all the singletons, then as already pointed out in the work of~\cite{Miranda2015460}, any joint linear prevision that  marginalises to these local models will dominate the independent natural extension.

\begin{proposition}\label{prop:dominatedbysingletonzero}
Consider a linear prevision $\pr_1$ on $\gambleson{1}$ and a linear prevision $\pr_2$ on $\gambleson{2}$ such that $\pr_1(\ind{x_1})=0$ and $\pr_2(\ind{x_2})=0$ for all $x_1\in\states_1$ and $x_2\in\states_2$. 
Let $P_{12}$ be a linear prevision on $\G(\states_1\times\states_2)$ that has $\pr_1$ and $\pr_2$ as its marginals. For atom-independence, we then have that for all $f\in\G(\states_1\times\states_2)$:
\begin{equation*}
(\pr_1\underline{\otimes}\pr_2)(f)\leq P_{12}(f)\leq(\pr_1\overline{\otimes}\pr_2)(f).
\end{equation*}
\end{proposition}

We now use this result to obtain the following example. It demonstrates that $(\pr_1\overline{\otimes}\pr_2)(f\vert B)$ and $(\pr_1\underline{\otimes}\pr_2)(f\vert B)$ may be---substantially---different, and therefore, that $\mathbb{P}_{P_1\otimes P_2}$ may have more than one element. Furthermore, and perhaps most importantly, it proves the claim that we made at the end of Example~\ref{ex:oddandeven:factorisation}.

\begin{example}\label{ex:oddandeven:lower}
Let $\states_1=\states_2=\nats$, and let $g=\ind{\mathrm{odd}}\in\mathcal{G}_{\geq0}(\states_1)$ and $h=\ind{\mathrm{even}}\in\mathcal{G}_{\geq0}(\states_2)$, with $\ind{\mathrm{odd}}$ and $\ind{\mathrm{even}}$ defined as in Example~\ref{ex:oddandeven}. Consider now two linear previsions $\pr_{\mathrm{odd}}$ and $\pr_{\mathrm{even}}$ on $\G(\nats)$ such that $\pr_{\mathrm{odd}}(\ind{n})=0$ and $\pr_{\mathrm{even}}(\ind{n})=0$ for all $n\in\nats$ and such that $\pr_{\mathrm{odd}}(\ind{\mathrm{odd}})=1=\pr_{\mathrm{even}}(\ind{\mathrm{even}})$, and let $P\coloneqq\nicefrac{1}{2}(P_{\mathrm{odd}}+P_{\mathrm{even}})$. It then follows from Corollary~\ref{corol:fulllinearPiffthreeaxioms} that $P$ is a linear prevision on $\G(\nats)$. Furthermore, we find that $P(\ind{n})=\nicefrac{1}{2}(P_{\mathrm{odd}}(\ind{n})+P_{\mathrm{even}}(\ind{n}))=0$ for all $n\in\nats$ and, in the same way, that $P(\ind{\mathrm{odd}})=P(\ind{\mathrm{even}})=\nicefrac{1}{2}$. 

Now let $P_1=P_2=P$. Then as demonstrated in~\cite[Example 5]{Miranda2015460}, it is possible to construct a linear prevision $\pr_{12}$ on $\G(\states_1\times\states_2)$ that has $P_1$ and $P_2$ as its marginals and for which $\pr_{12}(X_1\text{ odd and }X_2\text{ even})=\pr_{12}(gh)=0$. Therefore, for atom-independence, Proposition~\ref{prop:dominatedbysingletonzero} implies that $(\pr_1\underline{\otimes}\pr_2)(gh)\leq0$. Since the converse inequality follows from Property~\ref{def:prev:bounded} and the non-negativity of $fg$, this implies that for atom-independence, $(\pr_1\underline{\otimes}\pr_2)(gh)=0$. Since $\pr_1(g)\pr_2(h)=\pr(\ind{\mathrm{odd}})\pr(\ind{\mathrm{even}})=\nicefrac{1}{2}\cdot\nicefrac{1}{2}=\nicefrac{1}{4}$, this shows that atom-independence does not lead to factorisation here, and proves our claim from Example~\ref{ex:oddandeven:factorisation}.

For event-independence however, since linear previsions are a special case of coherent lower previsions, we already know from Example~\ref{ex:oddandeven:factorisation} that we do have factorisation here, and therefore, that $(\pr_1\underline{\otimes}\pr_2)(gh)=\nicefrac{1}{4}$.

Finally, we note that it follows from Property~\ref{def:lowerprev:lowerbelowupper} that, regardless of whether we adopt epistemic subset- or atom-indepence, $(\pr_1\underline{\otimes}\pr_2)(gh)\leq(\pr_1\overline{\otimes}\pr_2)(gh)$. Therefore, and because we know from Section~\ref{sec:choiceofevents} that epistemic event-independence leads to more informative joint models than epistemic atom-independence, our results in this example imply that for epistemic atom-independence: 
\begin{equation*}
(\pr_1\overline{\otimes}\pr_2)(gh)\geq\nicefrac{1}{4}>0=(\pr_1\underline{\otimes}\pr_2)(gh).
\vspace{-8pt}
\end{equation*}
\hfill$\lozenge$
\end{example}

As explained in Section~\ref{sec:factadd}, the failure of factorisation that we observe in this example is a first important reason why we prefer epistemic event-independence over epistemic atom-independence. The second reason is that epistemic atom-independence leads to an independent natural extension that may be too uninformative, in the sense that $P_1\underline{\otimes}P_2$ can be excessively small. This too was illustrated in the example above: for the same local models, epistemic event-independence gave rise to substantially higher joint lower previsions.

That said, even for event-independence, $(\pr_1\overline{\otimes}\pr_2)(f\vert B)$ and $(\pr_1\underline{\otimes}\pr_2)(f\vert B)$ may still be different. In that case, however, the reason is more subtle, and can be partially attributed to the fact that for infinite spaces and finitely additive probability measures, Fubini's theorem may not hold. We demonstrate this in our final example, which relies heavily on the following proposition.

\begin{proposition}\label{prop:dominatedbynested}
Consider a linear prevision $\pr_1$ on $\gambleson{1}$ and a linear prevision $\pr_2$ on $\gambleson{2}$. Then for all $f\in\G(\states_1\times\states_2)$, we have that
\begin{equation*}
(\pr_1\underline{\otimes}\pr_2)(f)\leq\pr_1(\pr_2(f))\leq(\pr_1\overline{\otimes}\pr_2)(f),
\end{equation*}
with $P_2(f)$ a gamble on $\states_1$, defined by $P_2(f)(x_1)\coloneqq P_2(f(x_1,X_2))$\, for all $x_1\in\states_1$.
\end{proposition}

\begin{example}\label{ex:noproduct}
Consider a linear prevision $\pr_1$ on $\gambleson{1}$, a linear prevision $\pr_2$ on $\gambleson{2}$ and any $f\in\G(\states_1\times\states_2)$ such that $\pr_1(\pr_2(f))\neq\pr_2(\pr_1(f))$, where $P_2(f)$ is defined as in Proposition~\ref{prop:dominatedbynested} and where, similarly, $P_1(f)$ is a gamble on $\states_2$ that is defined by $P_1(f)(x_2)\coloneqq P_1(f(X_1,x_2))$\, for all $x_2\in\states_2$. It then follows from Proposition~\ref{prop:dominatedbynested} and symmetry---reversing the role of $P_1$ and $P_2$---that
\begin{equation*}
(\pr_1\underline{\otimes}\pr_2)(f)
\leq
\min\{\pr_1(\pr_2(f)),\pr_2(\pr_1(f))\}
<\max\{\pr_1(\pr_2(f)),\pr_2(\pr_1(f))\}
\leq(\pr_1\overline{\otimes}\pr_2)(f).
\end{equation*}
Concrete examples where this situation occurs---that is, where $\pr_1(\pr_2(f))\neq\pr_2(\pr_1(f))$---can be found in \cite[Example 1]{Miranda2015460} and~\cite[Example 4]{Schervisch:2017}.
\hfill$\lozenge$
\end{example}

\section{Conclusions and Future Work}\label{sec:conclusions}

The main conclusion of this work is that by combining Williams-coherence with epistemic event-independence, we obtain a notion of independent natural extension that always exists, and that furthermore satisfies factorisation and external additivity. For weaker types of epistemic independence, including epistemic atom-independence, the existence result and the external additivity property remain valid, but factorisation then requires measurability conditions and the resulting inferences become less informative. For that reason, I think that when it comes to choosing between epistemic event-independence and epistemic atom-independence, the former should be preferred over the latter. In fact, I would advocate that from now on, and contrary to the current convention, epistemic independence should be taken to mean epistemic event-independence.

As far as future research is concerned, a first important step would be to extend our results from the case of two variables to that of any finite number of variables. Based on our own preliminary exploration of this topic, we expect that our proofs can be easily extended to that case. However, care will have to be taken when considering concepts such as independence, factorisation and external additivity, because for multiple variables, these have several variations; for finite state spaces,~\citep{deCooman:2011ey} provides an excellent starting point.

Next, these extended versions of our results could then be used to develop efficient algorithms for credal networks whose variables take values in infinite spaces, by suitably adapting existing algorithms for the finite case. See for example the work by~\cite{deCooman:2010gd},~\cite{DeBock:2014ts} and~\cite{de2015credal}.

On the more technical side, it would be useful to see whether our results can be extended from gambles---which are taken to be bounded---to the more general case of unbounded functions; in that case, establishing factorisation could prove to be tricky, because our proof for Corollary~\ref{corol:fact-measurable} relies rather heavily on the fact that gambles are bounded.

Finally, for variables that take values in Euclidean space, I would suggest to take a closer look at the case where $\B_1$ and $\B_2$ are restricted to the Lebesgue measurable events. Combined with a suitably chosen assessment of continuity, I think that this might lead to the development of a notion of independent natural extension that includes sigma additive product measures as a special case.

\acks{Much of the research that lead to this paper was conducted during a ten month research visit to the Imprecise Probability Group of IDSIA (Institute Dalle Molle for Artificial Intelligence), the members of which I would like to thank for their warm hospitality.
This visit was funded by the Research Foundation -- Flanders (FWO).

I would also like to thank a number of people for their feedback on previous versions. First, several anonymous reviewers, for their generous constructive comments on preliminary versions of this manuscript and the conference paper from which it grew \citep{debock2017williamsPMLR}. Second, Enrique Miranda, for commenting on an even more preliminary version, and for suggesting the idea of adopting a general notion of epistemic independence where $\B_1$ and $\B_2$ are allowed to be arbitrary. Third, Gert de Cooman, for his generous feedback on many of the things I do, including this work. And last but not least, I would like to thank Teddy Seidenfeld. The day before my presentation of this work at ISIPTA'17, I received an email from him in which he questioned the validity of my results. In particular, he presented a counterexample to Fubini and explained how this made him sceptical about my claim that the independent natural extension always exists, because his counterexample to Fubini implied that the independent natural extension of two precise models may not be fully precise itself. Fortunately---although I must admit that the email did scare me---my results withstood the test. The solution to this apparent conflict turned out to be, quite simply, that we were both correct. On the one hand, my notion of independent natural extension does indeed always exist. On the other hand, as I now explain in Section~\ref{sec:specialcases}, the independent natural extension of two precise models may indeed not be precise itself. My decision to include Section~\ref{sec:specialcases}---and Example~\ref{ex:noproduct} in particular---is a direct consequence of his email, and the results in that section are inspired by the insights that he shared with me during several follow up conversations. For both of these, I'd like to sincerely thank him.
}

\bibliography{general}

\appendix

\section{Proofs and Additional Material}\label{sec:proofs}

In order to avoid forward referencing---and the associated risk of circular reasoning---the ordering of the proofs in this appendix sometimes differs from the order in which the corresponding results appear in the main text. Most importantly, the \iftoggle{arxiv}{proofs and}{} additional material for Section~\ref{sec:probs} \iftoggle{arxiv}{are}{is} presented immediately after \iftoggle{arxiv}{those}{that} of Section~\ref{sec:modellinguncertainty}. Furthermore, since the proof of Proposition~\ref{prop:measurability:equivalenceforsigmafield} relies on Proposition~\ref{prop:measurable:sufficient:general}, the order of the proofs of these two results is reversed, and similarly for Propositions~\ref{prop:lowerenvelope} and~\ref{prop:lowerenvelope:onlyif:natex}.

\subsection{Proofs and Additional Material for Section~\ref{sec:prelim}}
\vspace{5pt}

\begin{proof}{\bf of Proposition~\ref{prop:measurable:sufficient:general}~}
Since $g\geq0$ is a gamble and therefore by definition bounded, there is some $\alpha\in\rationals_{>0}$ such that $0\leq g<\alpha$. Fix any $n\in\nats$ and let $g_n\in\gambles$ be defined by
\begin{equation*}
g_n\coloneqq\frac{1}{n}\alpha\sum_{k=1}^{n-1}\ind{A_k},
~\text{where, for all $k\in\{1,\dots,n-1\}$, }A_k\coloneqq\Big\{x\in\states\colon g(x)\geq\frac{k}{n}\alpha\Big\}.
\vspace{2pt}
\end{equation*}
For all $x\in\states$, we then find that
\vspace{3pt}
\begin{equation*}
g_n(x)
=
\frac{k_x}{n}\alpha
\leq
g(x)
\leq
\frac{k_x+1}{n}\alpha,
\text{~where we let~}
k_x\coloneqq\max\{k\in\{0,\dots,n-1\}\colon g(x)\geq\frac{k}{n}\alpha\},
\vspace{2pt}
\end{equation*}
which implies that $\abs{g(x)-g_n(x)}\leq\nicefrac{\alpha}{n}$. Since this is true for every $x\in\states$, this allows us to infer that $\sup\abs{g-g_n}\leq\nicefrac{\alpha}{n}$.

Consider now any $k\in\{1,\dots,n-1\}$. Since $\nicefrac{k}{n}\alpha\in\rationals_{\geq0}$, it follows from our assumptions on $g$ that $A_k$ is a finite union of pairwise disjoint events in $\B\cup\{\states,\emptyset\}$. Therefore, there is some $m_k\in\nats$ and, for all $i\in\{1,\dots,m_k\}$, some $B_{k,i}\in\B\cup\{\states,\emptyset\}$ such that $\ind{A_k}=\sum_{i=1}^{m_k}\ind{B_{k,i}}$. Since this is true for every $k\in\{1,\dots,n-1\}$, it follows that $g_n=\nicefrac{\alpha}{n}\sum_{k=1}^{n-1}\sum_{i=1}^{m_k}\ind{B_{k,i}}$. Since $g_n$ is clearly non-negative, and because $\ind{\states}=1$ and $\ind{\emptyset}=0$, it now follows from Definition~\ref{def:measurable:simple} that $g_n\in\gamblesnonneg$ is a simple $\B$-measurable gamble.

So, in summary then, for any fixed $n\in\nats$, we know that we can construct a simple $\B$-measurable gamble $g_n\in\gamblesnonneg$ such that $\sup\abs{g-g_n}\leq\nicefrac{\alpha}{n}$. Definition~\ref{def:measurable:uniform} therefore clearly implies that $g$ is $\B$-measurable.
\end{proof}
\vspace{-6pt}

\begin{proof}{\bf of Proposition~\ref{prop:measurability:equivalenceforsigmafield}~}
Consider any $\B\subseteq\nonemptypower$ such that $\B^*\coloneqq\B\cup\{\emptyset\}$ is a sigma field and fix some $g\in\gamblesnonneg$.

We first prove the `only if' part of the statement. So assume that $g$ is $\B^*$-measurable in the measure-theoretic sense~\cite[Definition~10.1]{Nielsen1997}. It then follows from~\cite[Corollary~10.5]{Nielsen1997} that $\{x\in\states\colon g(x)\geq r\}\in\B^*=\B\cup\{\emptyset\}$ for all $r\in\rationals_{\geq0}$. Therefore, it follows from Proposition~\ref{prop:measurable:sufficient:general} that $g$ is $\B$-measurable in the sense of Definition~\ref{def:measurable:uniform}.

We end by proving the `if' part of the statement. So assume that $g$ is $\B$-measurable in the sense of Definition~\ref{def:measurable:uniform}. This means that there is a sequence $\{g_n\}_{n\in\nats}$ of simple $\B$-measurable gambles in $\mathcal{G}_{\geq0}(\states)$ such that $\lim_{n\to+\infty}\sup\abs{g-g_n}=0$. Then on the one hand, since $\lim_{n\to+\infty}\sup\abs{g-g_n}=0$ implies that $\lim_{n\to+\infty}\abs{g(x)-g_n(x)}=0$ for all $x\in\states$, we know that $\{g_n\}_{n\in\nats}$ converges pointwise to $g$ on $\states$. On the other hand, for any $n\in\nats$, we know from Definition~\ref{def:measurable:simple} that there are $c_0\in\reals_{\geq0}$, $m\in\natswith$ and, for all $i\in\{1,\dots,m\}$, $c_i\in\reals_{\geq0}$ and $B_i\in\B$, such that $g=c_0+\sum_{i=1}^mc_i\ind{B_i}$. Let $B_0=\states$. Since $\ind{\states}=1$, and because $\B^*$ is a sigma field and therefore includes $\states$, we then find that $g=\sum_{i=0}^mc_i\ind{B_i}$, where, for all $i\in\{0,\dots,n\}$, $B_i\in\B^*$. \cite[Example~10.2]{Nielsen1997} therefore implies that $g_n$ is a $\B^*$-measurable function in the measure-theoretic sense. Since this is true for every $n\in\nats$, and because $\{g_n\}_{n\in\nats}$ converges pointwise to $g$ on $\states$, it now follows from~\cite[Corollary~10.11(a)]{Nielsen1997} that $g$ is $\B^*$-measurable in the measure-theoretic sense.
\end{proof}
\vspace{-6pt}

\begin{proof}{\bf of Corollary~\ref{corol:measurable:sufficient:allsets}~}
Immediate consequence of Proposition~\ref{prop:measurable:sufficient:general}.
\end{proof}

\vspace{-16pt}

\subsection{\iftoggle{arxiv}{Proofs and }{}Additional Material for Section~\ref{sec:modellinguncertainty}}\label{app:modellinguncertainty}
\vspace{5pt}

\iftoggle{arxiv}{
Contrary to what the length of this section of the appendix might suggest, it should be noted that many of the results and proofs in this section are essentially well-known.}{Many of the results in Section~\ref{sec:modellinguncertainty} and this corresponding part of the appendix are essentially well-known.} Historically, most of them date back to~\cite{williams1975,Williams:2007eu}. Our versions are basically just minor variations of his results, expressed in terms of lower previsions---instead of upper previsions---and without imposing structural constraints on the domain. Similar results can also be found in~\citep{Pelessoni:2009co}\iftoggle{arxiv}{---although often without or with only a minimal proof---}{ }and, for the case of Walley-coherence, in \citep{Miranda2010}. \iftoggle{arxiv}{}{For that reason, and in order not to dilute the core novel part of our work, we state these results without proof. For those interested, explicit proofs are available in an online arXiv version of this contribution~\citep{DeBock2018Williams}.}

\begin{lemma}\label{lemma:coherenceiffLP4}
For any $\A\subseteq\gambles$, $\E(\A)$ is a coherent set of desirable gambles on $\states$ if and only if it satisfies~\ref{def:SDG:partialloss}.
\end{lemma}

\iftoggle{arxiv}{
\begin{proof}{\bf of Lemma~\ref{lemma:coherenceiffLP4}~}
Since Equation~\eqref{eq:natextop} implies that $\E(\A)$ satisfies~\ref{def:SDG:partialgain},~\ref{def:SDG:homo} and~\ref{def:SDG:convex}, this follows trivially from Definition~\ref{def:SDG}.
\end{proof}
\vspace{-16pt}
}{\vspace{-6pt}}

\begin{lemma}\label{lemma:nestedpropsofposandE}
Let $\A_1$ and $\A_2$ be two subsets of $\gambles$ such that $\A_1\subseteq\A_2$. Then
\vspace{4pt}
\begin{equation*}
\posi(\A_1)\subseteq\posi(\A_2)
~\text{ and }~
\E(\A_1)\subseteq\E(\A_2).
\end{equation*}
\end{lemma}
\iftoggle{arxiv}{
\begin{proof}{\bf of Lemma~\ref{lemma:nestedpropsofposandE}~}
This follows trivially from Equations~\eqref{eq:posi} and~\eqref{eq:natextop}.
\end{proof}
\vspace{-16pt}
}{\vspace{-6pt}}

\begin{lemma}\label{lemma:natextDisD}
Let $\desir$ be a coherent set of desirable gambles on $\states$. Then $\E(\desir)=\desir$.
\end{lemma}
\iftoggle{arxiv}{
\begin{proof}{\bf of Lemma~\ref{lemma:natextDisD}~}
$\desir$ is trivially a subset of $\E(\desir)$. The converse inclusion, that is, $\E(\desir)\subseteq\desir$, is a straightforward consequence of the coherence of $\desir$.
\end{proof}
\vspace{-16pt}
}{\vspace{-6pt}}

\begin{lemma}\label{lemma:extendD}
Let $\desir$ be a coherent set of desirable gambles on $\states$. If $f\in\gambles$ and $f\notin\desir\cup\{0\}$, then $\E(\desir\cup\{-f\})$ is a coherent set of desirable gambles on $\states$.
\end{lemma}
\iftoggle{arxiv}{
\begin{proof}{\bf of Lemma~\ref{lemma:extendD}~}
Consider any $f\in\gambles$ such that $f\notin\desir\cup\{0\}$. Because of Lemma~\ref{lemma:coherenceiffLP4}, it suffices to prove that $\E(\desir\cup\{-f\})$ satisfies \ref{def:SDG:partialloss}. So consider any $g\in\gambles$ such that $g\leq0$. In the remainder of this proof, we show that $g\notin\E(\desir\cup\{-f\})$. 

Assume \emph{ex absurdo} that $g\in\E(\desir\cup\{-f\})$. Since $\desir$ is coherent, this implies that $g=\lambda h-\mu f$, with $h\in\desir$, $\lambda,\mu\in\realsnonneg$ and $\lambda+\mu>0$. If $\mu=0$, then because $h\in\desir$, the coherence of $\desir$ implies that $g=\lambda h\in\desir$, which implies that $\desir$ does not satisfy \ref{def:SDG:partialloss}, a contradiction. Hence, it follows that $\mu>0$, which implies that $f=\nicefrac{1}{\mu}(\lambda h-g)$. Therefore, since $h\in\desir$ and $-g\geq0$, it follows from the coherence of $\desir$ that $f=0$ (if $\lambda=0$ and $g=0$) or $f\in\desir$. In both cases, we contradict our assumptions.
\end{proof}
\vspace{-6pt}
}{}

\iftoggle{arxiv}{
\begin{proof}{\bf of Proposition~\ref{prop:equivalentToPelessoniAndVicig}~} Consider any conditional lower prevision $\lp$ on $\C\subseteq\C(\states)$. 

We start by proving the `only if' part of the statement. So let us assume that $\lp$ is coherent. According to Definition~\ref{def:cohlp}, this implies that there is a coherent set of desirable gambles $\desir$ on $\states$ such that $\lp_\desir$ coincides with $\lp$ on $\C$. We need prove that $\lp$ is real-valued and that it satisfies Equation~\eqref{eq:prop:equivalentToPelessoniAndVicig}.

We begin by establishing that $\lp$ is real-valued. So fix any $(f,B)\in\C$. For all $\mu\in\reals$ such that $\mu<\inf f$, it then follows from the coherence of $\desir$---and~\ref{def:SDG:partialgain} in particular---that $[f-\mu]\ind{B}\in\desir$.  Similarly, for all $\mu\in\reals$ such that $\mu>\sup f$, it follows from~\ref{def:SDG:partialloss} that $[f-\mu]\ind{B}\notin\desir$. Hence, we find that $\inf f\leq\lp_\desir(f\vert B)\leq\sup f$. Since $f$ is a gamble and therefore by definition bounded, this implies that $\lp_\desir(f\vert B)$ is real-valued, which in turn implies that $\lp(f\vert B)$ is real-valued because $\lp_\desir$ coincides with $\lp$ on $\C$. Since $(f,B)\in\C$ was arbitrary, this means that $\lp$ is real-valued. 

Next, we show that $\lp$ satisfies Equation~\eqref{eq:prop:equivalentToPelessoniAndVicig}. Fix any $n\in\natswith$, choose any $\lambda_0,\dots,\lambda_n\in\reals_{\geq0}$ and $(f_0,B_0),\dots,(f_n,B_n)\in\mathcal{C}$, let $B\coloneqq\cup_{i=0}^nB_i$ and let $h\in\gambles$ be defined by
\begin{equation*}
h(x)
\coloneqq\sum_{i=1}^n
\lambda_i\ind{B_i}(x)
[f_i(x)-\lp(f_i\vert B_i)]
-\lambda_0\ind{B_0}(x)
[f_0(x)-\lp(f_0\vert B_0)]
~~\text{for all $x\in\states$}.
\end{equation*}
We need to prove that $\sup_{x\in B}h(x)\geq0$. In order to do that, we start by fixing some $\epsilon>0$. Let $\epsilon_0\coloneqq\epsilon$. Since $\lp_\desir$ coincides with $\lp$ on $\C$, it then follows from Equation~\eqref{eq:LPfromD} that
\begin{equation}\label{eq:prop:equivalentToPelessoniAndVicig:proof:1}
g_0\coloneqq[f_0-\lp(f_0\vert B_0)-\epsilon_0]\ind{B_0}=[f_0-\lp_\desir(f_0\vert B_0)-\epsilon_0]\ind{B_0}\notin\desir.
\end{equation}
Similarly, for all $i\in\{1,\dots,n\}$, Equation~\eqref{eq:LPfromD} implies that there is some $\epsilon_i\geq0$ such that $\epsilon_i\leq\epsilon$ and
\begin{equation}\label{eq:prop:equivalentToPelessoniAndVicig:proof:2}
g_i\coloneqq[f_i-\lp(f_i\vert B_i)+\epsilon_i]\ind{B_i}=[f_i-\lp_\desir(f_i\vert B_i)+\epsilon_i]\ind{B_i}\in\desir.
\end{equation}
Now let $g\coloneqq\lambda_0g_0-\sum_{i=1}^n\lambda_ig_i$ and assume \emph{ex absurdo} that $g\in\desir$. Since $\desir$ is coherent and therefore satisfies~\ref{def:SDG:homo} and~\ref{def:SDG:convex}, it then follows from Equation~\eqref{eq:prop:equivalentToPelessoniAndVicig:proof:2} that
\begin{equation*}
\lambda_0g_0
=
g+\sum_{i=1}^n\lambda_ig_i
=
g+\sum_{\substack{i=1\\\lambda_i\neq0}}^n\lambda_ig_i\in\desir.
\end{equation*}
If $\lambda_0=0$, this implies that $0\in\desir$, which contradicts~\ref{def:SDG:partialloss}. If $\lambda_0>0$, this implies that $g_0\in\desir$ because of~\ref{def:SDG:homo}, which contradicts Equation~\eqref{eq:prop:equivalentToPelessoniAndVicig:proof:1}. Since both cases lead to a contradiction, we conclude that $g\notin\desir$. Since the coherence of $\desir$ implies that $\gamblespos\subseteq\desir$, this allows us to infer that $g\notin\gamblespos$. 
Since $g(x)=0$ for all $x\in\states\setminus B$, this implies that $\inf_{x\in B}g(x)\leq0$. Hence, we find that
\begin{align*}
0\leq
-\inf_{x\in B}g(x)
=\sup_{x\in B}-g(x)
&=\sup_{x\in B}\big(h(x)+\sum_{i=0}^n\lambda_i\epsilon_i\ind{B_i}(x)\big)\\
&\leq\sup_{x\in B}h(x)+\sup_{x\in B}\big(\sum_{i=0}^n\lambda_i\epsilon_i\ind{B_i}(x)\big)
\leq\sup_{x\in B}h(x)+\sum_{i=0}^n\lambda_i\epsilon_i,\vspace{-3pt}
\end{align*}
which implies that
\vspace{-3pt}
\begin{equation*}
\sup_{x\in B}h(x)\geq-\sum_{i=0}^n\lambda_i\epsilon_i\geq-\epsilon\sum_{i=0}^n\lambda_i.
\vspace{8pt}
\end{equation*}
Since this is true for every $\epsilon>0$, it follows that $\sup_{x\in B}h(x)\geq0$, as desired.

It remains to prove the `if' part of the statement. So let us assume that $\lp$ is real-valued and that it satisfies Equation~\eqref{eq:prop:equivalentToPelessoniAndVicig}. We need to prove that $\lp$ is coherent.

Let $\A_{\lp}$ and $\E(\lp)$ be defined by Equation~\eqref{eq:AfromLP}. We start by proving that $\E(\lp)$ is a coherent set of desirable gambles on $\states$. Fix any $f\in\E(\lp)$. We then know from Equations~\eqref{eq:posi},~\eqref{eq:natextop} and~\eqref{eq:AfromLP} that\vspace{-4pt}
\begin{equation}
f=\sum_{i=1}^n\lambda_i\ind{B_i}[f_i-\mu_i]+\sum_{j=n+1}^m\lambda_jf_j\geq\sum_{i=1}^n\lambda_i\ind{B_i}[f_i-\mu_i]
\label{eq:prop:equivalentToPelessoniAndVicig:proof:3}
\end{equation}
for some $n\in\natswith$ and $m\in\nats$ such that $n\leq m$, with $\lambda_1,\dots,\lambda_m\in\reals_{>0}$, $f_{n+1},\dots,f_m\in\gamblespos$, $(f_1,B_1),\dots,(f_n,B_n)\in\C$ and $\mu_1,\dots,\mu_n\in\reals$ such that $\mu_i<\lp(f_i\vert B_i)$ for all $i\in\{1,\dots,n\}$. 
We consider two cases: $f\in\gamblespos$ and $f\notin\gamblespos$. 
If $f\in\gamblespos$, then $f\not\leq0$. If $f\notin\gamblespos$, then $n\neq0$. Therefore, if we let $A\coloneqq\cup_{i=1}^nB_i\neq\emptyset$, it follows from Equations~\eqref{eq:prop:equivalentToPelessoniAndVicig:proof:3} and~\eqref{eq:prop:equivalentToPelessoniAndVicig} that
\begin{align*}
\sup_{x\in A}f(x)
&\geq\sup_{x\in A}\Big(\sum_{i=1}^n\lambda_i\ind{B_i}(x)[f_i(x)-\mu_i]\Big)\\
&=\sup_{x\in A}\Big(\sum_{i=1}^n\lambda_i\ind{B_i}(x)[f_i(x)-\lp(f_i\vert B_i)]+\sum_{i=1}^n\lambda_i\ind{B_i}(x)[\lp(f_i\vert B_i)-\mu_i]\Big)\\
&\geq\sup_{x\in A}\Big(\sum_{i=1}^n\lambda_i\ind{B_i}(x)[f_i(x)-\lp(f_i\vert B_i)]\Big)+\inf_{x\in A}\Big(\sum_{i=1}^n\lambda_i\ind{B_i}(x)[\lp(f_i\vert B_i)-\mu_i]\Big)\\
&\geq\inf_{x\in A}\Big(\sum_{i=1}^n\lambda_i\ind{B_i}(x)[\lp(f_i\vert B_i)-\mu_i]\Big)
\geq\min_{1\leq i\leq n}\lambda_i[\lp(f_i\vert B_i)-\mu_i]>0,\\[-13pt]
\end{align*}
which implies that $f\not\leq0$. Hence, in both cases, we find that $f\not\leq0$. Since $f\in\E(\lp)$ is arbitrary, this implies that $\E(\lp)$ satisfies~\ref{def:SDG:partialloss}. Since $\E(\lp)\coloneqq\E(\A_{\lp})$, it now follows from Lemma~\ref{lemma:coherenceiffLP4} that $\E(\lp)$ is a coherent set of desirable gambles on $\states$.

In the remainder of this proof, we will show that $\lp_{\E(\lp)}$ coincides with $\lp$ on $\C$. Since $\E(\lp)$ is a coherent set of desirable gambles on $\states$, Definition~\ref{def:cohlp} then implies that $\lp$ is coherent, as desired. So fix any $(f,B)\in\C$. We need to prove that $\lp(f\vert B)=\lp_{\E(\lp)}(f\vert B)$. However, since Equation~\eqref{eq:AfromLP} implies that $[f-\mu]\ind{B}\in\E(\lp)$ for all $\mu<\lp(f\vert B)$, it follows trivially from Equation~\eqref{eq:LPfromD} that $\lp(f\vert B)\leq\lp_{\E(\lp)}(f\vert B)$. Therefore, it remains to prove that $\lp(f\vert B)\geq\lp_{\E(\lp)}(f\vert B)$.

Consider any $\mu\in\reals$ such that $[f-\mu]\ind{B}\in\E(\lp)$. We then know from Equations~\eqref{eq:posi},~\eqref{eq:natextop} and~\eqref{eq:AfromLP} that\vspace{-4pt}
\begin{equation}
[f-\mu]\ind{B}=\sum_{i=1}^n\lambda_i\ind{B_i}[f_i-\mu_i]+\sum_{j=n+1}^m\lambda_jf_j\geq\sum_{i=1}^n\lambda_i\ind{B_i}[f_i-\mu_i]
\label{eq:prop:equivalentToPelessoniAndVicig:proof:4}
\end{equation}
for some $n\in\natswith$ and $m\in\nats$ such that $n\leq m$, with $\lambda_1,\dots,\lambda_m\in\reals_{>0}$, $f_{n+1},\dots,f_m\in\gamblespos$, and $(f_1,B_1),\dots,(f_n,B_n)\in\C$ and $\mu_1,\dots,\mu_n\in\reals$ such that $\mu_i<\lp(f_i\vert B_i)$ for all $i\in\{1,\dots,n\}$. 
Therefore, if we let $A\coloneqq B\cup\big(\cup_{i=1}^nB_i\big)\neq\emptyset$, we find that
\begin{align}
&\sup_{x\in A}\Big(\sum_{i=1}^n\lambda_i\ind{B_i}(x)[\mu_i-\lp(f_i\vert B_i)]-\ind{B}(x)[\mu-\lp(f\vert B)]\Big)\notag\\
&\geq
\sup_{x\in A}\Big(\sum_{i=1}^n\lambda_i\ind{B_i}(x)[\mu_i-\lp(f_i\vert B_i)]-\ind{B}(x)[\mu-\lp(f\vert B)]
+
\sum_{i=1}^n\lambda_i\ind{B_i}(x)[f_i(x)-\mu_i]-\ind{B}(x)[f(x)-\mu]
\Big)\notag\\
&=
\sup_{x\in A}\Big(\sum_{i=1}^n\lambda_i\ind{B_i}(x)[f_i(x)-\lp(f_i\vert B_i)]
-\ind{B}(x)[f(x)-\lp(f\vert B)]\Big)\geq0\label{eq:prop:equivalentToPelessoniAndVicig:proof:5}\\[-14pt]\notag
\end{align}
where the first inequality follows from Equation~\eqref{eq:prop:equivalentToPelessoniAndVicig:proof:4} and the last inequality follows from Equation~\eqref{eq:prop:equivalentToPelessoniAndVicig}. 
Since $\lambda_i>0$ and $\mu_i-\lp(f_i\vert B_i)<0$, this implies that $\mu\leq\lp(f\vert B)$. Since this true for every $\mu\in\reals$ such that $[f-\mu]\ind{B}\in\E(\lp)$, it follows from Equation~\eqref{eq:LPfromD} that $\lp_{\E(\lp)}(f\vert B)\leq\lp(f\vert B)$.
\end{proof}
}{}

\iftoggle{arxiv}{
\begin{proof}{\bf of Proposition~\ref{prop:smallestSDGfromLP}~}
Consider any coherent set of desirable gambles $\desir$ on $\states$ such that $\lp_\desir$ coincides with $\lp$ on $\C$. Since $\lp$ is coherent, we know from Definition~\ref{def:cohlp} that there is at least one such set $\desir$. We start by proving that $\E(\lp)\subseteq\desir$.

Fix any $(f,B)\in\C$ and any $\mu<\lp(f\vert B)$. Since $\lp_\desir(f\vert B)=\lp(f\vert B)$, we know that $\mu<\lp_\desir(f\vert B)$, and therefore, it follows from Equation~\eqref{eq:LPfromD} that there is some $\mu^*\in\reals$ such that $[f-\mu^*]\ind{B}\in\desir$ and $\mu<\mu^*\leq\lp_\desir(f\vert B)$. Furthermore, since $\mu^*>\mu$ and $B\neq\emptyset$, we also know that $[\mu^*-\mu]\ind{B}\in\mathcal{G}_{>0}(\states)$, which implies that $[\mu^*-\mu]\ind{B}\in\desir$ because of~\ref{def:SDG:partialgain}. Since $[f-\mu^*]\ind{B}\in\desir$ and $[\mu^*-\mu]\ind{B}\in\desir$, it now follows from~\ref{def:SDG:convex} that $[f-\mu]\ind{B}=[f-\mu^*]\ind{B}+[\mu^*-\mu]\ind{B}\in\desir$. Since this is true for every $(f,B)\in\C$ and $\mu<\lp(f\vert B)$, we infer that $\A_{\lp}\subseteq\desir$, and therefore, because of Lemmas~\ref{lemma:nestedpropsofposandE} and~\ref{lemma:natextDisD}, that $\E(\lp)=\E(\A_{\lp})\subseteq\E(\desir)=\desir$.

Next, since $\desir$ is coherent and $\E(\lp)\subseteq\desir$, it follows from Definition~\ref{def:SDG} that $\E(\lp)$ satisfies~\ref{def:SDG:partialloss}. Therefore, and because $\E(\lp)=\E(\A_{\lp})$, it follows from Lemma~\ref{lemma:coherenceiffLP4} that $\E(\lp)$ is a coherent set of desirable gambles on $\states$. Hence, it remains to prove that $\lp_{\E(\lp)}$ coincides with $\lp$ on $\C$.

Fix any $(f,B)\in\C$. Then on the one hand, since $\E(\lp)\subseteq\desir$, we have that
\begin{equation*}
\lp_{\E(\lp)}(f\vert B)\leq\lp_{\desir}(f\vert B)=\lp(f\vert B).
\end{equation*}
On the other hand, since we know from Equation~\eqref{eq:AfromLP} that $[f-\mu]\ind{B}\in\E(\lp)$ for all $\mu<\lp(f\vert B)$, it follows from Equation~\eqref{eq:LPfromD} that $\lp_{\E(\lp)}(f\vert B)\geq\lp(f\vert B)$. Hence, we find that $\lp_{\E(\lp)}(f\vert B)=\lp(f\vert B)$. Since $(f,B)\in\C$ is arbitrary, this implies that $\lp_{\E(\lp)}$ coincides with $\lp$ on $\C$.
\end{proof}
\vspace{-16pt}
}{\vspace{-6pt}}

\begin{proposition}\label{prop:naturalextension}
Let $\lp$ be a coherent conditional lower prevision on $\C\subseteq\C(\states)$. Then for any $\C'\subseteq\C(\states)$ such that $\C\subseteq\C'$, the restriction of $\natexLP$ to $\C'$ is the pointwise smallest coherent conditional lower prevision on $\C'$ that coincides with $\lp$ on $\C$.
\end{proposition}

\iftoggle{arxiv}{
\begin{proof}{\bf of Proposition~\ref{prop:naturalextension}~}
Let $\lp$ be a coherent conditional lower prevision on $\C\subseteq\C(\states)$ and consider any $\C'\subseteq\C(\states)$ such that $\C\subseteq\C'$. Then as we know from Proposition~\ref{prop:smallestSDGfromLP}, $\natexLP$ is a coherent conditional lower prevision on $\C(\states)$ that coincides with $\lp$ on $\C$. Since it follows trivially from Definition~\ref{def:cohlp} that restricting the domain of a coherent conditional lower prevision preserves its coherence, this implies that the restriction of $\natexLP$ to $\C'$ is a coherent conditional lower prevision on $\C'$ that coincides with $\lp$ on $\C$. It remains to show that it is dominated by any other coherent conditional lower prevision on $\C'$ that coincides with $\lp$ on $\C$.

So consider any coherent conditional lower prevision $\lp'$ on $\C'$ that coincides with $\lp$ on $\C$. Because of Definition~\ref{def:cohlp}, this implies that there is a coherent set of desirable gambles $\desir$ on $\states$ such that $\lp_\desir$ coincides with $\lp'$ on $\C'$. Since this clearly implies that $\lp_\desir$ coincides with $\lp$ on $\C$, it now follows from Proposition~\ref{prop:smallestSDGfromLP} that $\E(\lp)\subseteq\desir$, which implies that $\natexLP=\lp_{\E(\lp)}\leq\lp_\desir$. Hence, since $\lp_\desir$ coincides with $\lp'$ on $\C'$, we find that $\natexLP$ is dominated by $\lp'$ on $\C'$, as desired.
\end{proof}
\vspace{-6pt}
}

\iftoggle{arxiv}{
\begin{proof}{\bf of Proposition~\ref{prop:naturalextension:full}~}
Immediate consequence of Proposition~\ref{prop:naturalextension}.
\end{proof}
\vspace{-16pt}
}{\vspace{-6pt}}

\begin{lemma}\label{lemma:nonnegativeELPcond}
Let $\lp$ be a coherent conditional lower prevision on $\C\subseteq\C(\states)$. Then for any $(f,B)\in\C$ such that $f\,\ind{B}\in\E(\lp)$, we have that $\lp(f\vert B)\geq0$.
\end{lemma}
\begin{proof}{\bf of Lemma~\ref{lemma:nonnegativeELPcond}}
It suffices to notice that
\begin{equation*}
\lp(f\vert B)=\natexLP(f\vert B)=\lp_{\E(\lp)}(f\vert B)
=\sup\{\mu\in\reals\colon [f-\mu]\ind{B}\in\E(\lp)\}\geq0,
\end{equation*}
where the equalities follow from Proposition~\ref{prop:naturalextension:full}, Equation~\eqref{eq:naturalextension} and Equation~\eqref{eq:LPfromD}, respectively, and where the final inequality follows from the fact that $f\,\ind{B}\in\E(\lp)$.
\end{proof}
\vspace{-16pt}

\begin{lemma}\label{lemma:nonnegativeELP}
Let $\lp$ be a coherent lower prevision on $\G\subseteq\gambles$. Then for any $f\in\G\cap\E(\lp)$, we have that $\lp(f)\geq0$.
\end{lemma}
\begin{proof}{\bf of Lemma~\ref{lemma:nonnegativeELP}}
Immediate consequence of Lemma~\ref{lemma:nonnegativeELPcond}, for $\C=\{(f,\states)\colon f\in\G\}$.
\end{proof}
\vspace{-10pt}

\iftoggle{arxiv}{
\begin{proof}{\bf of Proposition~\ref{prop:propertiesofLP}~}
Let $\lp$ be a coherent conditional lower prevision on $\C\subseteq\C(\states)$. Definition~\ref{def:cohlp} and Equation~\eqref{eq:LPfromD} then immediately imply that $\lp$ satisfies~\ref{def:lowerprev:monotonicity}. Now let $\natexLP$ be its natural extension to $\C(\states)$. We will prove that $\natexLP$ satisfies~\ref{def:lowerprev:bounded}--\ref{def:lowerprev:constantadditivity} and \ref{def:lowerprev:lowerbelowupper}. Since we know from Proposition~\ref{prop:naturalextension} that $\natexLP$ coincides with $\lp$ on $\C$, this then implies that $\lp$ satisfies~\ref{def:lowerprev:bounded}--\ref{def:lowerprev:constantadditivity} and \ref{def:lowerprev:lowerbelowupper} on its domain---that is, whenever the expressions are well-defined.

Since we know from Proposition~\ref{prop:naturalextension} that $\natexLP$ is coherent, it follows from Proposition~\ref{prop:equivalentToPelessoniAndVicig} that $\natexLP$ is real-valued and satisfies Equation~\eqref{eq:prop:equivalentToPelessoniAndVicig}, which means that it satisfies the notion of Williams coherence that is considered in~\cite{Pelessoni:2009co} and~\cite{Williams:2007eu}. It therefore follows from~\cite[Theorem~2]{Pelessoni:2009co} or~\cite[(A1*)--(A4*)]{Williams:2007eu} that $\natexLP$ satisfies~\ref{def:lowerprev:bounded}--\ref{def:lowerprev:GBR}. Consider now any $B\in\nonemptypower$. Since the operator $\natexLP(\cdot\vert B)\colon\gambles\to\reals$ satisfies~\ref{def:lowerprev:bounded}--\ref{def:lowerprev:superadditive}, it is a coherent lower prevision on $\gambles$ in the sense of~\cite[Section~2.5.5]{Walley:1991vk}. Therefore, it follows from~\cite[Section~2.6.1]{Walley:1991vk} that $\natexLP(\cdot\vert B)$ satisfies~\ref{def:lowerprev:uniformcontinuity},~\ref{def:lowerprev:constantadditivity} and~\ref{def:lowerprev:lowerbelowupper}. Since this is true for every $B\in\nonemptypower$, it follows that $\natexLP$ satisfies~\ref{def:lowerprev:uniformcontinuity},~\ref{def:lowerprev:constantadditivity} and~\ref{def:lowerprev:lowerbelowupper} as well.
\end{proof}
\vspace{-6pt}

\begin{proof}{\bf of Proposition~\ref{prop:cohlpifffouraxioms}~}
If $\lp$ is coherent, we know from Proposition~\ref{prop:equivalentToPelessoniAndVicig} that $\lp$ is real-valued and from Proposition~\ref{prop:propertiesofLP} that it satisfies~\ref{def:lowerprev:bounded}--\ref{def:lowerprev:GBR}. So assume that $\lp$ is real-valued and satisfies~\ref{def:lowerprev:bounded}--\ref{def:lowerprev:GBR}. We need to prove that $\lp$ is coherent.

Because of Proposition~\ref{prop:equivalentToPelessoniAndVicig}, it suffices to show for all $n\in\natswith$ and all choices of $\lambda_0,\dots,\lambda_n\in\reals_{\geq0}$ and $(f_0,B_0),\dots,(f_n,B_n)\in\mathcal{C}$ that
\begin{equation*}
\sup_{x\in B}
\Big(\,
\sum_{i=1}^n
\lambda_i\ind{B_i}(x)
[f_i(x)-\lp(f_i\vert B_i)]
-\lambda_0\ind{B_0}(x)
[f_0(x)-\lp(f_0\vert B_0)]
\Big)
\geq0,
\vspace{6pt}
\end{equation*}
with $B\coloneqq\cup_{i=0}^nB_i$. So let us consider any $n\in\natswith$, $\lambda_0,\dots,\lambda_n\in\reals_{\geq0}$ and $(f_0,B_0),\dots,(f_n,B_n)\in\mathcal{C}$ and let $B\coloneqq\cup_{i=0}^nB_i$. Since $B$ is a finite union of events in $\B$ and because $\B$ is closed under finite unions, we know that $B\in\B$.
 Therefore, and because $\mathcal{F}$ is a linear space such that $\ind{B}f\in\mathcal{F}$ and $\ind{B}\in\mathcal{F}$ for all $f\in\mathcal{F}$ and $B\in\B$, it now follows from~\eqref{def:lowerprev:superadditive} that
\vspace{3pt}
\begin{align*}
&\lp(\lambda_0\ind{B_0}[f_0-\lp(f_0\vert B_0)]\vert B)\\
&\geq
\lp\Big(\lambda_0\ind{B_0}[f_0-\lp(f_0\vert B_0)]-\sum_{i=1}^n\lambda_i\ind{B_i}[f_i-\lp(f_i\vert B_i)]\Big\vert B\Big)
+
\sum_{i=1}^n\lp\big(\lambda_i\ind{B_i}[f_i-\lp(f_i\vert B_i)]\big\vert B\big).
\end{align*}
Hence, since we know from~\ref{def:lowerprev:homo} and~\ref{def:lowerprev:GBR}---and our assumptions on $\mathcal{F}$ and $\B$---that
\begin{equation*}
\lp\big(\lambda_i\ind{B_i}[f_i-\lp(f_i\vert B_i)]\big\vert B\big)
=\lambda_i\lp\big(\ind{B_i}[f_i-\lp(f_i\vert B_i)]\big\vert B\big)=0
~~\text{for all $i\in\{0,\dots,n\}$,}
\end{equation*}
it follows from~\ref{def:lowerprev:bounded} that
\begin{align*}
0&\geq\lp\Big(\lambda_0\ind{B_0}[f_0-\lp(f_0\vert B_0)]-\sum_{i=1}^n\lambda_i\ind{B_i}[f_i-\lp(f_i\vert B_i)]\Big\vert B\Big)\\
&\geq\inf_{x\in B}\Big(\lambda_0\ind{B_0}(x)[f_0(x)-\lp(f_0\vert B_0)]-\sum_{i=1}^n\lambda_i\ind{B_i}(x)[f_i(x)-\lp(f_i\vert B_i)]\Big)\\
&=
-\sup_{x\in B}\Big(\sum_{i=1}^n\lambda_i\ind{B_i}(x)[f_i(x)-\lp(f_i\vert B_i)]-\lambda_0\ind{B_0}(x)[f_0(x)-\lp(f_0\vert B_0)]\Big),
\end{align*}
as desired.
\end{proof}
\vspace{-10pt}

\begin{proof}{\bf of Corollary~\ref{corol:cohlpifffouraxioms}~}
Immediate consequence of Proposition~\ref{prop:cohlpifffouraxioms}.
\end{proof}
\vspace{-10pt}

\begin{proof}{\bf of Corollary~\ref{corol:cohlpiffthreeaxioms}~}
Let $\mathcal{F}\coloneqq\G$ and $\B\coloneqq\{\states\}$. Since $\ind{\states}=1$, the conditions of Proposition~\ref{prop:cohlpifffouraxioms} are then clearly satisfied. Hence, since a lower prevision on $\G$ is by definition a conditional lower prevision on $\C\coloneqq\{(f,\states)\colon f\in\G\}$, it follows from Proposition~\ref{prop:cohlpifffouraxioms} that $\lp$ is coherent if and only if it is real-valued and satisfies~\ref{def:lowerprev:bounded}--\ref{def:lowerprev:GBR}. Therefore, the only thing that we still need to prove is that any real-valued lower prevision $\lp$ on $\G$ that satisfies \ref{def:lowerprev:bounded}--\ref{def:lowerprev:superadditive} will also satisfy \ref{def:lowerprev:GBR}.

So consider any real-valued lower prevision $\lp$ on $\G$ that satisfies \ref{def:lowerprev:bounded}--\ref{def:lowerprev:superadditive}. By definition, this operator is then a coherent lower prevision in the sense of~\citep[Section~2.5.5]{Walley:1991vk}, and it therefore follows from \cite[Section 2.6.1]{Walley:1991vk} that $\lp(f-\lp(f))=\lp(f)-\lp(f)=0$ for all $f\in\G$. Since the domain of a(n unconditional) lower prevision contains only couples of the form $(f,\states)$, this implies that $\lp$ satisfies \ref{def:lowerprev:GBR}.
\end{proof}
\vspace{-10pt}
}

\subsection{\iftoggle{arxiv}{Proofs and }Additional Material for Section~\ref{sec:probs}}
\vspace{5pt}

Here too, as in Appendix~\ref{app:modellinguncertainty}, it should be noted that \iftoggle{arxiv}{the results and proofs in this section}{the results in Section~\ref{sec:probs} and this corresponding part of the appendix} are essentially well-known; they are basically just minor variations of the results of~\cite{williams1975,Williams:2007eu} and \cite{Pelessoni:2009co}. \iftoggle{arxiv}{}{As before, we therefore state these results without proof. For those interested, explicit proofs and intermediate results are available in an online arXiv version of this contribution~\citep{DeBock2018Williams}.}

\vspace{10pt}

\iftoggle{arxiv}{
\begin{proof}{\bf of Proposition~\ref{prop:propertiesofP}~}
Because of definitions~\ref{def:prev} and~\ref{def:linearprev}, we know that $\pr$ is a coherent conditional lower prevision on $\C$ that satisfies Equation~\eqref{eq:self-conjugate}. Due to Proposition~\ref{prop:propertiesofLP}, this implies that $\pr$ satisfies~\ref{def:lowerprev:bounded}--\ref{def:lowerprev:monotonicity}. \ref{def:prev:bounded} and~\ref{def:prev:uniformcontinuity}--\ref{def:prev:monotonicity} follow trivially from~\ref{def:lowerprev:bounded} and~\ref{def:lowerprev:uniformcontinuity}--\ref{def:lowerprev:monotonicity}, respectively. \ref{def:prev:homo} holds because
\begin{equation*}
P(\lambda f\vert B)
=
\begin{cases}
\lambda P(f\vert B)&\text{~if $\lambda\geq0$}\\
-\lambda P(-f\vert B)&\text{~if $\lambda\leq0$}
\end{cases}
~~=\lambda P(f\vert B)
\end{equation*}
where the first equality follows from~\ref{def:lowerprev:homo} and the second one follows from Equation~\eqref{eq:self-conjugate}.
\ref{def:prev:additive} holds because
\begin{equation*}
P(f\vert B)+P(g\vert B)\leq P(f+g\vert B)=-P(-f-g\vert B)\leq-P(-f\vert B)-P(-g\vert B)=P(f\vert B)+P(g\vert B),
\end{equation*}
where the inequalities follow from~\ref{def:lowerprev:superadditive} and the equalities follow from Equation~\eqref{eq:self-conjugate}.
Finally, \ref{def:prev:GBR} holds because
\begin{equation*}
\pr(\ind{B}f\vert A)-\pr(f\vert A\cap B)\pr(B\vert A)
=\pr(\ind{B}f\vert A)-\pr(f\vert A\cap B)\pr(\ind{B}\vert A)
=\pr(\ind{B}[f-\pr(f\vert A\cap B)]\vert A)=0
\end{equation*}
where second equality follows from~\ref{def:prev:homo} and~\ref{def:prev:additive} and the third equality follows from~\ref{def:lowerprev:GBR}.
\end{proof}
}{}

\iftoggle{arxiv}{
\begin{proof}{\bf of Proposition~\ref{prop:linearPifffouraxioms}~}
If $\pr$ is a conditional linear prevision on $\C$, we know from Proposition~\ref{prop:cohlpifffouraxioms} that $\pr$ is real-valued and from Proposition~\ref{prop:propertiesofP} that it satisfies~\ref{def:prev:bounded}--\ref{def:prev:GBR}. So assume that $\pr$ is real-valued and satisfies~\ref{def:prev:bounded}--\ref{def:prev:GBR}. We need to prove that $\pr$ is a conditional linear prevision on $\C$.

Since $\pr$ satisfies~\ref{def:prev:bounded}--\ref{def:prev:additive}, it clearly satisfies~\ref{def:lowerprev:bounded}--\ref{def:lowerprev:superadditive} as well. $\pr$ also satisfies~\ref{def:lowerprev:GBR} because, for all $f\in\mathcal{F}$ and \mbox{$A,B\in\nonemptypower$} such that $A\in\B$ and $\emptyset\neq A\cap B\in\B$, it follows from~\ref{def:prev:homo}--\ref{def:prev:GBR} that
\begin{equation*}
\pr(\ind{B}[f-\pr(f\vert A\cap B)]\vert A)
=\pr(\ind{B}f\vert A)-\pr(f\vert A\cap B)\pr(\ind{B}\vert A)
=
\pr(\ind{B}f\vert A)-\pr(f\vert A\cap B)\pr(B\vert A)=0.
\end{equation*}
Since $\pr$ is real-valued and satisfies~\ref{def:lowerprev:bounded}--\ref{def:lowerprev:GBR}, and because we know from Definition~\ref{def:prev} that $\pr$ is a conditional lower prevision on $\C$, Proposition~\ref{prop:cohlpifffouraxioms} now implies that $\pr$ is coherent. Therefore, it follows from Definition~\ref{def:linearprev} that $\pr$ is a conditional linear prevision on $\C$.
\end{proof}
}{}

\iftoggle{arxiv}{
\begin{proof}{\bf of Corollary~\ref{corol:fulllinearPifffouraxioms}~}
Immediate consequence of Proposition~\ref{prop:linearPifffouraxioms}.
\end{proof}
\vspace{-10pt}
}{}

\iftoggle{arxiv}{
\begin{proof}{\bf of Corollary~\ref{corol:fulllinearPiffthreeaxioms}~}
Let $\mathcal{F}\coloneqq\G$ and $\B\coloneqq\{\states\}$. Since $\ind{\states}=1$, the conditions of Proposition~\ref{prop:linearPifffouraxioms} are then clearly satisfied. Hence, since a prevision on $\G$ is by definition a conditional prevision on $\C\coloneqq\{(f,\states)\colon f\in\G\}$, it follows from Proposition~\ref{prop:linearPifffouraxioms} that $\pr$ is coherent if and only if it is real-valued and satisfies~\ref{def:prev:bounded}--\ref{def:prev:GBR}. Therefore, the only thing that we still need to prove is that any real-valued prevision $\pr$ on $\G$ that satisfies \ref{def:prev:bounded}--\ref{def:prev:additive} will also satisfy \ref{def:prev:GBR}.

So consider any real-valued prevision $\pr$ on $\G$ that satisfies \ref{def:prev:bounded}--\ref{def:prev:additive}. Since the domain of a(n unconditional) prevision contains only couples of the form $(f,\states)$, proving \ref{def:prev:GBR} reduces to showing that $P(f\vert\states)=P(f\vert\states)P(\states\vert\states)$ for all $f\in\G$. Hence, it suffices to prove that $P(\states\vert\states)=1$.

On the one hand, we know from~\ref{def:prev:bounded} that $P(1)\geq1$. On the other hand, it follows from~\ref{def:prev:homo} and~\ref{def:prev:bounded} that $P(1)=-P(-1)\leq-(-1)=1$. Hence, we find that $P(1)=1$. Since $P(\states\vert\states)=P(\states)=P(\ind{\states})=P(1)$, this establishes the desired result.
\end{proof}
\vspace{-16pt}
}{}

\iftoggle{arxiv}{
\begin{lemma}\label{lemma:fulllinearlpifffouraxioms:withsup}
A conditional prevision on $\C(\states)$ is linear if and only if it is real-valued and satisfies~\ref{def:prev:boundedbysup} and~\ref{def:prev:homo}--\ref{def:prev:GBR}, with
\begin{enumerate}[label=P\arabic*':,ref=P\arabic*']
\item
$\pr(f\vert B)\leq\sup_{x\in B}f(x)$ for all $(f,B)\in\C(\states)$.\label{def:prev:boundedbysup}
\end{enumerate}
\end{lemma}
\begin{proof}{\bf of Lemma~\ref{lemma:fulllinearlpifffouraxioms:withsup}~}
Since $\pr$ is a conditional prevision and therefore satisfies Equation~\eqref{eq:self-conjugate}, we see that $\pr$ satisfies~\ref{def:prev:bounded} if and only if it satisfies~\ref{def:prev:boundedbysup}. Therefore, the result follows from Corollary~\ref{corol:fulllinearPifffouraxioms}.
\end{proof}
\vspace{-16pt}
}{}

\begin{proposition}\label{prop:lowerenvelope:onlyif:full}
Let $\lp$ be a coherent conditional lower prevision on $\C(\states)$ and let $\overline{P}$ be the corresponding conditional upper prevision on $\C(\states)$. Then there is a non-empty set $\mathbb{P}^*$ of conditional linear previsions on $\C(\states)$ such that
\begin{equation*}
\lp(f\vert B)=\min\{P(f\vert B)\colon P\in\mathbb{P}^*\}
\text{~~and~~}
\overline{P}(f\vert B)=\max\{P(f\vert B)\colon P\in\mathbb{P}^*\}
\end{equation*}
for all $(f,B)\in\C(\states)$. 
Furthermore, for any $(f,B)\in\C(\states)$ and $\alpha\in[\lp(f\vert B),\overline{P}(f\vert B)]$, there is some $P\in\mathbb{P}^*$ such that $P(f\vert B)=\alpha$.
\end{proposition}
\iftoggle{arxiv}{
\begin{proof}{\bf of Proposition~\ref{prop:lowerenvelope:onlyif:full}~}
Since $\lp$ is coherent, it follows from Proposition~\ref{prop:equivalentToPelessoniAndVicig} that $\lp$ is real-valued and satisfies Equation~\eqref{eq:prop:equivalentToPelessoniAndVicig}. Therefore, for all $n\in\natswith$ and all choices of $\lambda_0,\dots,\lambda_n\in\reals_{\geq0}$ and $(f_0,B_0),\dots,(f_n,B_n)\in\mathcal{C}(\states)$, if we let $B\coloneqq\cup_{i=0}^nB_i$, we find that
\vspace{-4pt}
\begin{multline*}
\sup_{x\in B}
\Big(
\sum_{i=1}^n
\lambda_i\ind{B_i}(x)
[(-f_i(x))-\lp(-f_i\vert B_i)]
-
\lambda_0\ind{B_0}(x)
[(-f_0(x))-\lp(-f_0\vert B_0)]
\Big)\\
=\sup_{x\in B}
\Big(
\lambda_0\ind{B_0}(x)
[f_0(x)-\overline{P}(f_0\vert B_0)]
-\sum_{i=1}^n
\lambda_i\ind{B_i}(x)
[f_i(x)-\overline{P}(f_i\vert B_i)]
\Big)
\geq0.
\end{multline*}
Since this means that $\overline{P}$ satisfies condition (A*) in~\cite{Williams:2007eu}, it now follows from~\cite[Theorem~2, Definition~2 and Proposition~6]{Williams:2007eu} and Lemma~\ref{lemma:fulllinearlpifffouraxioms:withsup} that there is a non-empty set $\mathbb{P}^*$ of conditional linear previsions on $\C(\states)$ such that
\begin{equation}\label{eq:prop:lowerenvelope:onlyif:full:2}
\overline{P}(f\vert B)=\sup\{P(f\vert B)\colon P\in\mathbb{P}^*\}=\max\{P(f\vert B)\colon P\in\mathbb{P}^*\}
~~\text{for all $(f,B)\in\C(\states)$}
\end{equation}
and such that, for all $(f,B)\in\C(\states)$ and $\alpha\in[\lp(f\vert B),\overline{P}(f\vert B)]$, there is some $P\in\mathbb{P}^*$ such that $P(f\vert B)=\alpha$. 
The first equality in Equation~\eqref{eq:prop:lowerenvelope:onlyif:full:2} corresponds to~\cite[Theorem~2]{Williams:2007eu}; the second equality in Equation~\eqref{eq:prop:lowerenvelope:onlyif:full:2} and the statement about $\alpha$ is not stated in~\cite[Theorem~2]{Williams:2007eu} itself, but clearly follows from the end of its proof. \cite[Definition~2 and Proposition~6]{Williams:2007eu} and Lemma~\ref{lemma:fulllinearlpifffouraxioms:withsup} are needed solely for the purpose of establishing that what Williams calls a conditional prevision in~\cite[Theorem~2]{Williams:2007eu} is equivalent to what we here call a conditional linear prevision on $\C(\states)$. Finally, we have that
\begin{align*}
\lp(f\vert B)
=-\overline{P}(-f\vert B)
&=-\max\{P(-f\vert B)\colon P\in\mathbb{P}^*\}\\
&=-\max\{-P(f\vert B)\colon P\in\mathbb{P}^*\}
=\min\{P(f\vert B)\colon P\in\mathbb{P}^*\},
\end{align*}
where the second equality follows from Equation~\eqref{eq:prop:lowerenvelope:onlyif:full:2} and the third equality follows from Equation~\eqref{eq:self-conjugate}.
\end{proof}
\vspace{-16pt}
}{}

\iftoggle{arxiv}{
\begin{proposition}\label{prop:lowerenvelope:if}
Let $\mathbb{P}^*$ be a non-empty set of conditional linear previsions on $\C(\states)$, consider some domain $\C\subseteq\C(\states)$, and let $\lp$ be the conditional lower prevision on $\C$ that is defined by
\begin{equation}\label{eq:prop:lowerenvelope:if:1}
\lp(f\vert B)=\inf\{P(f\vert B)\colon P\in\mathbb{P}^*\}
\text{ for all $(f,B)\in\C$.}
\end{equation}
Then $\lp$ is coherent.
\end{proposition}
\begin{proof}{\bf of Proposition~\ref{prop:lowerenvelope:if}~}
We will prove that $\lp$ is real-valued and that, for all $n\in\natswith$ and all choices of $\lambda_0,\dots,\lambda_n\in\reals_{\geq0}$ and $(f_0,B_0),\dots,(f_n,B_n)\in\mathcal{C}$,
\begin{equation}\label{eq:prop:lowerenvelope:if:2}
\sup_{x\in B}
\Big(\,
\sum_{i=1}^n
\lambda_i\ind{B_i}(x)
[f_i(x)-\lp(f_i\vert B_i)]
-\lambda_0\ind{B_0}(x)
[f_0(x)-\lp(f_0\vert B_0)]
\Big)
\geq0,
\vspace{6pt}
\end{equation}
with $B\coloneqq\cup_{i=0}^nB_i$. Proposition~\ref{prop:equivalentToPelessoniAndVicig} then implies that $\lp$ is coherent.

Let us first prove that $\lp$ is real-valued. Fix any $(f,B)\in\C$. For every $P\in\mathbb{P}^*$, it then follows from Proposition~\ref{prop:propertiesofP} and Lemma~\ref{lemma:fulllinearlpifffouraxioms:withsup} that $\inf_{x\in B}\leq P(f\vert B)\leq\sup_{x\in B}f(x)$. Hence, since $\mathbb{P}^*$ is non-empty, it follows from Equation~\eqref{eq:prop:lowerenvelope:if:1} that $\inf_{x\in B}\leq\lp(f\vert B)\leq\sup_{x\in B}f(x)$. Since $f$ is a gamble and therefore by definition bounded, this implies that $\lp(f\vert B)$ is real-valued. Since this is true for every $(f,B)\in\C$, it follows that $\lp$ is real-valued.

Finally, fix any $n\in\natswith$, any $\lambda_0,\dots,\lambda_n\in\reals_{\geq0}$ and $(f_0,B_0),\dots,(f_n,B_n)\in\mathcal{C}$, let $B\coloneqq\cup_{i=0}^nB_i$ and consider any $\epsilon>0$. It then follows from Equation~\eqref{eq:prop:lowerenvelope:if:1} that there is a conditional linear prevision $P\in\mathbb{P}^*$ on $\C(\states)$ such that $\lambda_0 P(f_0\vert B_0)\leq\lambda_0\lp(f_0\vert B_0)+\epsilon$. Furthermore, for any $i\in\{1,\dots,n\}$, Equation~\eqref{eq:prop:lowerenvelope:if:1} also implies that $P(f_i\vert B_i)\geq\lp(f_i\vert B_i)$. Hence, we find that
\begin{align*}
\sup_{x\in B}
\Big(\,
\sum_{i=1}^n
\lambda_i\ind{B_i}(x)
&[f_i(x)-\lp(f_i\vert B_i)]
-\lambda_0\ind{B_0}(x)
[f_0(x)-\lp(f_0\vert B_0)]
\Big)\\
&\geq\sup_{x\in B}
\Big(\,
\sum_{i=1}^n
\lambda_i\ind{B_i}(x)
[f_i(x)-\pr(f_i\vert B_i)]
-\lambda_0\ind{B_0}(x)
[f_0(x)-\pr(f_0\vert B_0)]
-\ind{B_0}\epsilon
\Big)\\
&\geq\sup_{x\in B}
\Big(\,
\sum_{i=1}^n
\lambda_i\ind{B_i}(x)
[f_i(x)-\pr(f_i\vert B_i)]
-\lambda_0\ind{B_0}(x)
[f_0(x)-\pr(f_0\vert B_0)]
\Big)-\epsilon\geq-\epsilon,
\end{align*}
where the last inequality follows from Proposition~\ref{prop:equivalentToPelessoniAndVicig} because we know from Definition~\ref{def:linearprev} that $P$ is coherent. Since $\epsilon>0$ is arbitrary, we obtain Equation~\eqref{eq:prop:lowerenvelope:if:2}, as desired.
\end{proof}
\vspace{-10pt}
}{}

\iftoggle{arxiv}{
\begin{proof}{\bf of Proposition~\ref{prop:lowerenvelope:onlyif:natex}~}
Since we know from Proposition~\ref{prop:naturalextension:full} that $\natexLP$ is coherent, Proposition~\ref{prop:lowerenvelope:onlyif:full} implies that there is a non-empty set $\mathbb{P}^*$ of conditional linear previsions on $\C(\states)$ such that
\begin{equation}\label{eq:prop:lowerenvelope:onlyif:natex:1}
\natexLP(f\vert B)=\min\{P(f\vert B)\colon P\in\mathbb{P}^*\}
\text{~~and~~}
\natexUP(f\vert B)=\max\{P(f\vert B)\colon P\in\mathbb{P}^*\}
\end{equation}
for all $(f,B)\in\C(\states)$ and such that, furthermore, for any $(f,B)\in\C(\states)$ and $\alpha\in[\lp(f\vert B),\overline{P}(f\vert B)]$, there is some $P\in\mathbb{P}^*$ such that $P(f\vert B)=\alpha$.

Consider now any $P\in\mathbb{P}^*$. For any $(f,B)\in\C$, Proposition~\ref{prop:naturalextension:full} and Equation~\eqref{eq:prop:lowerenvelope:onlyif:natex:1} then imply that
\begin{equation*}
\lp(f\vert B)=\natexLP(f\vert B)=-\natexUP(-f\vert B)=-\max\{\tilde P(f\vert B)\colon\tilde P\in\mathbb{P}^*\}\leq P(f\vert B).
\end{equation*}
Therefore, it follows that $P\in\mathbb{P}_{\lp}$ and, since this is true for every $P\in\mathbb{P}^*$, that $\mathbb{P}^*\subseteq\mathbb{P}_{\lp}$. Hence, it follows that for all $(f,B)\in\C(\states)$ and $\alpha\in[\natexLP(f\vert B),\natexUP(f\vert B)]$, there is some $P\in\mathbb{P}^*\subseteq\mathbb{P}_{\lp}$ such that $P(f\vert B)=\alpha$, which already establishes the final part of our proposition. 

The first part of the proposition follows from this last part. Fix any $(f,B)\in\C(\states)$. Since $\natexLP$ is coherent, we know from Proposition~\ref{prop:propertiesofLP} that $\natexLP(f\vert B)\leq\natexUP(f\vert B)$, which implies that the interval $[\natexLP(f\vert B),\natexUP(f\vert B)]$ is non-empty. The final part of our proposition therefore implies that $\mathbb{P}_{\lp}\neq\emptyset$, simply by choosing some $\alpha\in[\natexLP(f\vert B),\natexUP(f\vert B)]$ and considering the corresponding $P\in\mathbb{P}_{\lp}$. 

The only thing left to prove now is Equation~\eqref{eq:prop:lowerenvelope:onlyif:natex}. However, since $\natexUP(f\vert B)=-\natexLP(-f\vert B)$, the second part of that equation follows trivially from the first part. Hence, it suffices to prove that $\natexLP(f\vert B)=\min\{P(f\vert B\colon P\in\mathbb{P}_{\lp}\}$ for all $(f,B)\in\C(\states)$.

To this end, let $\tilde\lp$ be the conditional lower prevision on $\C(\states)$ that is defined by
\begin{equation}\label{eq:prop:lowerenvelope:onlyif:natex:2}
\tilde\lp(f\vert B)\coloneqq\inf\{P(f\vert B\colon P\in\mathbb{P}_{\lp}\}
\text{ for all $(f,B)\in\C(\states)$.}
\end{equation}
Then on the one hand, since $\mathbb{P}^*\subseteq\mathbb{P}_{\lp}$, it follows from Equation~\eqref{eq:prop:lowerenvelope:onlyif:natex:1} that for all $(f,B)\in\C(\states)$:
\begin{equation}\label{eq:prop:lowerenvelope:onlyif:natex:3}
\natexLP(f\vert B)=\min\{P(f\vert B)\colon P\in\mathbb{P}^*\}\geq\inf\{P(f\vert B)\colon P\in\mathbb{P}_{\lp}\}=\tilde\lp(f\vert B).
\end{equation}
On the other hand, for any $(f,B)\in\C$, Proposition~\ref{prop:naturalextension:full} and Equation~\eqref{eq:setofdominatingcondlinprev} imply that
\begin{equation*}
\natexLP(f\vert B)=\lp(f\vert B)\leq\inf\{P(f\vert B)\colon P\in\mathbb{P}_{\lp}\}=\tilde\lp(f\vert B).
\end{equation*}
By combining these two results, it follows that $\tilde\lp$ coincides with $\lp$ on $\C$. Furthermore, since $\mathbb{P}_{\lp}$ is a non-empty set of conditional linear previsions on $\C(\states)$, we know from Proposition~\ref{prop:lowerenvelope:if} that $\tilde\lp$ is coherent. Therefore, it follows from Proposition~\ref{prop:naturalextension:full} that $\tilde\lp(f\vert B)\geq\natexLP(f\vert B)$ for all $(f,B)\in\C(\states)$. Combined, with Equation~\eqref{eq:prop:lowerenvelope:onlyif:natex:2} and~\eqref{eq:prop:lowerenvelope:onlyif:natex:3}, this implies that
\begin{equation}\label{eq:prop:lowerenvelope:onlyif:natex:4}
\natexLP(f\vert B)
=\tilde\lp(f\vert B)
=\inf\{P(f\vert B)\colon P\in\mathbb{P}_{\lp}\}
\text{ for all $(f,B)\in\C(\states)$.}
\end{equation}
It remains to show that this infimum is actually a minimum. This can easily be inferred from Equation~\eqref{eq:prop:lowerenvelope:onlyif:natex:1}, as follows. For any $(f,B)\in\C(\states)$, Equation~\eqref{eq:prop:lowerenvelope:onlyif:natex:1} implies that there is some $P\in\mathbb{P}^*$ such that $P(f\vert B)=\natexLP(f\vert B)$ and, since $\mathbb{P}^*\subseteq\mathbb{P}_{\lp}$, this in turn implies that the infimum in Equation~\eqref{eq:prop:lowerenvelope:onlyif:natex:4} is a minimum.
\end{proof}
\vspace{-10pt}
}{}

\iftoggle{arxiv}{
\begin{proof}{\bf of Proposition~\ref{prop:lowerenvelope}~}
We first prove the `only if' part of the statement. So consider any coherent conditional lower prevision on $\C\subseteq\C(\states)$. For any $(f,B)\in\C$, we then have that
\begin{equation*}
\lp(f\vert B)=\natexLP(f\vert B)=\min\{P(f\vert B)\colon P\in\mathbb{P}_{\lp}\},
\end{equation*}
where $\natexLP$ is the natural extension of $\lp$ to $\C(\states)$, and where the first equality follows from Proposition~\ref{prop:naturalextension:full} and the second from Proposition~\ref{prop:lowerenvelope:onlyif:natex}. Since every minimum is an infimum, the `only if' part of the statement follows by choosing $\mathbb{P}^*=\mathbb{P}_{\lp}$.

The `if' part of the statement is an immediate consequence of Proposition~\ref{prop:lowerenvelope:if}.
\end{proof}
}{}
\vspace{-10pt}

\subsection{Proofs and Additional Material for Section~\ref{sec:indnatext}}
\vspace{5pt}

\subsubsection{The Sets of Desirable Gambles Part}
\vspace{5pt}

\begin{proposition}\label{prop:productcoherent:SDG}
$\desir_1\otimes\desir_2$ is a coherent set of desirable gambles on $\states_1\times\states_2$.
\end{proposition}
\begin{proof}{\bf of Proposition~\ref{prop:productcoherent:SDG}~} 
Because of Lemma~\ref{lemma:coherenceiffLP4}, it suffices to prove~\ref{def:SDG:partialloss}. So consider any $f\in\desir_1\otimes\desir_2$ and assume \emph{ex absurdo} that $f\leq0$. We will prove that this leads to a contradiction.

Since $\desir_1$ and $\desir_2$ are coherent, they are closed with respect to positive scaling and finite sums. Therefore, and because $f\in\desir_1\otimes\desir_2=\E
\left(
\A_{1\to2}
\cup
\A_{2\to1}
\right)$, it follows from Equations~\eqref{eq:A12s} and~\eqref{eq:A21s} that
\vspace{2pt}
\begin{equation}\label{eq:fexplicit}
f=\sum_{i\in I}\ind{B_{1,i}}(X_1)f_{2,i}(X_2)+\sum_{j\in J}\ind{B_{2,j}}(X_2)f_{1,j}(X_1)+g,
\end{equation}
with $I$ and $J$ finite---possibly empty---index sets, with $B_{1,i}\in\nonemptypoweron{1}$ and $f_{2,i}\in\desir_2$ for all $i\in I$, with $B_{2,j}\in\nonemptypoweron{2}$ and $f_{1,j}\in\desir_1$ for all $j\in J$, with $g\geq0$, and where $g=0$ is only possible if $\abs{I}+\abs{J}>0$. 

Let us assume~\emph{ex absurdo} that $\abs{I}+\abs{J}=0$. Then on the one hand, since we know that  $g=0$ is only possible if $\abs{I}+\abs{J}>0$, it follows that $g\neq0$. On the other hand, $\abs{I}+\abs{J}=0$ also implies that $I=J=\emptyset$, and therefore, due to Equation~\eqref{eq:fexplicit}, that $f=g$. Since $g\geq0$ and $f\leq0$, this in turn implies that $g=0$, thereby contradicting the fact that $g\neq0$. Hence, it follows that at least one of the two \emph{ex absurdo} assumptions that we have so far made must be wrong. If $f\not\leq0$, then the proof is finished. For that reason, in the remainder of the proof, we can assume that $\abs{I}+\abs{J}\neq0$, and therefore, that $\abs{I}+\abs{J}>0$. The only \emph{ex absurdo} assumption that still remains is that $f\leq0$.

Now let $\{B_{1,k}\}_{k\in K}$ be the set consisting of those atoms of the algebra generated by $\{B_{1,i}\}_{i\in I}$ that belong to $\cup_{i\in I}B_{1,i}$ and, for all $k\in K$, let $f_{2,k}\coloneqq\sum_{i\in I\colon B_{1,k}\subseteq B_{1,i}}f_{2,i}$. The following properties are then easily verified. First, since $I$ is finite, $K$ is also finite. Secondly, $\abs{K}=0$ if and only if $\abs{I}=0$. Thirdly, for all $k\in K$, we have that $B_{1,k}\in\nonemptypoweron{1}$ and---since $\desir_1$ is coherent and therefore satisfies~\ref{def:SDG:convex}---that $f_{2,k}\in\desir_1$. Fourthly, $\sum_{k\in K}\ind{B_{1,k}}(X_1)f_{2,k}(X_2)$ is equal to $\sum_{i\in I}\ind{B_{1,i}}(X_1)f_{2,i}(X_2)$. Fifthly, the events in $\{B_{1,k}\}_{k\in K}$ are pairwise disjoint. For this reason, without loss of generality, we can assume the events $\{B_{1,i}\}_{i\in I}$ in Equation~\eqref{eq:fexplicit} to be pairwise disjoint. A completely similar argument leads us to conclude that the events $\{B_{2,j}\}_{j\in J}$ in Equation~\eqref{eq:fexplicit} can be assumed to be pairwise disjoint, again without loss of generality.

If $\{B_{1,i}\}_{i\in I}$ is a partition of $\states_1$, then we let $\Y_1\coloneqq I$. Otherwise, we let $\Y_1\coloneqq I\cup\{i^*\}$ and define $B_{1,i^*}\coloneqq\states_1\setminus\cup_{i\in I}B_{1,i}$. Similarly, we let $\Y_2\coloneqq J$ if $\{B_{2,j}\}_{j\in J}$ is a partition of $\states_2$, and let $\Y_2\coloneqq J\cup\{j^*\}$ and $B_{1,j^*}\coloneqq\states_2\setminus\cup_{j\in J}B_{2,j}$ otherwise.
Next, for every $i\in I$, we let $h_{2,i}$ be a gamble on $\Y_2$, defined by
\begin{equation}\label{eq:discretise2}
h_{2,i}(y_2)\coloneqq\sup\{f_{2,i}(x_2)\colon x_2\in B_{2,y_2}\}
~\text{ for all $y_2\in\Y_2$.}
\vspace{4pt}
\end{equation}
Similarly, for every $j\in J$, we let $h_{1,j}$ be a gamble on $\Y_1$, defined by
\begin{equation}\label{eq:discretise1}
h_{1,j}(y_1)\coloneqq\sup\{f_{1,j}(x_1)\colon x_1\in B_{1,y_1}\}
~\text{ for all $y_1\in\Y_1$.}
\vspace{2pt}
\end{equation}
Using these gambles on $\Y_1$ and $\Y_2$, we now construct a real-valued function $h$ on $\Y_1\times\Y_2$, defined by
\begin{equation}\label{eq:hforEcontradiction}
h(y_1,y_2)
\coloneqq
\sum_{i\in I}\ind{i}(y_1)h_{2,i}(y_2)+\sum_{j\in J}\ind{j}(y_2)h_{1,j}(y_1)
\text{~~for all $y_1\in\Y_1$ and $y_2\in\Y_2$}
\vspace{3pt}
\end{equation}
This function is non-positive, in the sense that $h\leq0$. In order to prove that, let us fix any $y_1\in\Y_1$ and $y_2\in\Y_2$. It then follows from Equations~\eqref{eq:discretise2} and~\eqref{eq:discretise1} that
\vspace{4pt}
\begin{equation*}
h(y_1,y_2)
=
\sum_{i\in I}\ind{i}(y_1)\sup_{x_2\in B_{2,y_2}}f_{2,i}(x_2)+\sum_{j\in J}\ind{j}(y_2)\sup_{x_1\in B_{1,y_1}}f_{1,j}(x_1).
\end{equation*}
Since $\ind{i}(y_1)$ can be non-zero for at most one $i\in I$ and $\ind{j}(y_2)$ can be non-zero for at most one $j\in J$, we know that each of the two summations on the right hand side contains at most one non-zero term. The suprema can therefore be moved outside of the summations, yielding
\begin{equation*}
h(y_1,y_2)
=
\sup_{x_1\in B_{1,y_1}}\sup_{x_2\in B_{2,y_2}}\left(\sum_{i\in I}\ind{i}(y_1)f_{2,i}(x_2)+\sum_{j\in J}\ind{j}(y_2)f_{1,j}(x_1)\right).
\end{equation*}
For the next step, we start by observing the following. For any $x_1\in B_{1,y_1}$ and any $i\in I$, since the sets $\{B_{1,i}\}_{i\in I}$ are pairwise disjoint, we know that $x_1\in B_{1,i}$ if and only if $y_1=i$, which implies that $\ind{i}(y_1)=\ind{B_{1,i}}(x_1)$. Similarly, for any $x_2\in B_{2,y_2}$ and any $j\in J$, since the sets $\{B_{2,j}\}_{j\in J}$ are pairwise disjoint, we know that $x_2\in B_{2,j}$ if and only if $y_2=j$, which implies that $\ind{j}(y_2)=\ind{B_{2,j}}(x_2)$. As an immediate consequence, it follows that
\begin{equation*}
h(y_1,y_2)
=
\sup_{x_1\in B_{1,y_1}}\sup_{x_2\in B_{2,y_2}}\left(\sum_{i\in I}\ind{B_{1,i}}(x_1)f_{2,i}(x_2)+\sum_{j\in J}\ind{B_{2,j}}(x_2)f_{1,j}(x_1)\right).
\end{equation*}
Finally, in combination with Equation~\eqref{eq:fexplicit}, this implies that
\begin{equation*}
h(y_1,y_2)
=
\sup_{x_1\in B_{1,y_1}}\sup_{x_2\in B_{2,y_2}}\left(
f(x_1,x_2)-g(x_1,x_2)
\right)
\leq0,
\end{equation*}
where, for the last inequality, we use the fact that $f\leq0$ and $g\geq0$. Since this true for every $y_1\in\Y_1$ and $y_2\in\Y_2$, it follows that $h\leq0$.

Now let $\A_1\coloneqq\{h_{1,j}\colon j\in J\}$ and $\A_2\coloneqq\{h_{2,i}\colon i\in I\}$ and assume~\emph{ex absurdo} that $\mathcal{H}_1\coloneqq\E(\A_1)$ and $\mathcal{H}_2\coloneqq\E(\A_2)$ are coherent sets of desirable gambles on $\Y_1$ and $\Y_2$, respectively. We will prove that this is impossible, by constructing a probability mass function $p$ on $\states_1\times\states_2$ such that the corresponding expectation of $h$ is both non-positive \emph{and}  positive, thereby obtaining a contradiction. In order to do that, we borrow an argument of De Cooman and Miranda (\citeyear[Proof of Proposition~15]{deCooman:2012vba}) that is based on a very useful lemma of them, which, in order to make this paper self-contained, is restated here in Lemma~\ref{lemma:GertandQuique}.

Since $\mathcal{H}_1$ is a coherent set of desirable gambles on $\Y_1$, it follows from Definition \ref{def:SDG}---and \ref{def:SDG:partialloss} in particular---that $0\notin\mathcal{H}_1=\E(\A_1)$. Therefore, and because $\Y_1$ and $J$---and hence also $\A_1$---are finite, it follows from Lemma~\ref{lemma:GertandQuique} that there is a probability mass function $p_1$ on $\Y_1$ such that $p_1(y_1)>0$ for all $y_1\in\Y_1$ and $\sum_{y_1\in\Y_1}p_1(y_1)h_{1,j}(y_1)>0$ for all $j\in J$. Using a completely analogous argument, we also infer that there is a probability mass function $p_2$ on $\Y_2$ such that $p_2(y_2)>0$ for all $y_2\in\Y_2$ and $\sum_{y_2\in\Y_2}p_2(y_2)h_{2,i}(y_2)>0$ for all $i\in I$. 

We now let $p$ be the probability mass function on $\Y_1\times\Y_2$ that is defined by $p(y_1,y_2)\coloneqq p_1(y_1)p_2(y_2)$ for all $y_1\in\Y_1$ and $y_2\in\Y_2$, and we let $E_p(h)$ be the expectation of $h$ with respect to $p$, as defined by
\begin{equation*}
E_p(h)
\coloneqq
\sum_{y_1\in\Y_1}
\sum_{y_2\in\Y_2}
p(y_1,y_2)h(y_1,y_2)
=
\sum_{y_1\in\Y_1}
\sum_{y_2\in\Y_2}
p_1(y_1)p_2(y_2)h(y_1,y_2).
\end{equation*}
Then on the one hand, since $h\leq0$, we have that $E_p(h)\leq0$. On the other hand, however, it follows from Equation~\eqref{eq:hforEcontradiction} that
\vspace{-3pt}
\begin{align*}
E_p(h)
&=\sum_{y_1\in\Y_1}
\sum_{y_2\in\Y_2}
p_1(y_1)p_2(y_2)\left(
\sum_{i\in I}\ind{i}(y_1)h_{2,i}(y_2)+\sum_{j\in J}\ind{j}(y_2)h_{1,j}(y_1)
\right)\\
&=\sum_{y_1\in\Y_1}
\sum_{y_2\in\Y_2}
p_1(y_1)p_2(y_2)
\sum_{i\in I}\ind{i}(y_1)h_{2,i}(y_2)+
\sum_{y_1\in\Y_1}
\sum_{y_2\in\Y_2}
p_1(y_1)p_2(y_2)
\sum_{j\in J}\ind{j}(y_2)h_{1,j}(y_1)\\
&=
\sum_{i\in I}
\sum_{y_1\in\Y_1}
p_1(y_1)
\ind{i}(y_1)
\sum_{y_2\in\Y_2}
p_2(y_2)
h_{2,i}(y_2)
+
\sum_{j\in J}
\sum_{y_2\in\Y_2}
p_2(y_2)
\ind{j}(y_2)
\sum_{y_1\in\Y_1}
p_1(y_1)
h_{1,j}(y_1)\\
&=
\sum_{i\in I}
p_1(i)
\sum_{y_2\in\Y_2}
p_2(y_2)
h_{2,i}(y_2)
+
\sum_{j\in J}
p_2(j)
\sum_{y_1\in\Y_1}
p_1(y_1)
h_{1,j}(y_1).
\vspace{4pt}
\end{align*}
For every $i\in I$, it follows from the properties of $p_1$ and $p_2$ that the corresponding term in this summation is positive. Similarly, for every $j\in J$, it follows from the properties of $p_1$ and $p_2$ that the corresponding term in this summation is positive. Since $\abs{I}+\abs{J}>0$, this implies that $E_p(h)>0$, thereby contradicting the fact that $E_p(h)\leq0$. Hence, it follows that one of the two remaining \emph{ex absurdo} assumptions is wrong. If $f\leq0$, then the proof is finished. Therefore, in the remainder of the proof, we can assume that there is at least one $i\in\{1,2\}$ for which $\mathcal{H}_i$ is incoherent. Without loss of generality, symmetry allows us to assume that $i=1$, that is, that $\mathcal{H}_1$ is incoherent. The only \emph{ex absurdo} assumption that still remains is that $f\leq0$.

Since $\mathcal{H}_1$ is incoherent, it follows from Lemma~\ref{lemma:coherenceiffLP4} that there is some $h^*\in\mathcal{H}_1$ such that $h^*\leq0$. Furthermore, since $h^*\in\mathcal{H}_1$, Equation~\eqref{eq:natextop} implies that $h^*=\lambda g^*+\sum_{j\in J}\lambda_j h_{1,j}$, for some $\lambda\in\realsnonneg$ and $g^*\in\mathcal{G}_{>0}(\Y_1)$ and, for all $j\in J$, some $\lambda_j\in\realsnonneg$, with $\lambda+\sum_{j\in J}\lambda_j>0$. If $\lambda_j=0$ for all $j\in J$, then $\lambda>0$ and $g^*=\nicefrac{1}{\lambda}h^*\leq0$, which is impossible because $g^*\in\mathcal{G}_{>0}(\Y_1)$. Therefore, we know that there is at least one $j\in J$ such that $\lambda_j>0$.

Now let $f_1\coloneqq\sum_{j\in J}\lambda_jf_{1,j}$ and fix any $x_1^*\in\states_1$. Since the events in $\{B_{1,y_1}\}_{y_1\in \Y_1}$ are pairwise disjoint, there will then be a unique $y_1^*\in\Y_1$ such that $x_1^*\in B_{1,y_1^*}$. For this particular choice of $y_1^*$, we then find that
\begin{equation*}
f_1(x_1^*)
=
\sum_{j\in J}\lambda_jf_{1,j}(x_1^*)
\leq
\sum_{j\in J}\lambda_j
\sup_{x_1\in B_{1,y_1^*}}
f_{1,j}(x_1)
=
\sum_{j\in J}\lambda_j
h_{1,j}(y_1^*)
=h^*(y_1^*)-\lambda g^*(y_1^*)\leq0,
\end{equation*}
where the first equality follows from Equation~\eqref{eq:discretise1} and the second inequality follows from the fact that $h^*\leq0$, $\lambda\geq0$ and $g^*\in\mathcal{G}_{>0}(\Y_1)$. Since this is true for every $x_1^*\in\states_1$, we infer that $f_1\leq0$. However, on the other hand, since there is at least one $j\in J$ such that $\lambda_j>0$, and because $f_{1,j}\in\desir_1$ for all $j\in J$, the coherence of $\desir_1$ implies that $f_1\in\desir_1$ and therefore, because of~\ref{def:SDG:partialloss}, that $f_1\not\leq0$. From this contradiction, it follows that one of our \emph{ex absurdo} assumptions must be false. Since the only remaining \emph{ex absurdo} assumption is that $f\leq0$, this concludes the proof.
\end{proof}
\vspace{-16pt}

\begin{lemma}\label{lemma:GertandQuique}\cite[Lemma~2]{deCooman:2012vba}
Let $\Omega$ be a finite set and consider some finite subset $\A$ of $\mathcal{G}(\Omega)$. Then $0\notin\E(\A)$ if and only if there is a probability mass function $p$ on $\Omega$ such that $p(\omega)>0$ for all $\omega\in\Omega$ and $\sum_{\omega\in\Omega}p(\omega)f(\omega)>0$ for all $f\in\A$.
\end{lemma}

\begin{proposition}\label{prop:productindependent:SDG}
$\desir_1\otimes\desir_2$ is an independent product of $\desir_1$ and $\desir_2$.
\end{proposition}
\begin{proof}{\bf of Proposition~\ref{prop:productindependent:SDG}}
For ease of notation, let $\desir\coloneqq\desir_1\otimes\desir_2$. Because of symmetry, it clearly suffices to prove that
\begin{equation*}
(\forall B_2\in\B_2)~~\desir_1=\marg_1(\desir)=\marg_1(\desir\vert B_2),
\vspace{6pt}
\end{equation*}
which, since $\marg_1(\desir)=\marg_1(\desir\vert\states_2)$, is equivalent to proving that, for all $f_1\in\gambleson{1}$ and $B_2\in\B_2\cup\{\states_2\}$,
\begin{equation*}
f_1(X_1)\ind{B_2}(X_2)\in\desir
~\asa~
f_1\in\desir_1.
\vspace{5pt}
\end{equation*}
Since $f_1\in\desir_1$ implies that $f_1(X_1)\ind{B_2}(X_2)\in\A_{2\to1}\subseteq\desir$ for all $B_1\in\B_2\cup\{\states_2\}$, the converse implication holds trivially. So consider any $f_1\in\gambleson{1}$ and $B_2\in\B_2\cup\{\states_2\}$ such that $f_1(X_1)\ind{B_2}(X_2)\in\desir$. Since we know from Proposition~\ref{prop:productcoherent:SDG} that $\desir$ is coherent, this implies that $f_1\neq0$. It remains to prove that $f_1\in\desir_1$.

Assume \emph{ex absurdo} that $f_1\notin\desir_1$. Then since $f_1\neq0$, $\desir_1^{\bullet}\coloneqq\E(\desir_1\cup\{-f_1\})$ is a coherent set of desirable gambles on $\states_1$ because of Lemma~\ref{lemma:extendD}, and therefore, if we let
\vspace{4pt}
\begin{equation}\label{eq:Abullet}
\A_{2\to1}^{\bullet}
\coloneqq
\left\{
f'_1(X_1)\ind{B'_2}(X_2)
\colon
f'_1\in\desir_1^{\bullet}, 
B'_2\in\B_2\cup\{\states_2\}
\right\},
\vspace{4pt}
\end{equation}
it follows from Proposition~\ref{prop:productcoherent:SDG} that $\desir_1^\bullet\otimes\desir_2\coloneqq\E(\A_{1\to2}\cup\A^\bullet_{2\to1})$ is a coherent set of desirable gambles on $\states_1\times\states_2$. Now on the one hand, since $-f_1\in\desir_1^\bullet$, it follows from Equation~\eqref{eq:Abullet} that $-f_1(X_1)\ind{B_2}(X_2)\in\A_{2\to1}^{\bullet}\subseteq\desir_1^\bullet\otimes\desir_2$. On the other hand, since $\desir_1\subseteq\desir_1^{\bullet}$ implies that $\desir_1\otimes\desir_2\subseteq\desir_1^\bullet\otimes\desir_2$, we infer from $f_1(X_1)\ind{B_2}(X_2)\in\desir$ that $f_1(X_1)\ind{B_2}(X_2)\in\desir_1^\bullet\otimes\desir_2$. Since $\desir_1^\bullet\otimes\desir_2$ is coherent, this implies that
\vspace{3pt}
\begin{equation*}
0=f_1(X_1)\ind{B_2}(X_2)-f_1(X_1)\ind{B_2}(X_2)\in
\desir_1^\bullet\otimes\desir_2,
\vspace{3pt}
\end{equation*}
which contradicts~\ref{def:SDG:partialloss}.
\end{proof}
\vspace{-6pt}

\begin{proof}{\bf of Theorem~\ref{theo:natext:SDG}~}
Since we know from Proposition~\ref{prop:productindependent:SDG} that $\desir_1\otimes\desir_2$ is an independent product of $\desir_1$ and $\desir_2$, it suffices to prove that any other such independent product of $\desir_1$ and $\desir_2$ is a superset of $\desir_1\otimes\desir_2$.

So let $\desir$ be any independent product of $\desir_1$ and $\desir_2$. Definition~\ref{def:subsetindependentproduct} then implies that $\desir$ is coherent and that $\A_{1\to2}
\cup
\A_{2\to1}
\subseteq\desir$. Hence, we find that
\begin{equation*}
\desir_1\otimes\desir_2
=
\E
\left(
\A_{1\to2}
\cup
\A_{2\to1}
\right)
\subseteq\E(\desir)=\desir, 
\end{equation*}
where the inclusion follows from Lemma~\ref{lemma:nestedpropsofposandE} and the final equality from Lemma~\ref{lemma:natextDisD}.
\end{proof}
\vspace{-13pt}

\subsubsection{The Conditional Lower Previsions Part}
\vspace{5pt}

\begin{proposition}\label{prop:productcoherent:LP}
$\lp_1\otimes\lp_2$ is a coherent conditional probability on $\C(\states_1\times\states_2)$.
\end{proposition}

\begin{proof}{\bf of Proposition~\ref{prop:productcoherent:LP}~}
For all $i\in\{1,2\}$, since $\lp_i$ is a coherent conditional lower prevision on $\C_i$, it follows from Proposition~\ref{prop:smallestSDGfromLP} that $\E(\lp_i)$ is a coherent set of desirable gambles on $\states_i$. Therefore, Proposition~\ref{prop:productcoherent:SDG} implies that $\E(\lp_1)\otimes\E(\lp_2)$ is a coherent set of desirable gambles on $\states_1\times\states_2$. The result now follows from Definition~\ref{def:cohlp}.
\end{proof}
\vspace{-16pt}

\begin{proposition}\label{prop:independent:localnatex}
Consider two indexes $i$ and $j$ such that $\{i,j\}=\{1,2\}$. Then for any $f_i\in\gambleson{i}$ and $B_i\in\nonemptypoweron{i}$ and any $B_j\in\B_j$, we have that
\begin{equation}\label{eq:prop:independent:localnatex}
(\lp_1\otimes\lp_2)(f_i\vert B_i\cap B_j)
=
(\lp_1\otimes\lp_2)(f_i\vert B_i)
=
\natexLP_i(f_i\vert B_i).
\end{equation}
\end{proposition}

\begin{proof}{\bf of Proposition~\ref{prop:independent:localnatex}~}
For all $i\in\{1,2\}$, since $\lp_i$ is a coherent conditional lower prevision on $\C_i$, it follows from Proposition~\ref{prop:smallestSDGfromLP} that $\E(\lp_i)$ is a coherent set of desirable gambles on $\states_i$. Therefore, we infer from Proposition~\ref{prop:productindependent:SDG} that $\E(\lp_1)\otimes\E(\lp_2)$ is an independent product of $\E(\lp_1)$ and $\E(\lp_2)$.
For ease of notation, we now let $\lp\coloneqq\lp_1\otimes\lp_2$ and $\desir\coloneqq\E(\lp_1)\otimes\E(\lp_2)$. As we know from Equation~\eqref{eq:indnatex:LP}, $\lp$ is then equal to $\lp_\desir$. Furthermore, since $\desir$ is an independent product of $\E(\lp_1)$ and $\E(\lp_2)$, we know that $\desir$ is epistemically independent and that it has $\E(\lp_1)$ and $\E(\lp_2)$ as its marginals.

We are now ready to prove Equation~\eqref{eq:prop:independent:localnatex}. In order to do that, we fix any two indexes $i$ and $j$ such that $\{i,j\}=\{1,2\}$, any $f_i\in\gambleson{i}$ and $B_i\in\nonemptypoweron{i}$ and any $B_j\in\B_j$. We start by proving the first equality. Since $\desir$ is epistemically independent, we know that
\begin{align*}
[f_i-\mu]\ind{B_i}\in\desir
&\asa
[f_i-\mu]\ind{B_i}\in\marg_i(\desir)\\
&\asa
[f_i-\mu]\ind{B_i}\in\marg_i(\desir\vert B_j)
\asa
[f_i-\mu]\ind{B_i}\ind{B_j}\in\desir
\asa
[f_i-\mu]\ind{B_i\cap B_j}\in\desir
\end{align*}
for all $\mu\in\reals$,
and therefore, we find that
\begin{equation*}
\lp(f_i\vert B_i)
=\sup\big\{\mu\in\reals\colon[f_i-\mu]\ind{B_i}\in\desir\big\}
=\sup\big\{\mu\in\reals\colon[f_i-\mu]\ind{B_i\cap B_j}\in\desir\big\}
=
\lp(f_i\vert B_i\cap B_j).
\end{equation*}
Next, we prove the second equality of Equation~\eqref{eq:prop:independent:localnatex}. Since $\desir$ has $\E(\lp_1)$ and $\E(\lp_2)$ as its marginals, we know that
\vspace{-6pt}
\begin{equation*}
[f_i-\mu]\ind{B_i}\in\desir
\asa
[f_i-\mu]\ind{B_i}\in\marg_i(\desir)
\asa
[f_i-\mu]\ind{B_i}\in\E(\lp_i)
\vspace{3pt}
\end{equation*}
for all $\mu\in\reals$, and therefore, we find that
\begin{equation*}
\lp(f_i\vert B_i)
=
\sup\big\{\mu\in\reals\colon[f_i-\mu]\ind{B_i}\in\desir\big\}
=
\sup\big\{\mu\in\reals\colon[f_i-\mu]\ind{B_i}\in\E(\lp_i)\big\}
=\natexLP_i(f_i\vert B_i),
\end{equation*}
using Equation~\ref{eq:naturalextension} to establish the last equality.
\end{proof}
\vspace{-16pt}

\begin{proposition}\label{prop:productindependent:LP}
The restriction of $\lp_1\otimes\lp_2$ to $\C$ is an independent product of $\lp_1$ and $\lp_2$.
\end{proposition}
\begin{proof}{\bf of Proposition~\ref{prop:productindependent:LP}~}
Since we know from Proposition~\ref{prop:productcoherent:LP} that $\lp_1\otimes\lp_2$ is a coherent lower prevision on $\C(\states_1\times\states_2)$, it follows from Definition~\ref{def:cohlp} that its restriction to $\C$ is coherent as well. Due to Definition~\ref{def:independentproduct:LP}, it remains to show that this restriction of $\lp_1\otimes\lp_2$ to $\C$ is epistemically independent and that it coincides with $\lp_1$ and $\lp_2$ on their domain. Epistemic independence follows trivially from Definition~\ref{def:epistemicindependence:LP} and Proposition~\ref{prop:independent:localnatex}. Hence, it remains to prove that the restriction of $\lp_1\otimes\lp_2$ to $\C$ coincides with $\lp_1$ and $\lp_2$ on their domain, or equivalently, that
\begin{equation*}
(\lp_1\otimes\lp_2)(f_i\vert B_i)=\lp_i(f_i\vert B_i)
\text{~~for all $i\in\{1,2\}$ and $(f_i,B_i)\in\C_i$.}
\end{equation*}
So fix any $i\in\{1,2\}$ and $(f_i,B_i)\in\C_i$. We then find that indeed, as desired,
\begin{equation*}
(\lp_1\otimes\lp_2)(f_i\vert B_i)
=\natexLP_i(f_i\vert B_i)
=\lp_i(f_i\vert B_i),
\end{equation*}
where the first equality follows from Proposition~\ref{prop:independent:localnatex} and the second equality follows from Equation~\eqref{eq:naturalextension} and Proposition~\ref{prop:smallestSDGfromLP}.
\end{proof}

\begin{proof}{\bf of Theorem~\ref{theo:natext:LP}~}
Since we know from Proposition~\ref{prop:productindependent:LP} that the restriction of $\lp_1\otimes\lp_2$ to $\C$ is an independent product of $\lp_1$ and $\lp_2$, it suffices to prove that any other such independent product of $\lp_1$ and $\lp_2$ dominates $\lp_1\otimes\lp_2$ on $\C$.

So let $\lp$ be any independent product of $\lp_1$ and $\lp_2$. 
Definition~\ref{def:independentproduct:LP} then implies that $\lp$ is an epistemically independent coherent conditional lower prevision on $\C$ that coincides with $\lp_1$ and $\lp_2$ on their domain. Let $\A_{\lp}$ be the corresponding set of gambles, as defined by Equation~\eqref{eq:AfromLP}, and let \mbox{$\desir\coloneqq\E(\lp)=\E(\A_{\lp})$}. We then know from Proposition~\ref{prop:smallestSDGfromLP} that $\desir$ is a coherent set of desirable gambles on $\states_1\times\states_2$ and that $\lp_\desir$ coincides with $\lp$ on $\C$.
In the remainder of this proof, we will show that $\E(\lp_1)\otimes\E(\lp_2)\subseteq\desir$. Because of Equation~\eqref{eq:indnatex:LP}, this clearly implies that $\lp_\desir(f\vert B)\geq(\lp_1\otimes\lp_2)(f\vert B)$ for all $(f,B)\in\C$. Since $\lp_\desir$ coincides with $\lp$ on $\C$, this implies that $\lp$ dominates $\lp_1\otimes\lp_2$ on $\C$, thereby concluding the proof.

Let $\desir_1\coloneqq\E(\lp_1)$ and let $\A_{2\to1}$ be the corresponding set of gambles on $\states_1\times\states_2$, as defined by Equation~\eqref{eq:A21s}. We will now prove that $\A_{2\to1}\subseteq\desir$. So consider any $f_1\in\desir_1$ and any $B_2\in\B_2\cup\{\states_2\}$. We need to prove that $f_1(X_1)\ind{B_2}(X_2)\in\desir$. Since $f_1\in\desir_1=\E(\lp_1)=\posi(\A_{\lp_1}\cup\mathcal{G}_{>0}(\states_1))$, it follows from Equation~\eqref{eq:posi} that there are $n\in\nats$ and, for all $i\in\{1,\dots,n\}$, $\lambda_i\in\reals_{>0}$ and $g_i\in\A_{\lp_1}\cup\mathcal{G}_{>0}(\states_1)$ such that $f_1=\sum_{i=1}^n\lambda_ig_i$.

For any $i\in\{1,\dots,n\}$, we now let $h_i(X_1,X_2)\coloneqq g_i(X_1)\ind{B_2}(X_2)\in\mathcal{G}(\states_1\times\states_2)$. As we will show, this gamble $h_i$ belongs to $\desir$. We consider two cases: $g_i\in\mathcal{G}_{>0}(\states_1)$ and $g_i\notin\mathcal{G}_{>0}(\states_1)$. If $g_i\in\mathcal{G}_{>0}(\states_1)$, then $h_i\in\mathcal{G}_{>0}(\states_1\times\states_2)$, which, since $\desir$ is a coherent set of desirable gambles on $\states_1\times\states_2$, implies that $h_i\in\desir$. If $g_i\not\in\mathcal{G}_{>0}$, then since $g_i\in\A_{\lp_1}\cup\mathcal{G}_{>0}(\states_1)$, it follows that $g_i\in\A_{\lp_1}$, which implies that there are $(f'_1,B_1)\in\C_1$ and  $\mu<\lp_1(f'_1\vert B_1)$ such that $g_i=[f'_1-\mu]\ind{B_1}$. Furthermore, since $\lp$ coincides with $\lp_1$ on its domain, we also know that $\lp_1(f'_1\vert B_1)=\lp(f'_1\vert B_1)$. If $B_2=\states_2$, Equation~\eqref{eq:AfromLP} therefore implies that $h_i\in\A_{\lp}\subseteq\desir$ because $\ind{B_2}=1$. If $B_2\neq\states_2$, then $B_2\in\B_2$. Since $\lp$ is epistemically independent, this implies that $\lp(f'_1\vert B_1)=\lp(f'_1\vert B_1\cap B_2)$. Hence, here too, Equation~\eqref{eq:AfromLP} implies that $h_i\in\A_{\lp}\subseteq\desir$---because $\ind{B_1\cap B_2}=\ind{B_1}\ind{B_2}$.

In summary then, we have found that $h_i\in\desir$ for all $i\in\{1,\dots,n\}$. Since $f_1=\sum_{i=1}^n\lambda_ig_i$, this implies that
\vspace{-5pt}
\begin{equation*}
f_1(X_1)\ind{B_2}(X_2)
=
\left(\sum_{i=1}^n
\lambda_ig_i(X_1)
\right)
\ind{B_2}(X_2)
=
\sum_{i=1}^n
\lambda_i
g_i(X_1)
\ind{B_2}(X_2)
=
\sum_{i=1}^n
\lambda_i
h_i(X_1,X_2)
\in\desir,
\vspace{4pt}
\end{equation*}
where the inclusion holds because $\desir$ is coherent. Since this is true for every $f_1\in\desir_1$ and every $B_2\in\B_2\cup\{\states_2\}$, it follows that $\A_{2\to1}\subseteq\desir$. 
Using a completely analogous argument, it also follows that $\A_{1\to2}\subseteq\desir$, with $\A_{1\to2}$ defined by Equation~\eqref{eq:A12s} for $\desir_2\coloneqq\E(\lp_2)$.
Hence, we find that $\A_{1\to2}
\cup
\A_{2\to1}
\subseteq\desir$, and therefore,  that 
\begin{equation*}
\E(\lp_1)\otimes\E(\lp_2)
=
\desir_1\otimes\desir_2
=
\E
\left(
\A_{1\to2}
\cup
\A_{2\to1}
\right)
\subseteq\E(\desir)=\desir, 
\end{equation*}
where the second equality follows from Equation~\eqref{eq:indnatext:SDG}, the inclusion follows from Lemma~\ref{lemma:nestedpropsofposandE}, and the last equality follows from Lemma~\ref{lemma:natextDisD}.
\end{proof}
\vspace{-10pt}

\subsection{Proofs and Additional Material for Section~\ref{sec:choiceofevents}}
\vspace{5pt}

\begin{proof}{\bf of Proposition~\ref{prop:addfiniteunionstoB}~}
We only prove the result for $\desir_1\otimes\desir_2$. The result for $\lp_1\otimes\lp_2$ then follows trivially from Equation~\eqref{eq:indnatex:LP}.

Let $\desir_1\otimes\desir_2$ be the independent natural extension that corresponds to $\B_1$ and $\B_2$, as defined by Equations~\eqref{eq:indnatext:SDG}--\eqref{eq:A21s}, and let $\desir_1\otimes'\desir_2$ be the independent natural extension that corresponds to $\B'_1$ and $\B'_2$, defined by
\begin{equation*}
\desir_1\otimes'\desir_2
\coloneqq
\E
\left(
\A'_{1\to2}
\cup
\A'_{2\to1}
\right),
\vspace{-4pt}
\end{equation*}
with
\begin{equation*}
\A'_{1\to2}
\coloneqq
\left\{
f_2(X_2)\ind{B'_1}(X_1)
\colon
f_2\in\desir_2, 
B'_1\in\B'_1\cup\{\states_1\}
\right\}
\vspace{-3pt}
\end{equation*}
and
\begin{equation*}
\A'_{2\to1}
\coloneqq
\left\{
f_1(X_1)\ind{B'_2}(X_2)
\colon
f_1\in\desir_1, 
B'_2\in\B'_2\cup\{\states_2\}
\right\}.
\vspace{8pt}
\end{equation*}
Then as explained in the main text, in the paragraph that precedes Proposition~\ref{prop:addfiniteunionstoB}, we have that $\desir_1\otimes\desir_2\subseteq\desir_1\otimes'\desir_2$. It remains to prove that $\desir_1\otimes'\desir_2\subseteq\desir_1\otimes\desir_2$. 

Fix any $f_2\in\desir_2$ and $B'_1\in\B'_1\cup\{\states_1\}$. We will prove that $f_2(X_2)\ind{B'_1}(X_1)\in\desir_1\otimes\desir_2$. If $B'_1=\states_1$, this follows trivially from Equations~\eqref{eq:indnatext:SDG} and~\eqref{eq:A12s}. Otherwise, it follows from our assumptions that there is some $m\in\nats$ and, for all $k\in\{1,\dots,m\}$, some $B_{1,k}\in\B_1$ such that $B'_1$ is a finite disjoint union of the events $\{B_{1,k}\}_{1\leq k\leq m}$, which implies that $\ind{B'_1}=\sum_{k=1}^m\ind{B_{1,k}}$ and therefore also that $f_2(X_2)\ind{B'_1}(X_1)=\sum_{k=1}^mf_2(X_2)\ind{B_1}(X_1)$. Hence, Equations~\eqref{eq:indnatext:SDG} and~\eqref{eq:A12s} again imply that $f_2(X_2)\ind{B'_1}(X_1)\in\desir_1\otimes\desir_2$. Since this is true for every $f_2\in\desir_2$ and $B'_1\in\B'_1\cup\{\states_1\}$, it follows that $\A'_{1\to2}\subseteq\desir_1\otimes\desir_2$. Using a completely analogous argument, we also infer that $\A'_{2\to1}\subseteq\desir_1\otimes\desir_2$. The result now follows because $\A'_{1\to2}
\cup
\A'_{2\to1}\subseteq\desir_1\otimes\desir_2$ implies that
\begin{equation*}
\desir_1\otimes'\desir_2
=
\E
\left(
\A'_{1\to2}
\cup
\A'_{2\to1}
\right)
\subseteq
\E\left(
\desir_1\otimes\desir_2
\right)
=
\desir_1\otimes\desir_2,
\end{equation*}
using Lemma~\ref{lemma:nestedpropsofposandE} for the inclusion and Lemma~\ref{lemma:natextDisD} and Proposition~\ref{prop:productcoherent:SDG} for the last equality.
\end{proof}
\vspace{-10pt}

\subsection{Proofs and Additional Material for Section~\ref{sec:factadd}}
\vspace{5pt}

\begin{lemma}\label{lemma:fact-add-simple-geq}
For any $f\in\gambleson{1}$ and $h\in\gambleson{2}$ and any simple $\B_1$-measurable $g\in\mathcal{G}_{\geq0}(\states_1)$, we have that
\begin{equation*}
(\lp_1\otimes\lp_2)(f+gh)
\geq\natexLP_1\big(f+g\natexLP_2(h)\big).
\vspace{8pt}
\end{equation*}
\end{lemma}
\begin{proof}{\bf of Lemma~\ref{lemma:fact-add-simple-geq}~}
Since $g\in\mathcal{G}_{\geq0}(\states_1)$ is a simple $\B$-measurable gamble, we know from Definition~\ref{def:measurable:simple} that there are $c_0\in\reals_{\geq0}$, $n\in\natswith$ and, for all $i\in\{1,\dots,n\}$, $c_i\in\reals_{\geq0}$ and $B_i\in\B_1$, such that $g=c_0+\sum_{i=1}^nc_i\ind{B_i}$. Furthermore, since we know from Proposition~\ref{prop:productcoherent:LP} that $\lp_1\otimes\lp_2$ is coherent, it follows from Proposition~\ref{prop:propertiesofLP} that $\lp_1\otimes\lp_2$ satisfies~\ref{def:lowerprev:homo},~\ref{def:lowerprev:superadditive} and~\ref{def:lowerprev:GBR}. Finally, since $\natexLP_2$ is coherent, we know from Proposition~\ref{prop:propertiesofLP} that it satisfies~\ref{def:lowerprev:constantadditivity}. Therefore, we find that
\begin{align*}
(\lp_1\otimes\lp_2)&(f+gh)
=
(\lp_1\otimes\lp_2)\Big(
f+g\natexLP_2(h)+\big(c_0+\sum_{i=1}^nc_i\ind{B_i}\big)[h-\natexLP_2(h)]
\Big)\\
&\geq
(\lp_1\otimes\lp_2)\big(
f+g\natexLP_2(h)\big)
+
c_0
(\lp_1\otimes\lp_2)\big(h-\natexLP_2(h)
\big)
+
\sum_{i=1}^nc_i
(\lp_1\otimes\lp_2)\big(\ind{B_i}[h-\natexLP_2(h)]
\big)\\
&=
\natexLP_1\big(
f+g\natexLP_2(h)\big)
+
c_0
\natexLP_2\big(h-\natexLP_2(h)
\big)
+
\sum_{i=1}^nc_i
(\lp_1\otimes\lp_2)\big(\ind{B_i}[h-(\lp_1\otimes\lp_2)(h\vert B_i)]
\big)\\
&=
\natexLP_1\big(
f+g\natexLP_2(h)\big)
+
c_0
\big(\natexLP_2(h)-\natexLP_2(h)
\big)
=\natexLP_1\big(
f+g\natexLP_2(h)\big),\\[-8pt]
\end{align*}
where the first equality follows because $g=c_0+\sum_{i=1}^nc_i\ind{B_i}$, where the first inequality follows because $\lp_1\otimes\lp_2$ satisfies~\ref{def:lowerprev:superadditive} and~\ref{def:lowerprev:homo}, where the second equality follows from Proposition~\ref{prop:independent:localnatex}, and where the third equality follows because $\natexLP_2$ satisfies~\ref{def:lowerprev:constantadditivity} and $\lp_1\otimes\lp_2$ satisfies~\ref{def:lowerprev:GBR}.
\end{proof}
\vspace{-16pt}

\begin{lemma}\label{lemma:fact-add-simple-leq}
For any $f\in\gambleson{1}$ and $h\in\gambleson{2}$ and any simple $\B_1$-measurable $g\in\mathcal{G}_{\geq0}(\states_1)$, we have that
\begin{equation*}
(\lp_1\otimes\lp_2)(f+gh)
\leq\natexLP_1\big(f+g\natexLP_2(h)\big).\vspace{8pt}
\end{equation*}
\end{lemma}
\begin{proof}{\bf of Lemma~\ref{lemma:fact-add-simple-leq}~}
Since $\natexLP_2$ is a coherent conditional lower prevision on $\C(\states_2)$, we know from Proposition~\ref{prop:lowerenvelope:onlyif:full} that there is a conditional linear prevision $P_2$ on $\C(\states_2)$ such that $P_2(h)=\natexLP_2(h)$ and $P_2\geq\natexLP_2$. Similarly, since $\natexLP_1$ is a coherent conditional lower prevision on $\C(\states_1)$, we know from Proposition~\ref{prop:lowerenvelope:onlyif:full} that there is a conditional linear prevision $P_1$ on $\C(\states_1)$ such that $P_1\big(f+g\natexLP_2(h)\big)=\natexLP_1\big(f+g\natexLP_2(h)\big)$ and $P_1\geq\natexLP_1$.

Consider now any $i\in\{1,2\}$. We then know from Proposition~\ref{prop:naturalextension} that $\natexLP_i$ coincides with $\lp_i$ on $\C_i$. Therefore, and because $P_i\geq\natexLP_i$, we also know that $P_i$ dominates $\lp_i$ on $\C_i$. Due to Equation~\eqref{eq:AfromLP}, this implies that $\A_{\lp_i}\subseteq\A_{P_i}$ and therefore, using Lemma~\ref{lemma:nestedpropsofposandE}, also that $\E(\lp_i)\subseteq\E(P_i)$. 
Since this is true for every $i\in\{1,2\}$, it follows from Equation~\eqref{eq:indnatext:SDG} and Lemma~\ref{lemma:nestedpropsofposandE} that $\E(\lp_1)\otimes\E(\lp_2)\subseteq\E(P_1)\otimes\E(P_2)$, and therefore, because of Equation~\eqref{eq:indnatex:LP}, that $\lp_1\otimes\lp_2\leq P_1\otimes P_2$. 

The result can now be proved as follows. First, since $\lp_1\otimes\lp_2\leq P_1\otimes P_2$, we find that 
\begin{equation}\label{eq:lemma:fact-add-simple-leq:1}
(\lp_1\otimes\lp_2)(f+gh)
\leq
(P_1\otimes P_2)(f+gh).
\end{equation}
Secondly, since we know from Proposition~\ref{prop:productcoherent:LP} that $(P_1\otimes P_2)$ is coherent, it follows from Proposition~\ref{prop:propertiesofLP} that $(P_1\otimes P_2)$ satisfies~\ref{def:lowerprev:lowerbelowupper}, which implies that
\begin{equation}\label{eq:lemma:fact-add-simple-leq:2}
(P_1\otimes P_2)(f+gh)
\leq
-(P_1\otimes P_2)(-f-gh)
\leq
-P_1\big(-f+gP_2(-h)\big),
\end{equation}
using Lemma~\ref{lemma:fact-add-simple-geq} for the second inequality.
Finally, we also know that
\begin{equation}\label{eq:lemma:fact-add-simple-leq:3}
-P_1\big(-f+gP_2(-h)\big)
=
P_1\big(f+gP_2(h)\big)
=
\natexLP_1\big(f+g\natexLP_2(h)\big),
\end{equation}
where the first equality follows from Definitions~\ref{def:prev} and~\ref{def:linearprev} because $P_1$ and $P_2$ are conditional linear previsions, and where the second equality follows because $P_2(h)=\natexLP_2(h)$ and $P_1\big(f+g\natexLP_2(h)\big)=\natexLP_1\big(f+g\natexLP_2(h)\big)$.
By combining Equations~\eqref{eq:lemma:fact-add-simple-leq:1}--\eqref{eq:lemma:fact-add-simple-leq:3}, the result is now immediate.
\end{proof}
\vspace{-16pt}

\begin{proposition}\label{prop:fact-add-simple}
For any $f\in\gambleson{1}$ and $h\in\gambleson{2}$ and any simple $\B_1$-measurable $g\in\mathcal{G}_{\geq0}(\states_1)$, we have that
\begin{equation*}
(\lp_1\otimes\lp_2)(f+gh)
=\natexLP_1\big(f+g\natexLP_2(h)\big).
\vspace{5pt}
\end{equation*}
\end{proposition}
\begin{proof}{\bf of Proposition~\ref{prop:fact-add-simple}~}
Immediate consequence of Lemmas~\ref{lemma:fact-add-simple-geq} and~\ref{lemma:fact-add-simple-leq}.
\end{proof}
\vspace{-7pt}

\begin{proof}{\bf of Theorem~\ref{theo:fact-add-measurable}~}
Since $g\in\mathcal{G}_{\geq0}(\states_1)$ is $\B_1$-measurable, we know from Definition~\ref{def:measurable:uniform} that there is a sequence $\{g_n\}_{n\in\nats}$ of simple $\B_1$-measurable gambles in $\mathcal{G}_{\geq0}(\states_1)$ such that $g_n$ converges uniformly to $g$.
This also implies that $f+g_n\natexLP_2(h)$ converges uniformly to $f+g\natexLP_2(h)$ and, since $h$ is a gamble and therefore by definition bounded, that $f+g_nh$ converges uniformly to $f+gh$. The result now follows from the following series of equalities:
\begin{equation*}
(\lp_1\otimes\lp_2)(f+gh)
=
\lim_{n\to+\infty}
(\lp_1\otimes\lp_2)(f+g_nh)
=
\lim_{n\to+\infty}
\natexLP_1\big(f+g_n\natexLP_2(h)\big)
=
\natexLP_1\big(f+g\natexLP_2(h)\big).
\end{equation*}
The first of these equalities holds because it follows from Propositions~\ref{prop:productcoherent:LP} and~\ref{prop:propertiesofLP} that $\lp_1\otimes\lp_2$ satisfies~\ref{def:lowerprev:uniformcontinuity}. The second equality follows from Proposition~\ref{prop:fact-add-simple}. The third equality holds because the coherence of $\natexLP_1$ allows us to infer from Proposition~\ref{prop:propertiesofLP} that $\natexLP_1$ satisfies~\ref{def:lowerprev:uniformcontinuity}.
\end{proof}

\begin{proof}{\bf of Corollary~\ref{corol:fact-measurable}~}
Let $f\coloneqq0\in\mathcal{G}(\states_i)$. We then know from Theorem~\ref{theo:fact-add-measurable} that 
\begin{equation*}
(\lp_1\otimes\lp_2)(gh)
=(\lp_1\otimes\lp_2)(f+gh)
=\natexLP_i\big(f+g\natexLP_j(h)\big)
=\natexLP_i\big(g\natexLP_j(h)\big).
\end{equation*}
The result can now be inferred from the non-negative homogeneity---\ref{def:lowerprev:homo}---of $\natexLP_i$ that is implied by its coherence. If $\natexLP_j(h)\geq0$, we simply apply the non-negative homogeneity for $\lambda\coloneqq\natexLP_j(h)$. If $\natexLP_j(h)\leq0$, we apply it for $\lambda\coloneqq-\natexLP_j(h)$ and combine this with the fact that $\natexUP_i(g)\coloneqq-\natexLP_i(-g)$.
\end{proof}

\begin{proof}{\bf of Corollary~\ref{corol:add}~}
Let $g\coloneqq1$. Then $g$ belongs to $\mathcal{G}_{\geq0}(\states_1)$ and is $\B_1$-measurable. Therefore, we know from Theorem~\ref{theo:fact-add-measurable} that
\begin{equation*}
(\lp_1\otimes\lp_2)(f+h)
=(\lp_1\otimes\lp_2)(f+gh)
=\natexLP_1\big(f+g\natexLP_2(h)\big)
=\natexLP_1\big(f+\natexLP_2(h)\big).
\end{equation*}
The result now follows from the constant additivity---\ref{def:lowerprev:constantadditivity}---of $\natexLP_1$ that is implied by its coherence.
\end{proof}
\vspace{-10pt}

\subsection{Proofs and Additional Material for Section~\ref{sec:specialcases}}
\vspace{5pt}

\begin{proof}{\bf of Proposition~\ref{prop:dominatedbysingletonzero}~}
Fix any $\epsilon>0$. It then follows from Equations~\eqref{eq:indnatex:LP} and~\eqref{eq:LPfromD} that there is some $\mu\in\reals$ such that $\mu>(\pr_1\underline{\otimes}\pr_2)(f)-\epsilon$ and $f-\mu\in\E(\pr_1)\otimes\E(\pr_2)$. 
Consider any such $\mu$. Since $f-\mu\in\E(\pr_1)\otimes\E(\pr_2)$, it follows from Equations~\eqref{eq:indnatext:SDG},~\eqref{eq:A12s},~\eqref{eq:A21s},~\eqref{eq:natextop} and~\eqref{eq:posi} that
\begin{equation*}
f-\mu=g+\sum_{i=1}^n\lambda_if_{2,i}(X_2)\ind{B_{1,i}}(X_1)+\sum_{j=1}^m\lambda_jf_{1,j}(X_1)\ind{B_{2,j}}(X_2)
\end{equation*}
with $g\in\mathcal{G}_{\geq0}$ and $n,m\in\natswith$ and, for all $i\in\{1,\dots,n\}$, $\lambda_i\in\reals_{>0}$, $f_{2,i}\in\E(\pr_2)$ and $B_{1,i}\in\B_1\cup\{\states_1\}$ and, for all $j\in\{1,\dots,m\}$, $\lambda_j\in\reals_{>0}$, $f_{1,j}\in\E(\pr_1)$ and $B_{2,j}\in\B_2\cup\{\states_2\}$. Since $\pr_{12}$ is a linear prevision---and hence satisfies~\ref{def:prev:bounded}--\ref{def:prev:additive}---this implies that
\begin{align*}
\pr_{12}(f)&=\pr_{12}(\mu)+\pr_{12}(g)+\sum_{i=1}^n\lambda_i\pr_{12}(f_{2,i}(X_2)\ind{B_{1,i}}(X_1))+\sum_{j=1}^m\lambda_j\pr_{12}(f_{1,j}(X_1)\ind{B_{2,j}}(X_2))\\
&\geq\mu+\sum_{i=1}^n\lambda_i\pr_{12}(f_{2,i}(X_2)\ind{B_{1,i}}(X_1))+\sum_{j=1}^m\lambda_j\pr_{12}(f_{1,j}(X_1)\ind{B_{2,j}}(X_2))
\end{align*}
Furthermore, for any $i\in\{1,\dots,n\}$, there are two cases. If $B_{1,i}=\states_1$, then
\begin{equation*}
\pr_{12}(f_{2,i}(X_2)\ind{B_{1,i}}(X_1))
=\pr_{12}(f_{2,i}(X_2))=P_2(f_{2,i})\geq0,
\end{equation*}
where the second equality follows from our assumptions on $P_{12}$ and the third equality follows from Lemma~\ref{lemma:nonnegativeELP}. If $B_{1,i}\in\B_1$, then 
\begin{equation*}
\pr_{12}(f_{2,i}(X_2)\ind{B_{1,i}}(X_1))
\geq\pr_{12}(\inf_{x_2\in\states_2}f_{2,i}\ind{B_{1,i}}(X_1))
=\inf_{x_2\in\states_2}f_{2,i}\pr_{12}(\ind{B_{1,i}}(X_1))
=\inf_{x_2\in\states_2}f_{2,i}\pr_{1}(\ind{B_{1,i}})
=0,
\end{equation*}
where the inequality follows from~\ref{def:lowerprev:monotonicity}, the first equality follows from~\ref{def:prev:homo} and the last two equalities follow from our assumptions on $P_{12}$ and $P_1$ and from the fact that---since we consider atom-independence---$B_{1,i}\in\B_1$ implies that there is some $x_1\in\states_1$ such that $B_{1,i}=\{x_1\}$ and hence also $\pr_{1}(\ind{B_{1,i}})=\pr_1(\ind{x_1})=0$. In both cases, we find that $\pr_{12}(f_{2,i}(X_2)\ind{B_{1,i}}(X_1))\geq0$. Similarly, for any $j\in\{1,\dots,m\}$, we find that $\pr_{12}(f_{1,j}(X_1)\ind{B_{2,j}}(X_2))\geq0$. Hence, it follows that $\pr_{12}(f)\geq\mu$. Since $\mu>(\pr_1\underline{\otimes}\pr_2)(f)-\epsilon$, this implies that $\pr_{12}(f)>(\pr_1\underline{\otimes}\pr_2)(f)-\epsilon$, and, since $\epsilon>0$ is arbitrary, this in turn implies that $\pr_{12}(f)\geq(\pr_1\underline{\otimes}\pr_2)(f)$. This already establishes the first inequality of the statement. For the second equality of the statement, it suffices to apply the first equality to $-f$, yielding $(\pr_1\underline{\otimes}\pr_2)(-f)\leq\pr_{12}(-f)$. Indeed, since $P_{12}$ is linear, this trivially implies that
\begin{equation*}
(P_1\overline{\otimes}P_2)(f)=-(P_1\underline{\otimes}P_2)(-f)\geq -P_{12}(-f)=P_{12}(f),
\end{equation*}
thereby establishing the second inequality of the statement.
\end{proof}

\vspace{-10pt}

\begin{proof}{\bf of Proposition~\ref{prop:dominatedbynested}~}
Fix any $\epsilon>0$. It then follows from Equations~\eqref{eq:indnatex:LP} and~\eqref{eq:LPfromD} that there is some $\mu\in\reals$ such that $\mu>(\pr_1\underline{\otimes}\pr_2)(f)-\epsilon$ and $f-\mu\in\E(\pr_1)\otimes\E(\pr_2)$. 
Consider any such $\mu$. Since $f-\mu\in\E(\pr_1)\otimes\E(\pr_2)$, it follows from Equations~\eqref{eq:indnatext:SDG},~\eqref{eq:A12s},~\eqref{eq:A21s},~\eqref{eq:natextop} and~\eqref{eq:posi} that
\begin{equation*}
f-\mu=g+\sum_{i=1}^n\lambda_if_{2,i}(X_2)\ind{B_{1,i}}(X_1)+\sum_{j=1}^m\lambda_jf_{1,j}(X_1)\ind{B_{2,j}}(X_2)
\end{equation*}
with $g\in\mathcal{G}_{\geq0}$ and $n,m\in\natswith$ and, for all $i\in\{1,\dots,n\}$, $\lambda_i\in\reals_{>0}$, $f_{2,i}\in\E(\pr_2)$ and $B_{1,i}\in\B_1\cup\{\states_1\}$ and, for all $j\in\{1,\dots,m\}$, $\lambda_j\in\reals_{>0}$, $f_{1,j}\in\E(\pr_1)$ and $B_{2,j}\in\B_2\cup\{\states_2\}$. Since $\pr_1$ and $\pr_2$ are linear previsions---and hence satisfy~\ref{def:prev:bounded}--\ref{def:prev:additive}---this implies that
\begin{align*}
\pr_1(\pr_2(f))&=\pr_1(\pr_2(\mu)+\pr_1(\pr_2(g))+\sum_{i=1}^n\lambda_i\pr_1(\pr_2(f_{2,i}(X_2)\ind{B_{1,i}}(X_1)))+\sum_{j=1}^m\lambda_j\pr_1(\pr_2(f_{1,j}(X_1)\ind{B_{2,j}}(X_2)))\\
&\geq\mu+\sum_{i=1}^n\lambda_i\pr_1(\pr_2(f_{2,i}(X_2))\ind{B_{1,i}}(X_1))+\sum_{j=1}^m\lambda_j\pr_1(f_{1,j}(X_1)\pr_2(\ind{B_{2,j}}(X_2)))\\
&=\mu+\sum_{i=1}^n\lambda_i\pr_1(\ind{B_{1,i}}(X_1))\pr_2(f_{2,i}(X_2))+\sum_{j=1}^m\lambda_j\pr_1(f_{1,j}(X_1))\pr_2(\ind{B_{2,j}}(X_2)).
\end{align*}
Furthermore, for any $i\in\{1,\dots,n\}$, it follows from Lemma~\ref{lemma:nonnegativeELP} that $\pr_2(f_{2,i}(X_2))\geq0$ and from~\ref{def:prev:bounded} that $\pr_1(\ind{B_{1,i}}(X_1))\geq0$, which implies that $\lambda_i\pr_1(\ind{B_{1,i}}(X_1))\pr_2(f_{2,i}(X_2))\geq0$. Similarly, for any $j\in\{1,\dots,m\}$, we find that $\lambda_j\pr_1(f_{1,j}(X_1))\pr_2(\ind{B_{2,j}}(X_2))\geq0$. Hence, it follows that $\pr_1(\pr_2(f))\geq\mu$. Since $\mu>(\pr_1\underline{\otimes}\pr_2)(f)-\epsilon$, this implies that $\pr_1(\pr_2(f))>(\pr_1\underline{\otimes}\pr_2)(f)-\epsilon$, and, since $\epsilon>0$ is arbitrary, this in turn implies that $\pr_1(\pr_2(f))\geq(\pr_1\underline{\otimes}\pr_2)(f)$. This already establishes the first inequality of the statement. For the second equality of the statement, it suffices to apply the first equality to $-f$, yielding $(\pr_1\underline{\otimes}\pr_2)(-f)\leq\pr_1(\pr_2(-f))$. Indeed, since $P_1$ and $P_2$ are linear, this trivially implies that
\begin{equation*}
(P_1\overline{\otimes}P_2)(f)=-(P_1\underline{\otimes}P_2)(-f)\geq -P_1(P_2(-f))=P_1(P_2(f)),
\end{equation*}
thereby establishing the second inequality of the statement.
\end{proof}

\end{document}